\documentclass[11pt,a4paper,usenames,dvipsnames]{article}

\usepackage{url,amsmath,amssymb,latexsym,pstricks,mathrsfs,comment,amsthm,graphicx,tikz,tikz-cd,enumerate,accents,pgffor,cite,wrapfig,multicol,float,cases,calc,geometry}
\usepackage[colorlinks]{hyperref}

\usepackage{xcolor}

\usepackage{enumitem}

\usepackage[T1]{fontenc}
\usepackage{ wasysym }
\usepackage{stmaryrd}

\usepackage{tikz}
\usetikzlibrary{decorations.markings}
\usetikzlibrary{arrows,matrix}
\usepgflibrary{arrows}
  
\tikzset{->-/.style={decoration={
  markings,
  mark=at position #1 with {\arrow{>}}},postaction={decorate}}}

\usetikzlibrary{calc}

\renewcommand{\arraystretch}{1.2}

\allowdisplaybreaks

\newcommand{\nc}{\newcommand}
\nc{\rnc}{\renewcommand}
\nc{\set}[2]{\{#1:#2\}}
\nc{\bigset}[2]{\big\{#1:#2\big\}}
\nc\sub\subseteq
\nc\mt\mapsto
\nc\al\alpha
\nc\be\beta
\nc\ga\gamma
\nc\de\delta
\nc\si\sigma
\nc\ep\epsilon
\nc\lam\lambda
\nc\ze\zeta
\nc\vk\varkappa
\nc\vt{\widetilde{\nu}}
\nc\vs\varsigma
\nc\ka\kappa
\rnc\th\theta
\nc\om\omega
\nc\ups\upsilon
\nc\Ga\Gamma
\nc\De\Delta
\nc\Si\Sigma
\nc\Om\Omega
\rnc\O{\mathbb O}
\nc\bit{\begin{itemize}}
\nc\eit{\end{itemize}}
\nc\ben{\begin{enumerate}[label=\textup{(\roman*)},leftmargin=7mm]}
\nc\een{\end{enumerate}}
\nc{\leqR}{\leq_{\R}}
\nc{\leqL}{\leq_{\L}}
\nc{\leqJ}{\leq_{\J}}
\nc{\leqK}{\leq_{\K}}
\nc{\geqR}{\geq_{\R}}
\nc{\geqL}{\geq_{\L}}
\nc{\geqJ}{\geq_{\J}}
\newcommand{\leqdom}{\leq^{\textup{D}}}

\newcommand{\leqC}{\leq^{\textup{C}}}
\newcommand{\wedgeC}{\wedge^{\textup{C}}}
\newcommand{\veeC}{\vee^{\textup{C}}}

\nc\bp{{\bf p}}
\nc\bq{{\bf q}}
\rnc\iff{\ \Leftrightarrow\ }
\rnc\implies{\ \Rightarrow\ }

\nc\pf{\begin{proof}}
\nc\epf{\end{proof}}
\nc\epfres{\hfill\qed}
\nc\epfreseq{\tag*{\qed}}

\let\oldproofname=\proofname
\renewcommand{\proofname}{\rm\bf{\oldproofname}}

\nc\AND{\qquad\text{and}\qquad}
\nc\WHERE{\qquad\text{where}\qquad}
\nc\ANd{\quad\text{and}\quad}
\nc\anD{\ \ \ \text{and}\ \ \ }
\nc\ANDSIM{\qquad\text{and similarly}\qquad}
\nc{\COMMA}{,\qquad}
\nc{\COMMa}{,\quad}
\nc\ul\underline
\nc\ol\overline
\nc\permdec[1]{#1^{\natural}}
\nc\ext[1]{#1^\textup{E}}
\nc\wh\widehat
\nc\normal\unlhd
\rnc\emptyset\varnothing
\nc{\pfitem}[1]{\medskip\noindent #1}
\nc{\pfcase}[1]{\medskip\noindent {\it Case #1.}}
\nc{\pfsubcase}[1]{\medskip\noindent {\it Subcase #1.}}
\nc\im{\operatorname{im}}
\nc\LSUB{\operatorname{LSUB}}
\DeclareMathOperator{\Lay}{Lay}
\nc\pre\preceq
\nc\Y{\mathcal Y}
\nc\B{\mathcal B}
\nc\Z{\mathcal Z}
\nc\ZZ{\mathbb Z}
\nc\F{\mathfrak F}
\nc\T{\mathcal T}
\nc\TT{\mathscr T}
\nc\PP{\mathscr P\mathcal P}
\nc\C{\mathscr C}
\nc\I{\mathcal I}
\nc\Eq{\mathbb{E}}
\nc\Part{\mathbb{P}}
\nc\cg[2]{(#1,#2)^\sharp}
\nc\Rev{\operatorname{Rev}}
\nc\cR{\mathcal R}
\nc\Ptop{P^\top}
\nc\Qtop{Q^\top}
\nc\la\langle
\nc\ra\rangle
\nc\tb[1]{\operatorname{Seq}(#1)}
\nc\RevX{{\Rev}\big([0,|X|],[0,|X|^+]\big)}

\makeatletter
\DeclareRobustCommand\widecheck[1]{{\mathpalette\@widecheck{#1}}}
\def\@widecheck#1#2{%
    \setbox\z@\hbox{\m@th$#1#2$}%
    \setbox\tw@\hbox{\m@th$#1%
       \widehat{%
          \vrule\@width\z@\@height\ht\z@
          \vrule\@height\z@\@width\wd\z@}$}%
    \dp\tw@-\ht\z@
    \@tempdima\ht\z@ \advance\@tempdima2\ht\tw@ \divide\@tempdima\thr@@
    \setbox\tw@\hbox{%
       \raise\@tempdima\hbox{\scalebox{1}[-1]{\lower\@tempdima\box
\tw@}}}%
    {\ooalign{\box\tw@ \cr \box\z@}}}
\makeatother

\nc\ch\widecheck

\newcommand{\uv}[1]{\fill (#1,2)circle(.17);}
\newcommand{\lv}[1]{\fill (#1,0)circle(.17);}
\newcommand{\uvs}[1]{{\foreach \x in {#1} { \uv{\x}}}}
\newcommand{\lvs}[1]{{\foreach \x in {#1} { \lv{\x}}}}
\newcommand{\darcx}[3]{\draw(#1,0)arc(180:90:#3) (#1+#3,#3)--(#2-#3,#3) (#2-#3,#3) arc(90:0:#3);}
\newcommand{\darc}[2]{\darcx{#1}{#2}{.4}}
\newcommand{\uarcx}[3]{\draw(#1,2)arc(180:270:#3) (#1+#3,2-#3)--(#2-#3,2-#3) (#2-#3,2-#3) arc(270:360:#3);}
\newcommand{\uarc}[2]{\uarcx{#1}{#2}{.4}}
\newcommand{\stline}[2]{\draw(#1,2)--(#2,0);}

\nc{\buv}[1]{\fill (#1,2)circle(.18);}
\nc{\buvs}[1]{{
\foreach \x in {#1}
{ \buv{\x}}
}}
\nc{\blv}[1]{\fill (#1,0)circle(.18);}
\nc{\blvs}[1]{{
\foreach \x in {#1}
{ \blv{\x}}
}}

\nc{\uarcs}[1]{
{\foreach \x/\y in {#1}
{ \uarc{\x}{\y} }
}
}

\nc{\darcs}[1]{
{\foreach \x/\y in {#1}
{ \darc{\x}{\y} }
}
}
\nc{\darcxhalf}[3]{\draw(#1,0)arc(180:90:#3) (#1+#3,#3)--(#2,#3) ;}
\nc{\darchalf}[2]{\darcxhalf{#1}{#2}{.4}}
\nc{\uarcxhalf}[3]{\draw(#1,2)arc(180:270:#3) (#1+#3,1.5-#3)--(#2,1.5-#3) ;}
\nc{\uarchalf}[2]{\uarcxhalf{#1}{#2}{.4}}
\nc{\colv}[3]{\fill[#3] (#1,#2)circle(.17);}

\nc{\uvert}[1]{\fill (#1,2)circle(.2);}
\rnc{\lvert}[1]{\fill (#1,0)circle(.2);}

\nc{\custpartn}[3]{{\lower1.4 ex\hbox{
\begin{tikzpicture}[scale=.3]
\foreach \x in {#1}
{ \uvert{\x}  }
\foreach \x in {#2}
{ \lvert{\x}  }
#3 \end{tikzpicture}
}}}

\renewcommand{\P}{\mathcal P} 
\newcommand{\PB}{\mathcal{PB}} 
\newcommand{\JJ}{\mathcal{J}} 
\renewcommand{\S}{\mathcal{S}}
\newcommand{\M}{\mathcal{M}_X}
\newcommand{\Mn}{\mathcal{M}_n}
\newcommand{\MY}{\mathcal{M}_Y}

\newcommand{\MW}{\mathcal{M}_W}
\nc\MYZ{\mathcal M_{Y\cup Z}}

\renewcommand{\H}{\mathrel{\mathscr H}}
\renewcommand{\L}{\mathrel{\mathscr L}}
\newcommand{\R}{\mathrel{\mathscr R}}
\newcommand{\D}{\mathrel{\mathscr D}}
\newcommand{\J}{\mathrel{\mathscr J}}
\newcommand{\K}{\mathrel{\mathscr K}}

\newcommand{\N}{\mathbb{N}}

\newcommand{\NN}{\mathcal{N}}
\nc\HH{\mathcal H}

\newcommand{\Cong}{\operatorname{Cong}}
\newcommand{\coker}{\operatorname{coker}}
\newcommand{\dom}{\operatorname{dom}} 
\newcommand{\codom}{\operatorname{codom}}
\newcommand{\rank}{\operatorname{rank}}
\newcommand{\crank}{\operatorname{crank}}
\newcommand{\drank}{\operatorname{drank}}
\newcommand{\id}{\operatorname{id}}
\DeclareMathOperator{\cof}{cof}

\renewcommand{\c}{@{}c@{}}
\newcommand{\cend}{@{}c@{\hspace{1.5truemm}}}
\newcommand{\cstart}{@{\hspace{1.5truemm}}c@{}}
\newcommand{\cstartend}{@{\hspace{1.5truemm}}c@{\hspace{1.5truemm}}}

\newcommand{\partn}[4]{
\Big(   \hspace{-1.5 truemm}
{ \scriptsize \renewcommand*{\arraystretch}{1}
\begin{array} {\cstart|\cend}
 #1 \:&\: #2 \\ \cline{2-2}
 #3 \:&\: #4 
\rule[0mm]{0mm}{2.7truemm}
\end{array} 
}
\hspace{-1.5 truemm} \Big)
}

\newcommand{\partABCD}{\partn{A_i}{C_j}{B_i}{D_k}}

\newcommand{\partXI}[4]{
\Big(  
{\scriptsize \renewcommand*{\arraystretch}{1} \begin{array} {\c|\cend}
#1 \:&\: #2  \\
#3 \:&\: #4
\rule[0mm]{0mm}{2.7mm}
\end{array}  }
\hspace{-1.5 truemm} \Big) 
}

\newcommand{\partXX}[2]{
\Big(  \hspace{-1.5 truemm}
{\scriptsize \renewcommand*{\arraystretch}{1} \begin{array} {\cstartend}
#1   \\ \cline{1-1}
#2
\rule[0mm]{0mm}{2.7mm}
\end{array}  }
\hspace{-1.5 truemm} \Big) 
}

\newcommand{\partXXI}[8]{
\Big( 
{ \scriptsize \renewcommand*{\arraystretch}{1}
\begin{array} {\c|\c|\c|\c|\cend}
 #1 \:&\: \cdots \:&\: #2 \:&\: #3 \:&\: #4 \\ \cline{5-5}
 #5 \:&\: \cdots \:&\: #6 \:&\: #7 \:&\: #8 
\rule[0mm]{0mm}{2.7mm}
\end{array} 
}
\hspace{-1.5 truemm} \Big)
}

\newcommand{\partXXII}[6]{
\Big( 
{ \scriptsize \renewcommand*{\arraystretch}{1}
\begin{array} {\c|\c|\c|\cend}
 #1 \:&\: \cdots \:&\: #2 \:&\: #3 \\ \cline{4-4}
 #4 \:&\: \cdots \:&\: #5 \:&\: #6
\rule[0mm]{0mm}{2.7mm}
\end{array} 
}
\hspace{-1.5 truemm} \Big)
}

\newcommand{\partXXIV}[6]{
\Big( 
{ \scriptsize \renewcommand*{\arraystretch}{1}
\begin{array} {\c|\c|\cend}
 #1 \:&\: #2 \:&\: #3 \\ \cline{3-3}
 #4 \:&\: #5 \:&\: #6
\rule[0mm]{0mm}{2.7mm}
\end{array} 
}
\hspace{-1.5 truemm} \Big)
}

\newcommand{\partXXV}[4]{
\Big(  \hspace{-1.5 truemm}
{\scriptsize \renewcommand*{\arraystretch}{1} \begin{array} {\cstart|\cend}
#1 \:&\: #2  \\ \cline{1-2}
#3 \:&\: #4 
\rule[0mm]{0mm}{2.7mm}
\end{array}  }
\hspace{-1.5 truemm} \Big) 
}

\newcommand{\partpermII}[4]{
\Big(
{ \scriptsize \renewcommand*{\arraystretch}{1}
\begin{array} {\c|\c}
 #1 \:&\: #2 \\ 
 #3 \:&\: #4
\rule[0mm]{0mm}{2.7mm}
\end{array} 
}
\Big)
}

\newcommand{\partpermIII}[6]{
\Big(
{ \scriptsize \renewcommand*{\arraystretch}{1}
\begin{array} {\c|\c|\c}
 #1 \:&\: #2 \:&\: #3 \\ 
 #4 \:&\: #5 \:&\: #6 
\rule[0mm]{0mm}{2.7mm}
\end{array} 
}
\Big)
}

\nc\congsquare[4]
{\fill[black!15] (#1*6,#2*6,#3*8)--({#1*6+(#4)*6},#2*6,#3*8)--({#1*6+(#4)*6},{#2*6+(#4)*6},#3*8)--(#1*6,{#2*6+(#4)*6},#3*8)--(#1*6,#2*6,#3*8);
\foreach \x in {0,1,...,#4} {\draw (#1*6,#2*6+\x*6,#3*8)--({#1*6+(#4)*6},#2*6+\x*6,#3*8); \draw (#1*6+\x*6,#2*6,#3*8)--(#1*6+\x*6,{#2*6+(#4)*6},#3*8);}}

\nc\congsquareconnections[7]
{\foreach \x in {0,...,#7} \foreach \y in {0,...,#7} {\draw(#1*6+6*\x,#2*6+6*\y,#3*8)--(#4*6+6*\x,#5*6+6*\y,#6*8);}}

\newtheorem{thm}{Theorem}[section]
\newtheorem{lemma}[thm]{Lemma}
\newtheorem{cor}[thm]{Corollary}
\newtheorem{prop}[thm]{Proposition}

\theoremstyle{definition}

\newtheorem{rem}[thm]{Remark}

\newenvironment{thmenum}{\begin{enumerate}[label=\textup{(\roman*)},leftmargin=9mm]}{\end{enumerate}}

\newcommand{\sd}{\mathrel{\triangle}}
\newcommand{\sm}{\setminus}
\newcommand{\restr}{{\restriction}}
\newcommand{\emptypart}{\ep_\emptyset}

\begin{document}

\title{Congruences on infinite partition and partial Brauer monoids}
\author{James East\footnote{Centre for Research in Mathematics, School of Computing, Engineering and Mathematics, Western Sydney University, Locked Bag 1797, Penrith NSW 2751, Australia. {\it Email:} {\tt j.east\,@\,westernsydney.edu.au}} 
\ and
Nik Ru\v{s}kuc\footnote{Mathematical Institute, School of Mathematics and Statistics, University of St Andrews, St Andrews, Fife KY16 9SS, UK. {\it Email:} {\tt nik.ruskuc\,@\,st-andrews.ac.uk}}}
\date{}

\maketitle

\begin{abstract}
We give a complete description of the congruences on the partition monoid $\P_X$ and the partial Brauer monoid $\PB_X$, where $X$ is an arbitrary infinite set, and also of the lattices formed by all such congruences.  Our results complement those from a recent article of East, Mitchell, Ru\v{s}kuc and Torpey, which deals with the finite case.  As a consequence of our classification result, we show that the congruence lattices of $\P_X$ and $\PB_X$ are isomorphic to each other, and are distributive and well quasi-ordered.  We also calculate the smallest number of pairs of partitions required to generate any congruence; when this number is infinite, it depends on the cofinality of certain limit cardinals.

\emph{Keywords}: Diagram monoids; Partition monoids; Partial Brauer monoids; Congruences; Well quasi-orderedness.

MSC:  20M20, 08A30, 06A06, 03E04.

\end{abstract}

\tableofcontents

\section{Introduction}\label{sect:intro}

Diagram algebras play a central role in many different areas of mathematics and science, from invariant theory \cite{Brauer1937,LZ2012} and knot theory \cite{Jones1987,Kauffman1990} to theoretical physics \cite{TL1971,Martin1994,Jones1994}.  
Typically, these algebras have bases consisting of various kinds of set partitions, which are represented and multiplied diagrammatically, hence the name.
Every diagram algebra arises as a twisted semigroup algebra~\cite{HR2005,Wilcox2007} of an underlying diagram monoid, key examples including (partial) Brauer monoids \cite{MarMaz2014,Brauer1937,Maz1998,DEG2017}, Temperley-Lieb (a.k.a.~Jones or Kauffman) monoids \cite{TL1971,BDP2002,LF2006}, Motzkin monoids \cite{DEG2017,BH2014} and partition monoids \cite{Jones1994,Martin1994,FL2011,JEgrpm}.  More background and references on diagram algebras and monoids may be found in the surveys \cite{Martin2008,Koenig2008} or in the introductions to \cite{DEEFHHL2015,EG2017}.  

The motivation for studying diagram monoids themselves comes from a number of directions.  On the one hand, some studies of diagram monoids have led immediately to important consequences for the corresponding algebras; for example, presentations of the algebras are deduced from presentations of the monoids in \cite{JE2020_TL,JEgrpm,East2018_rook,East2018_Pn}, while cellularity of the algebras (and hence a great deal of representation theoretic information) is deduced from structural properties of the monoids in \cite{Wilcox2007,DEG2017}; cf.~\cite{Xi1999}.
Families of diagram monoids are also playing an increasingly prominent role in semigroup theory itself, underpinning a number of recent studies of pseudovarieties of finite semigroups; see especially the work of Auinger and Volkov and their collaborators \cite{AV2020,ACHLV2015,ADV2012_2,ADV2012,Auinger2014,Auinger2012,CHKLV2019,KV2019}.
Diagram monoids have also provided a fruitful connection between algebra and many combinatorial themes such as graphs and matchings, lattice path enumeration, planar geometry and tilings, analysis of integer partitions, and more \cite{DEG2017,EG2017,DEEFHHL2015,MM2007,EmoJoKa,EKMW2018,AMM2015}.


The partition monoid over a set $X$, denoted $\P_X$ (and defined in Subsection \ref{subsect:M}), contains natural copies of many ``classical'' monoids \cite{JEgrpm,EF2012}, including the symmetric group $\S_X$, the full transformation monoid $\T_X$, and the symmetric and dual symmetric inverse monoids $\I_X$ and~$\JJ_X$.  The fundamental importance of these four monoids stems largely from their universality with respect to certain classes of algebraic structures.  The well-known Cayley Theorems \cite[Theorem~1.1.2]{Howie} state that every group or semigroup embeds in some $\S_X$ or $\T_X$, respectively, while the Wagner-Preston Theorem \cite[Theorem~5.1.7]{Howie} and the FitzGerald-Leech Theorem \cite[Theorem 4.1]{FL1998} state that every inverse semigroup embeds in some~$\I_X$ and also in some~$\JJ_X$.  As a result, these classical monoids have received an enormous amount of attention over the years, with the finite and infinite theories developing into somewhat separate disciplines.  
For example, a significant theme in the theory of infinite symmetric groups is exhibiting their ``largeness'' via concepts such as cofinality, the Bergman and small index properties, word universality, maximal and normal subgroups, to mention just a few \cite{CMM1996,Ball1966,BCPPW1994,ST1997,Bergman2006,DNT1986,MP1990,MN1990,Ore1951,Lyndon1990,DM1999}.  Building on this, numerous studies have compared and contrasted
other infinite classical monoids with the symmetric groups; see for example \cite{Banach1935,Sierpinski1935,MMMP2012,MMR2009,EMP2015,PS2013,HRH,HHMR,HJMP2018,Malcev1952} and the references therein.

Since $\P_X$ contains each of the above classical monoids, the partition monoids are also universal within the classes of groups, semigroups and inverse semigroups.  In this way, the Cayley, Wagner-Preston and FitzGerald-Leech Theorems may all be unified in a single result concerning embeddings in $\P_X$.  Moreover, $\P_X$ has a number of additional structural features not shared by~$\T_X$, including an involution and an ordering compatible with the product (and involution), making it an even more attractive target for embeddings.  All of this points to the fundamental importance of the partition monoids, and suggests that they are worthy of focussed study.  


While most existing investigations of partition monoids and other diagram monoids are restricted to the finite case, a number of recent studies treat the infinite case as well \cite{FitzGerald2013,JE2012,JE2017,FL2011,JEipms,JE_IBM,EF2012,EG2021}.  
One feature of virtually all articles on (finite and infinite) partition monoids is a two-way flow of ideas and methodology between diagram monoids on the one hand and classical ones on the other.  Previous studies of classical monoids have provided useful tools for studying~$\P_X$, while, in turn, diagrammatic techniques yield unifying and clarifying consequences for the classical monoids.  For example, in \cite{EF2012} the idempotent-generated subsemigroup of (finite and infinite)~$\P_X$ was described, and Howie's celebrated result on the idempotent-generated subsemigroup of~$\T_X$~\cite{Howie1966} was deduced as a corollary.  Another consequence of the main results of \cite{EF2012} is that every semigroup embeds in an idempotent-generated, ordered, involutory monoid.  
The diagrammatic methods introduced in the current article also lead to new ways to understand and prove a number of classical results \cite{Malcev1952,Liber1953,Scheiblich1973} concerning congruences, as we explain in Section~\ref{sect:OM}.

Of the above-mentioned papers on classical monoids, the most relevant to our current purposes is Mal'cev's 1952 article \cite{Malcev1952}, the main results of which classify the congruences on an arbitrary full transformation monoid.  (An excellent account of Mal'cev's paper may be found in \cite[Section~10.8]{CPbook2}.)  A congruence on an algebraic structure~$S$ is an equivalence relation compatible with all the basic operations on $S$; the set $\Cong(S)$ of all such congruences forms a lattice under inclusion.  In the cases of groups and rings, congruences correspond to normal subgroups and ideals, respectively, but congruences are not always determined by substructures in general, as is the case with monoids and categories for example.  Nevertheless, congruences are the tools for constructing quotient structures, defining kernels of homomorphisms/representations, and so on.
Mal'cev's article \cite{Malcev1952} was followed by a number of studies \cite{Liber1953,Malcev1953,Sutov1961_2,Sutov1961} classifying congruences on other classical monoids, including some of those discussed above.


The article \cite{EMRT2018} initiated the study of congruences on diagram monoids, the main results being the complete descriptions of the congruence lattices of finite partition, planar partition, Brauer, partial Brauer, Temperley-Lieb and Motzkin monoids.  Inspiration was drawn from Mal'cev's above-mentioned study of full transformation monoids \cite{Malcev1952}.  While there are some intriguing parallels between the theories of diagram and transformation monoids, the results of \cite{EMRT2018} also highlighted some striking differences.  For example, while the congruence lattices of finite full transformation monoids form chains under inclusion, the same is not true for any of the diagram monoids studied in \cite{EMRT2018}.  For each of these, the lattice has a prism-shaped lower part, the existence of which is partly explained by structural properties of the minimal ideal.

The current article is a natural sequel to \cite{EMRT2018}; it continues the study of congruences of diagram monoids, moving in the direction of infinite monoids.  Specifically, we are concerned here with the partition monoid $\P_X$ and the partial Brauer monoid $\PB_X$ over an arbitrary infinite set~$X$.  The paper has two broad goals, and we address these in two separate parts:
\bit
\item Part \ref{part:I} gives a complete classification of the congruences of infinite $\P_X$ and $\PB_X$.  
The statement of the classification is given in Section \ref{sect:statement} (see Theorem \ref{thm-main}), and the proof in Sections \ref{sect:stage1}--\ref{sect:tech}.
\item Part \ref{part:II} gives a detailed analysis of the algebraic and combinatorial structure of the congruence lattices $\Cong(\P_X)$ and $\Cong(\PB_X)$.  We describe the inclusion relation and the meet and join operations in Section~\ref{sect:order} (Theorems \ref{thm:comparisons} and \ref{thm:mj}), draw Hasse diagrams in Section \ref{sect:Hasse} (Figures \ref{fig:1layer}--\ref{fig:CongPX}), prove that the lattices are distributive and well quasi-ordered in Section \ref{sect:global} (Theorems \ref{thm:distr} and \ref{thm:wqo}), calculate the smallest sizes of generating sets in Section~\ref{sect:gen} (Theorems \ref{thm:pc}, \ref{thm:ranksCT1} and \ref{thm:ranksCT2}), and discuss connections with other results from the literature in Section \ref{sect:OM} where we also discuss possible directions for future research.
\eit

One of the intriguing consequences of our results is that the lattices $\Cong(\P_X)$ and $\Cong(\PB_X)$ are isomorphic (Corollary \ref{cor:iso}), even though the monoids $\P_X$ and $\PB_X$ are not (Proposition \ref{prop:not_iso}).  Indeed, our main results are essentially identical for both $\P_X$ and $\PB_X$, and can generally be proved with a uniform argument that works for both monoids; as an exception, some of the results in Sections \ref{sect:stage2} and \ref{sect:tech} require substantially different arguments for the two monoids, those for~$\P_X$ having a set-theoretical flavour, and those for $\PB_X$ being more combinatorial in nature.  As with the finite case \cite{EMRT2018}, certain parallels may be made with the theory of infinite transformation monoids, but at the same time a number of differences arise, some more subtle than others; these will be drawn out during the exposition.


\section{Preliminaries}\label{sect:prelim}

\subsection{General notation and conventions}\label{subsect:notation}

We work in standard ZFC set theory; see for example \cite[Chapters 1, 5 and 6]{Jech2003}.  We denote by~$\xi^+$ the successor of a cardinal $\xi$.  For cardinals $\xi_1$ and $\xi_2$, we write $[\xi_1,\xi_2]$ and $[\xi_1,\xi_2)$ for the (possibly empty) sets of all cardinals $\ze$ satisfying $\xi_1\leq\ze\leq\xi_2$ or $\xi_1\leq\ze<\xi_2$, respectively.  Unless otherwise stated, we assume all indexing sets are faithful, meaning that when we use notation such as $\set{x_i}{i\in I}$ we assume that the map $i\mt x_i$ is injective.  If $\si$ is an equivalence relation on a set $X$, and if $Y$ is a subset of $X$, we write $\si\restr_Y=\si\cap(Y\times Y)$ for the restriction of $\si$ to~$Y$; if~$Y$ is a union of $\si$-classes, we write $Y/\si$ for the set of all $\si$-classes contained in $Y$.  If $\si$ and~$\tau$ are equivalences on a set $X$, we denote by $\si\vee\tau$ the join of $\si$ and $\tau$, which is the least equivalence on $X$ containing both $\si$ and $\tau$: i.e., the transitive closure of $\si\cup\tau$.  For any set~$X$ we denote by~$\De_X=\set{(x,x)}{x\in X}$ the diagonal relation on~$X$.  The symbol $\sd$ denotes symmetric difference: if $X$ and $Y$ are sets, then $X\sd Y=(X\sm Y)\cup(Y\sm X)$.  If $G$ is a group, we write~$H\leq G$ and $N\normal G$ to indicate that $H$ is a subgroup and $N$ a normal subgroup.

\subsection{Semigroups and congruences}\label{subsect:S}

Let $S$ be a semigroup.  A (binary) relation $\si$ on $S$ is \emph{left compatible} if for all $(x,y)\in\si$ and $a\in S$, we have $(ax,ay)\in\si$.  \emph{Right compatible} relations are defined analogously.  A relation is \emph{compatible} if it is both left and right compatible.  A \emph{congruence} on $S$ is an equivalence relation that is compatible.  The set of all congruences on $S$ is denoted $\Cong(S)$; it is a lattice under inclusion.
The most obvious congruences on~$S$ are the trivial and universal congruences:
\[
\De_S=\set{(x,x)}{x\in S} \AND \nabla_S=S\times S.
\]
These are the least and greatest elements of $\Cong(S)$, respectively.  Another family of congruences comes from ideals.  Recall that a subset $I\sub S$ is an \emph{ideal} of $S$ if $ax,xa\in I$ for all $x\in I$ and $a\in S$.  For an ideal $I$, we have the so-called \emph{Rees congruence}
\[
R_I = \De_S \cup (I\times I).
\]
Further general families of congruences will be discussed in Section \ref{subsect:EMRT}.

Let $S$ be a semigroup.  As usual, $S^1$ denotes $S$ itself if $S$ is a monoid, or else the monoid obtained by adjoining an identity to $S$.  Recall that Green's pre-orders $\leqR$, $\leqL$ and $\leqJ$ are based on the inclusion ordering on principal ideals; specifically, for $x,y\in S$:
\[
x\leqR y \iff xS^1\subseteq yS^1 \COMMA
x\leqL y \iff S^1 x\subseteq S^1y \COMMA
x\leqJ y \iff S^1 xS^1\subseteq S^1yS^1 .
\]
Note that $x\leqR y \iff x\in yS^1$, with similar statements for $\leqL$ and $\leqJ$, so that Green's pre-orders may also be thought of in terms of division.
Green's $\R$, $\L$ and $\J$ relations are defined by ${\R}={\leqR}\cap{\geqR}$, ${\L}={\leqL}\cap{\geqL}$ and ${\J}={\leqJ}\cap{\geqJ}$.  Green's $\H$ and $\D$ relations are defined to be the meet and join, respectively, of $\R$ and $\L$ in the lattice of equivalences on~$S$.  That is, ${\H}={\R}\cap{\L}$, while ${\D}={\R}\vee{\L}$ is the least equivalence on $S$ containing both $\R$ and $\L$; it is well known that ${\D}={\R}\circ{\L}={\L}\circ{\R}\sub{\J}$ in any semigroup, and that ${\D}={\J}$ in any finite semigroup.  If $\K$ denotes any of Green's relations, and if $x\in S$, we write $K_x=\set{y\in S}{x\K y}$ for the $\K$-class of $x$ in $S$.  The $\leqJ$ pre-order on $S$ yields a natural partial order, denoted $\leq$, on the set $S/{\J}$ of all $\J$-classes of $S$: $J_x\leq J_y\iff x\leqJ y$.

\subsection{Partition and partial Brauer monoids}\label{subsect:M}

Let $X$ be an arbitrary set, and let $X'=\set{x'}{x\in X}$ be a disjoint copy of $X$.  The \emph{partition monoid} over $X$, denoted $\P_X$, consists of all set partitions of $X\cup X'$ under a product described below.  
So an element of $\P_X$ is of the form $\al=\set{A_i}{i\in I}$, where the $A_i$ are non-empty, pairwise disjoint subsets of $X\cup X'$ satisfying $X\cup X'=\bigcup_{i\in I}A_i$; the $A_i$ are called the \emph{blocks} of~$\al$.  
Of course $\al$ can also be viewed as an equivalence relation on $X\cup X'$; while it will be convenient to do so in the next two paragraphs, when defining the product on $\P_X$, we will generally not do so, and will always regard $\al$ as a set of subsets of $X\cup X'$ as above.

To define the product in $\P_X$, introduce yet another copy
$X''=\set{x''}{x\in X}$ of $X$, disjoint from both $X$ and $X'$.
For $\al\in\P_X$, denote by $\alpha_\downarrow$ the equivalence relation on the set $X\cup X'\cup X''$
 obtained by renaming every $x'$ into $x''$ and adding the 
diagonal $\Delta_{X'}$.
Dually, $\alpha^\uparrow$ is obtained by replacing every $x$ by $x''$ and adding the diagonal $\Delta_X$.

Now let $\al,\be\in\P_X$.  Noting that $\al_\downarrow$ and~$\be^\uparrow$ are equivalences on $X\cup X'\cup X''$, we write as usual $\al_\downarrow \vee \be^\uparrow$ for the least equivalence on $X\cup X'\cup X''$ containing both $\al_\downarrow$ and~$\be^\uparrow$, and define the \emph{product} $\alpha\beta\in\P_X$ to be $(\alpha_\downarrow\vee\beta^\uparrow)\restr_{X\cup X'}$:
i.e., the restriction of $\alpha_\downarrow\vee\beta^\uparrow$ to the set $X\cup X'$.
See below for an alternative, more visual, interpretation of this product.

Throughout the paper we will use two handy ways of representing and visualising partitions.
The first was introduced in \cite{EF2012}, harking back to the standard two-line notation for mappings on a set \cite[p241]{CPbook2}, and is defined as follows.
A non-empty subset $A$ of $X\cup X'$ is called
\bit
\item a \emph{transversal} if both $A\cap X$ and $A\cap X'$ are non-empty,
\item an \emph{upper non-transversal} if $A\subseteq X$, or
\item a \emph{lower non-transversal} if $A\subseteq X'$.
\eit
If $A$ is a transversal, then we refer to $A\cap X$ and $A\cap X'$ as the \emph{upper} and \emph{lower parts} of $A$, respectively.  
If $A\sub X$, we write $A'=\set{a'}{a\in A}\sub X'$.  If $\al\in\P_X$, we will write
\[
\al = \partABCD_{i\in I,\ j\in J,\ k\in K}
\]
to indicate that $\al$ has transversals $A_i\cup B_i'$ ($i\in I$), upper non-transversals $C_j$ ($j\in J$), and lower non-transversals $D_k'$ ($k\in K$).  Sometimes we just write $\al=\partABCD$, with the indexing sets~$I$,~$J$ and $K$ being implied, rather than explicitly named.  
Note that dashes are omitted from elements of $X'$ in this notation.
For extra convenience, some (but not necessarily all) singleton blocks of~$\al$ may be omitted from this notation; in other words, if $y\in X\cup X'$ does not belong to any of the blocks listed in $\al=\partABCD$ then $\{y\}$ is a singleton block of $\alpha$.

The second representation for partitions is more visual, and goes back to Brauer \cite{Brauer1937}.
Here, a partition $\al\in\P_X$ is represented as a graph with vertex set $X\cup X'$ and edges chosen so that its connected components are the blocks of $\al$; such a graph is not unique in general, but we identify~$\alpha$ with \emph{any} such graph.  
We think of the vertices from $X$ as \emph{upper} vertices, and those from $X'$ as \emph{lower} vertices.
The computation of the product $\alpha\beta$ of two partitions $\alpha,\beta\in\P_X$ can 
now be interpreted as follows.
Take the graphs on $X\cup X'\cup X''$ corresponding to $\alpha_\downarrow$ and $\beta^\uparrow$ (as defined above), typically drawn with vertices from $X''$ in a new middle row.
The \emph{product graph} $\Pi(\al,\be)$ is the graph on vertex set $X\cup X'\cup X''$ whose edge set is the union of the edge sets of $\al_\downarrow$ and $\be^\uparrow$.  
The product $\al\be$ is  the partition of $X\cup X'$
such that elements $u,v\in X\cup X'$ belong to the same block of $\al\be$ if and only if $u$ and~$v$ belong to the same connected component of $\Pi(\al,\be)$.

As an example, consider the partitions
\begin{align*}
\al &= \big\{ \{1,4\},\{2,3,4',5'\},\{5,6\},\{1',2',6'\},\{3'\}\big\} , \\
\be &= \big\{ \{1,2\}, \{3,4,1'\}, \{5,4',5',6'\}, \{6\}, \{2'\}, \{3'\} \big\}
\end{align*}
with $X=\{1,\ldots,6\}$. In the tableaux notation (keeping in mind the convention regarding singletons), they are written as
\[
\alpha=
\Big(   \hspace{-1.5 truemm}
{ \scriptsize \renewcommand*{\arraystretch}{1}
\begin{array} {\cstart|\c|\cend}
2,3\: &\: 1,4\: &\: 5,6 \\ \cline{2-3} 4,5\: & \multicolumn{2}{c}{1,2,6} 
\end{array} 
}
\hspace{-1.5 truemm} \Big)
\AND
\beta=
\Big(   \hspace{-1.5 truemm}
{ \scriptsize \renewcommand*{\arraystretch}{1}
\begin{array} {\cstart|\c|\cend}
3,4\: &\: 5\: &\: 1,2 \\ \cline{3-3} 1\: &\: 4,5,6\: &
\end{array} 
}
\hspace{-1.5 truemm} \Big).
\]
The computation of the product
\[
\alpha\beta=
\Big(   \hspace{-1.5 truemm}
{ \scriptsize \renewcommand*{\arraystretch}{1}
\begin{array} {\cstart|\c|\cend}
2,3\: &\: 1,4\: &\: 5,6 \\ \cline{2-3} 1,4,5,6\: & \multicolumn{2}{c}{}
\end{array} 
}
\hspace{-1.5 truemm} \Big)
\]
via the product graph is given in Figure \ref{fig:P6}.
An example with countably infinite $X$ is given in Figure \ref{fig:PBX}.

\begin{figure}
\begin{center}
\begin{tikzpicture}[scale=.5]
\begin{scope}[shift={(0,0)}]	
\uvs{1,...,6}
\lvs{1,...,6}
\uarcx14{.6}
\uarcx23{.3}
\uarcx56{.3}
\darc12
\darcx26{.6}
\darcx45{.3}
\stline34
\draw(0.6,1)node[left]{$\al=$};
\draw[->](7.5,-1)--(9.5,-1);
\end{scope}
\begin{scope}[shift={(0,-4)}]	
\uvs{1,...,6}
\lvs{1,...,6}
\uarc12
\uarc34
\darc45
\darc56
\stline31
\stline55
\draw(0.6,1)node[left]{$\be=$};
\end{scope}
\begin{scope}[shift={(10,-1)}]	
\uvs{1,...,6}
\lvs{1,...,6}
\uarcx14{.6}
\uarcx23{.3}
\uarcx56{.3}
\darc12
\darcx26{.6}
\darcx45{.3}
\stline34
\draw[->](7.5,0)--(9.5,0);
\end{scope}
\begin{scope}[shift={(10,-3)}]	
\uvs{1,...,6}
\lvs{1,...,6}
\uarc12
\uarc34
\darc45
\darc56
\stline31
\stline55
\end{scope}
\begin{scope}[shift={(20,-2)}]	
\uvs{1,...,6}
\lvs{1,...,6}
\uarcx14{.6}
\uarcx23{.3}
\uarcx56{.3}
\darc14
\darc45
\darc56
\stline21
\draw(6.4,1)node[right]{$=\al\be$};
\end{scope}
\end{tikzpicture}
\caption{Two partitions $\al,\be\in\P_X$ (left), the product graph $\Pi(\al,\be)$ (middle), and their product $\al\be\in\P_X$ (right), where $X=\{1,\ldots,6\}$.}
\label{fig:P6}
\end{center}
\end{figure}

\begin{figure}
   \begin{center}
\begin{tikzpicture}[scale=.4]
\begin{scope}[shift={(0,0)}]	
\buvs{1,...,5}
\blvs{1,...,5}
\stline11
\uarcs{2/3,4/5}
\darcs{2/3,4/5}
\draw[dotted] (6,2)--(9,2);
\draw[dotted] (6,0)--(9,0);
\draw(0.6,1)node[left]{$\al=$};
\draw[->](10.5,-1)--(12.5,-1);
\end{scope}
\begin{scope}[shift={(0,-4)}]	
\buvs{1,...,5}
\blvs{1,...,5}
\uarcs{1/2,3/4}
\darcs{1/2,3/4}
\darchalf5{5.5}
\uarchalf5{5.5}
\colv{5.5}1{white}
\draw[dotted] (6,2)--(9,2);
\draw[dotted] (6,0)--(9,0);
\draw(0.6,1)node[left]{$\be=$};
\end{scope}
\begin{scope}[shift={(13,-1)}]	
\buvs{1,...,5}
\blvs{1,...,5}
\stline11
\uarcs{2/3,4/5}
\darcs{2/3,4/5}
\draw[dotted] (6,2)--(9,2);
\draw[dotted] (6,0)--(9,0);
\draw[->](10.5,0)--(12.5,0);
\end{scope}
\begin{scope}[shift={(13,-3)}]	
\buvs{1,...,5}
\blvs{1,...,5}
\uarcs{1/2,3/4}
\darcs{1/2,3/4}
\darchalf5{5.5}
\uarchalf5{5.5}
\colv{5.5}1{white}
\draw[dotted] (6,2)--(9,2);
\draw[dotted] (6,0)--(9,0);
\end{scope}
\begin{scope}[shift={(26,-2)}]	
\buvs{1,...,5}
\blvs{1,...,5}
\uarcs{2/3,4/5}
\darcs{1/2,3/4}
\darchalf5{5.5}
\draw[dotted] (6,2)--(9,2);
\draw[dotted] (6,0)--(9,0);
\draw(9.4,1)node[right]{$=\al\be$};
\end{scope}
\end{tikzpicture}
\caption{Two partitions $\al,\be\in\P_X$ (left), the product graph $\Pi(\al,\be)$ (middle), and their product $\al\be\in\P_X$ (right), where $X=\{0,1,2,\ldots\}$.}
\label{fig:PBX}
\end{center}
\end{figure}
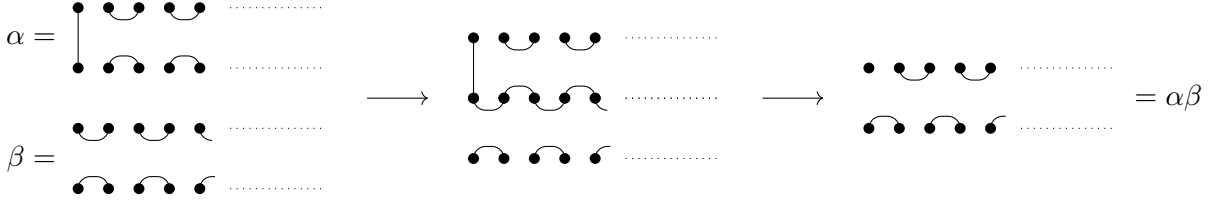

For a subset $Y\sub X$, we define the partition
\[
\ep_Y = \tbinom yy_{y\in Y}.
\]
It is easy to see that $\ep_Y\ep_Z=\ep_{Y\cap Z}$ for all $Y,Z\sub X$.  The partition $\epsilon_X$ is the identity element of~$\P_X$.
 A partition $\al\in\P_X$ is a unit (i.e., invertible with respect to $\ep_X$) if and only if each block of $\al$ is of the form $\{x,y'\}$ for some $x,y\in X$.  The group of all such units is clearly isomorphic to the symmetric group $\S_X$, which consists of all permutations of $X$; thus, we will identify $\S_X$ with the group of units of $\P_X$.

The element $\emptypart$, with all the blocks trivial, will also play an important role in many of our calculations, but it is worth noting that $\emptypart$ is not a zero element in $\P_X$; indeed, $\P_X$ has no zero element (unless $|X|\leq1$).

The \emph{partial Brauer monoid} over $X$, denoted $\PB_X$, is the submonoid of $\P_X$ consisting of all partitions whose blocks have size at most $2$.  Note that $\PB_X$ contains $\S_X$.  When $X$ is finite, the set of all partitions whose blocks all have size precisely  $2$ is also a monoid, known as the \emph{Brauer monoid} and denoted $\B_X$.  But when $X$ is infinite, $\B_X$ is not a submonoid, as the example in Figure \ref{fig:PBX} shows; here $\al,\be\in\B_X$ yet $\al\be\not\in\B_X$.  In fact, it was shown in \cite[Corollary 4.4]{JE_IBM} that when $X$ is infinite, any element of $\PB_X$ is the product of two elements of $\B_X$.

The \emph{domain}, \emph{codomain}, \emph{kernel}, \emph{cokernel} and \emph{rank} of a partition $\al=\partABCD\in\P_X$ are defined by
\bit
\item $\dom(\al) = \set{x\in X}{\text{$x$ belongs to a transversal of $\al$}} = \bigcup_{i\in I}A_i$,
\item $\codom(\al) = \set{x\in X}{\text{$x'$ belongs to a transversal of $\al$}} = \bigcup_{i\in I}B_i$,
\item $\ker(\al) = \bigset{(x,y)\in X\times X}{\text{$x$ and $y$ belong to the same block of $\al$}}$, the equivalence relation on $X$ 
associated with the partition $\overline{\alpha}=\set{A_i}{i\in I}\cup\set{C_j}{j\in J}$,
\item $\coker(\al) = \bigset{(x,y)\in X\times X}{\text{$x'$ and $y'$ belong to the same block of $\al$}}$, the equivalence relation on $X$ associated with the partition $\underline{\alpha}=\set{B_i}{i\in I}\cup\set{D_k}{k\in K}$,
\item $\rank(\al)=|I|$, the number of transversals of $\al$.
\eit
The above parameters allow for convenient descriptions of Green's relations and pre-orders on~$\P_X$ and~$\PB_X$:

\begin{lemma}\label{lem:Green_M}
Let $X$ be an arbitrary set, let $\M$ be either $\P_X$ or $\PB_X$, and let $\al,\be\in\M$.  Then in $\M$,
\begin{thmenum}
\item \label{it:GMiv} $\al\R\be \iff \dom(\al)=\dom(\be)$ and $\ker(\al)=\ker(\be)$,
\item \label{it:GMv}  $\al\L\be \iff \codom(\al)=\codom(\be)$ and $\coker(\al)=\coker(\be)$,
\item \label{it:GMvi} $\al\J\be \iff \al\D\be \iff \rank(\al)=\rank(\be)$,
\item \label{it:GMiii} $\al\leqJ\be \iff \rank(\al)\leq\rank(\be)$.
\end{thmenum}
\end{lemma}

\pf
This was proved in \cite[Lemma 3.1 and Theorem 3.3]{FL2011} in the case of $\M=\P_X$, using slightly different terminology.  The same proofs apply virtually unmodified to $\M=\PB_X$.
\epf

The next result follows quickly from parts \ref{it:GMvi} and \ref{it:GMiii} of Lemma \ref{lem:Green_M}.  For the statement, recall that $|X|^+$ denotes the successor cardinal to $|X|$.

\begin{cor}\label{cor:ideals_M}
Let $X$ be an arbitrary set, and let $\M$ be either $\P_X$ or $\PB_X$.  
\begin{thmenum}
\item 
\label{it:idMi}
The ideals of $\M$ are the sets $I_\xi = \set{\al\in\M}{\rank(\al)<\xi}$, for each cardinal ${\xi\in[1,|X|^+]}$,
and they form a chain under inclusion: $I_{\xi_1}\sub I_{\xi_2} \iff {\xi_1}\leq{\xi_2}$.  
\item 
\label{it:idMii}
The ${\D}={\J}$-classes of $\M$ are the sets $D_\xi = \set{\al\in\M}{\rank(\al)=\xi}$, for each cardinal $\xi\in[0,|X|]$, and
they form a chain under the $\J$-class ordering: ${D_{\xi_1}\leq D_{\xi_2} \iff {\xi_1}\leq{\xi_2}}$.
\epfres
\end{thmenum}
\end{cor}

In particular, the chains of ideals and of ${\J}={\D}$-classes of $\M$ are well-ordered, a fact that will prove crucial in what follows.  We also need to know that all group $\H$-classes of $\M$ are symmetric groups:

\begin{lemma}
Let $X$ be an arbitrary set, let $\M$ be either $\P_X$ or $\PB_X$, and let $\xi\in[0,|X|]$.  Then any group $\H$-class of $\M$ contained in $D_\xi$ is isomorphic to the symmetric group $\S_\xi$.
\end{lemma}

\pf
Fix some $A\sub X$ with $|A|=\xi$.  Since $\ep_A\in D_\xi$, \cite[Proposition 2.3.6]{Howie} says that all group $\H$-classes in $D_\xi$ are isomorphic to the $\H$-class of $\ep_A$, and it is easy to see that this is isomorphic to $\S_A\cong\S_\xi$.
\epf

One of the intriguing consequences of our main results in this paper is that the congruence lattices $\Cong(\P_X)$ and $\Cong(\PB_X)$ are isomorphic for any infinite set $X$.  This is also true in the finite case \cite[Theorems~5.4 and~6.1]{EMRT2018}.  Isomorphism of the lattices would of course be no surprise if the monoids $\P_X$ and $\PB_X$ were themselves isomorphic, but this is not the case except 
trivially for $|X|\leq1$:

\begin{prop}\label{prop:not_iso}
If $|X|\geq2$, then the monoids $\P_X$ and $\PB_X$ are not isomorphic.
\end{prop}

\pf
This is clear for $1<|X|<\aleph_0$, as $\PB_X$ is a proper subset of $\P_X$, so we assume~$X$ is infinite.  It suffices to prove the following two claims:
\ben
\item \label{it:not_iso_1} There exists more than one $\R$-class of $\PB_X$ containing only one idempotent.
\item \label{it:not_iso_2} There exists only one $\R$-class of $\P_X$ containing only one idempotent.
\een
First note that the $\R$-class of the identity element of any monoid has only one idempotent (the identity itself).  In what follows, we make repeated use of Lemma \ref{lem:Green_M} \ref{it:GMiv}. 

\pfitem{\ref{it:not_iso_1}}  Fix some $x\in X$ and write $Y=X\sm\{x\}$.  
We will show that $\ep_Y=\binom yy$ is the only idempotent in its $\R$-class (in $\PB_X$).
Indeed, supposing  $\ep_Y\R\al=\alpha^2\in\PB_X$, we may write $\al=\partn{y}{}{y\psi}{A_i}$ where $\psi$ is some injective map $Y\to X$. 
For any~$y\in Y$, since $\{y,(y\psi)'\}$ is a block of $\al=\al^2$, there is a path from $(y\psi)''$ to $y''$ in the product graph $\Pi(\al,\al)$; but such a path must have length zero since~$\al$ has no non-trivial upper non-transversals, and this means that $y\psi=y$.  From this it quickly follows that~$\al=\ep_Y$.  

\pfitem{\ref{it:not_iso_2}}  Consider some $\R$-class $R$ of $\P_X$ not containing the identity element $\ep_X$, and let $\partABCD$ be a representative of $R$.  Since $\ep_X\not\in R$, either $I$ is empty, or $|A_i|\geq2$ for some $i\in I$, or else 
both $I$ and~$J$ are non-empty.
\bit
\item If $I$ is empty, then $\partXX{C_j}{X}$ and $\partXX{C_j}{}$ are distinct idempotents of $R$. 
\item If $|A_i|\geq2$ for some $i\in I$, then we fix distinct $x,y\in A_i$, write $L=I\sm\{i\}$, and note that $\partXXIV{A_i}{A_l}{C_j}{x}{A_l}{}$ and $\partXXIV{A_i}{A_l}{C_j}{y}{A_l}{}$ are distinct idempotents of $R$. 
\item 
If $I,J\neq\emptyset$, let
 $C=\bigcup_{j\in J}C_j\neq\emptyset$, fix some $i\in I$, write $L=I\sm\{i\}$, and note that $\partXXIV{A_i}{A_l}{C_j}{A_i}{A_l}{}$ and $\partXXIV{A_i}{A_l}{C_j}{A_i\cup C}{A_l}{}$  are distinct idempotents of $R$. \qedhere 
\eit
\epf

\part{Classification of congruences}\label{part:I}

This part of the paper is devoted to the classification of congruences on the partition monoid~$\P_X$ and partial Brauer monoid $\PB_X$ over an arbitrary infinite set $X$.  The statement of the classification theorem (Theorem \ref{thm-main}) is given in Section \ref{sect:statement}, where we also discuss the strategy of proof.  The proof itself is given in Sections \ref{sect:stage1}--\ref{sect:tech}.  Almost all we say applies equally to both $\P_X$ and $\PB_X$, so as before we will use~$\M$ to stand for either of these two monoids.

\section{The classification theorem}\label{sect:statement}

\subsection{Statement of the theorem}\label{subsect:statement}

All congruences on $\M$ are built from five basic relations.  These are denoted $R_\xi$, $\lam_\ze$, $\rho_\ze$, $\mu_\ze$ and~$\nu_N$, and will be defined shortly; their deeper significance will be discussed in more detail in subsequent sections.

First, to each ideal $I_\xi$ of $\M$, as described in Corollary \ref{cor:ideals_M} \ref{it:idMi}, there corresponds the \emph{Rees congruence}
\[
R_\xi = \De_{\M}\cup(I_\xi\times I_\xi) \qquad \text{for any } \xi\in[1,|X|^+].
\]

Next, we have the relation
\[
\mu_\zeta = \bigset{ (\alpha,\beta)\in\M\times\M}{|\alpha\sd\beta|<\zeta} 
\qquad \text{for any } \ze\in[1,|X|^+].
\]
It is important to note here that $\al$ and $\be$ are regarded as sets of subsets of $X\cup X'$, not as equivalence relations on~$X\cup X'$ (i.e., not as sets of ordered pairs).
Informally, $|\al\sd\be|$ measures the difference between~$\al$ and~$\be$, by counting the blocks belonging to only one of them,  and $\mu_\ze$ gathers together the pairs of partitions that differ by less than $\ze$. 

The next two relations are analogous to $\mu_\zeta$, but refer to the partitions
$\overline{\alpha}$ and $\underline{\alpha}$ induced by the kernel and cokernel of $\alpha$ (and defined before Lemma \ref{lem:Green_M}), respectively: 
\begin{align*}
\lambda_\zeta&=\bigset{ (\alpha,\beta)\in\M\times\M}{ |\overline{\alpha}\sd\overline{\beta}|<\zeta} &&\text{for any  $\ze\in[1,|X|^+]$,}\\
\rho_\zeta&=\bigset{ (\alpha,\beta)\in\M\times\M}{ |\underline{\alpha}\sd\underline{\beta}|<\zeta} &&\text{for any  $\ze\in[1,|X|^+]$.}
\end{align*}
We will also need the intersections of the relations $\mu_\zeta$, $\lambda_\zeta$ and $\rho_\zeta$ with the Rees congruence~$R_\eta$:
\[
\mu_\zeta^\eta=\mu_\zeta\cap R_\eta \COMMA
\lambda_\zeta^\eta=\lambda_\zeta\cap R_\eta \COMMA
\rho_\zeta^\eta=\rho_\zeta\cap R_\eta .
\]

To describe the final kind of relation, we must first introduce some further notation.  Let ${n\in [1,\aleph_0)}$ be a positive integer, and let $N$ be a normal subgroup of the symmetric group~$\S_n$, which consists of all permutations of $\{1,\ldots,n\}$.  
Consider two partitions $\alpha,\beta\in D_n$ (the $\D$-class of all rank-$n$ elements of $\M$) such that $\al\H\be$.
Suppose the transversals of $\alpha$ are $A_i\cup B_i'$ ($i=1,\ldots,n$). 
Since $\al\H\be$, the transversals of $\beta$ are $A_i \cup B_{i\phi}'$ ($i=1,\ldots,n$),
where
$\phi$ is some permutation in $\S_n$.
It is straightforward to check that if we start from a different indexing of the transversals of $\alpha$, the resulting permutation will be conjugate to $\phi$ in $\S_n$; thus, $\phi(\al,\be)=\phi$ is well defined up to conjugation.
Since $N\unlhd \S_n$ it follows that there is a well-defined relation on~$D_n$ given by
\[
\nu_N=\bigset{ (\alpha,\beta)\in D_n\times D_n}{ \alpha\H\beta \text{ and } \phi(\alpha,\beta)\in N}.
\]
Additionally, for any cardinal $\zeta\in[1,|X|^+]$, we let
\[
\lambda_\zeta^N=\lambda_\zeta^n\cup\nu_N \AND
\rho_\zeta^N=\rho_\zeta^n\cup\nu_N.
\]
Note that for any $n\in[1,\aleph_0)$,  we have $\nu_{\{\id_n\}}=\Delta_{D_n}$, and hence
\[
\lambda_\zeta^{\{\id_n\}}=\lambda_\zeta^n \AND \rho_\zeta^{\{\id_n\}}=\rho_\zeta^n.
\]

Here is our main result, stated in terms of the relations defined above.

\begin{thm}
\label{thm-main}
Let $\M$ be either the partition monoid $\P_X$ or the partial Brauer monoid $\PB_X$, where $X$ is an arbitrary infinite set.  The congruences of $\M$ are precisely $\nabla_{\M}=\M\times\M$ (the universal congruence) and the following:
\begin{enumerate}[label=\textup{(CT\arabic*)},leftmargin=12mm]
\item \label{CT1} $\lambda_{\zeta_1}^N\cap \rho_{\zeta_2}^N$, where 
\bit
\item $N$ is a normal subgroup of $\S_n$ for some $n\in[1,\aleph_0)$, 
\item $\zeta_1,\zeta_2\in \{1\}\cup [\aleph_0,|X|^+]$ if $n\leq2$,  
\item $\zeta_1,\zeta_2\in [\aleph_0,|X|^+]$ if $n\geq3$,
\eit
\item \label{CT2} 
$(\lambda_{\zeta_1}^\eta\cap \rho_{\zeta_2}^\eta) \cup \mu_{\xi_1}^{\eta_1}\cup \cdots \cup \mu_{\xi_k}^{\eta_k}$, where 
\bit
\item 
$k\geq1$, $\eta\in [\aleph_0,|X|]$, $\ze_1,\ze_2,\eta_1,\ldots,\eta_k\in [\eta,|X|^+]$, $\xi_1,\ldots,\xi_k\in \{1\}\cup[\aleph_0,\eta]$, and 
\item 
$\xi_k<\dots<\xi_1\leq \eta<\eta_1<\dots<\eta_k=|X|^+$.
\eit
\end{enumerate}
\end{thm}

We shall refer to the two different groups \ref{CT1} and \ref{CT2} as \emph{types} of congruences.  Although the universal congruence $\nabla_{\M}$ is listed separately in Theorem \ref{thm-main}, we will think of  it as being of type \ref{CT2}, with $k=1$, $\eta=\zeta_1=\zeta_2=\eta_1=|X|^+$
and $\xi_1=1$, since
\[
\nabla_{\M}=\lam_{|X|^+}^{|X|^+}\cap\rho_{|X|^+}^{|X|^+}=\left(\lam_{|X|^+}^{|X|^+}\cap\rho_{|X|^+}^{|X|^+}\right)\cup\mu_1^{|X|^+}.
\]
Unlike for the other congruences of type \ref{CT2}, the above expression for
$\nabla_{\M}$ is not unique: indeed, we could let $\xi_1$ be any cardinal 
from $\{1\}\cup [\aleph_0,|X|^+]$.

Note that type \ref{CT1} deals with congruences of ``finite rank'': i.e., those for which there is a finite cardinal bounding the ranks of non-equal related pairs of partitions.
Type~\ref{CT2} contains all the congruences of infinite rank.

The proof of Theorem \ref{thm-main} occupies Sections \ref{sect:stage1}--\ref{sect:tech}, which form the bulk of this part of the paper. 

Before we outline the strategy of proof, it is worth ``locating'' some of the basic relations/congruences discussed above:
\bit
\item
The trivial congruence $\De_{\M}$ is of type \ref{CT1}, with $\ze_1=\ze_2=1$ and $N=\S_1$.
\item 
As noted above, we consider the universal congruence $\nabla_{\M}$ to be of type \ref{CT2}, with $k=1$, $\eta=\zeta_1=\zeta_2=\eta_1=|X|^+$ and $\xi_1=1$.
\item
If $n\in[1,\aleph_0)$, then the Rees congruence $R_n$ is of type \ref{CT1}, with $\ze_1=\ze_2=|X|^+$ and~${N=\{\id_n\}}$.
\item
If $\xi\in[\aleph_0,|X|^+]$, then $R_\xi$ is of type \ref{CT2}, with $\eta=\xi$, $k=1$, $\xi_1=1$ and ${\eta_1=\ze_1=\ze_2=|X|^+}$; 
this includes the universal congruence $\nabla_{\M}=R_{|X|^+}$.  
\item
We will see in Lemma \ref{la315} that $\mu_\xi$ is a congruence for any $\xi\in\{1\}\cup[\aleph_0,|X|^+]$.  Clearly $\mu_1=\De_{\M}$ (which is of type \ref{CT1}, as discussed above).  
If $\xi\in[\aleph_0,|X|]$, then $\mu_\xi$ is of type~\ref{CT2}, with $k=1$, $\ze_1=\ze_2=\eta=\xi_1=\xi$, $\eta_1=|X|^+$ (cf.~Lemma \ref{lem:mu}); finally, for~$\xi=|X|^+$ we have $\mu_{|X|^+}=\nabla_{\M}$.
\item
Similarly, we will see in Lemma \ref{la33} and Remark \ref{rem:lam_1^2} that $\lam_\ze^\eta$ is a congruence 
for $\ze=1$ and $\eta=1,2$, and for $\ze\in [\aleph_0,|X|^+]$ and $\eta\in[1,\ze]$.
  If $\eta<\aleph_0$, then $\lam_\ze^\eta$ is of type~\ref{CT1}, with $N=\{\id_\eta\}$, $\ze_1=\ze$ and $\ze_2=|X|^+$.  If $\eta\in[\aleph_0,|X|]$, then $\lam_\ze^\eta$ is of type~\ref{CT2}, with $\ze_1=\ze$, $\ze_2=|X|^+$, $k=1$, $\xi_1=1$ and $\eta_1=|X|^+$.  If $\eta=\ze=|X|^+$, then~${\lam_\ze^\eta=\nabla_{\M}}$.  Similar comments hold for the $\rho_\ze^\eta$ relations.
\eit

\subsection{Strategy of proof}\label{subsect:strategy}

The proof of Theorem \ref{thm-main} is broken up into two stages that are largely independent of each other, and which will be treated in Sections \ref{sect:stage1} and \ref{sect:stage2}, respectively:
\begin{description}
\item[Stage 1:]
Show that each relation listed in the theorem is indeed a congruence on $\M$.
\item[Stage 2:]
Show that any congruence on $\M$ is one of those listed in the theorem.
\end{description}
Considerations within Stage~1 naturally split into two strands: proving that the relations are equivalences (which in fact boils down to proving transitivity), and proving that they are compatible with multiplication.

The steps involved in Stage 2 are as follows: 
\begin{description}
\item[Stage 2.1:]
Given a congruence $\sigma$ on $\M$, identify its type.
\item[Stage 2.2:]
Describe how to find the relevant parameters for this type.
\item[Stage 2.3:]
Prove that the parameters fall within the prescribed ranges.
\item[Stage 2.4:]
Prove that $\sigma$ is indeed equal to the congruence from the list thus identified.
\end{description}
Most of the arguments in Sections \ref{sect:stage1} and \ref{sect:stage2} apply equally to $\M=\P_X$ or $\M=\PB_X$.  To make sure that the proof of a statement works for both monoids, we need to ensure that when the statement is interpreted in $\M=\PB_X$, any partition constructed during the proof belongs to~$\PB_X$ as well (and this might itself depend on the assumption that a partition appearing in the statement belongs to $\PB_X$).   A number of key lemmas used in Section \ref{sect:stage2} will require substantially different proofs for the two monoids, and we will postpone these proofs until Section \ref{sect:tech}.

\section{First stage of the proof: the stated relations are congruences}\label{sect:stage1}

We now embark on the first stage of the proof of Theorem \ref{thm-main}, namely the task of showing that the relations listed in the theorem are indeed congruences on $\M$, which throughout the entire section will stand for either of $\P_X$ or $\PB_X$ for a fixed infinite set $X$.  This will be achieved in Propositions~\ref{prop:CT1} and~\ref{prop:CT2} for type \ref{CT1}, and in Proposition \ref{prop:CT3} for type \ref{CT2}.

The section is structured as follows.  In Subsection \ref{subsect:EMRT} we recall some general machinery from \cite{EMRT2018} that allows for the construction of congruences in certain kinds of semigroups; we tie this in with~$\P_X$ and~$\PB_X$ in Subsection~\ref{subsect:stability}, and establish some useful inequalities in Subsection~\ref{subsect:inequalities}.  We then treat congruences of types \ref{CT1} and \ref{CT2} in Subsections \ref{subsect:CT1} and \ref{subsect:CT3}, respectively.

\subsection{General congruence constructions}\label{subsect:EMRT}

We begin with a review of some ideas from \cite{EMRT2018} that lead to the construction of several families of congruences on semigroups.  The results stated here are special cases of those in \cite{EMRT2018}, tailored to suit our purposes.

Throughout the following discussion, we fix a regular semigroup $S$ with a minimal ideal $M$.  Here,  regularity means that for every $x\in S$, we have $x=xax$ for some $a\in S$.
We also note that the minimal ideal, when it exists, is necessarily unique, and is also a $\J$-class. 

An ideal $I$ of $S$ is \emph{retractable} if there exists a homomorphism $f\colon I\to M$ such that $xf=x$ for all $x\in M$; such a map $f$ is called a \emph{retraction}.  If $I$ is retractable, then there is a unique such retraction \cite[Corollary 3.4]{EMRT2018}.  We say that a congruence $\si$ on $M$ is \emph{liftable} if $\De_S\cup\si$ is a congruence on $S$.  For any such congruence $\si$, and for any retractable ideal $I$, we define the relation
\[
R_{I,\si} = \De_S \cup \bigset{(x,y)\in I\times I}{(xf,yf)\in\si}.
\]
Note that when $\si=\nabla_M$ is the universal congruence on $M$, the relation $R_{I,\si}$ is equal to the Rees congruence $R_I=\De_S\cup(I\times I)$, as defined in Subsection \ref{subsect:S}.

A $\J$-class $J$ of $S$ is \emph{stable} if for all $x\in J$ and $a\in S$,
\[
xa\J x\implies xa\R x \AND ax\J x\implies ax\L x.
\]
Any stable $\J$-class is in fact a $\D$-class; see \cite[Lemma 3.10]{EMRT2018} or \cite[Proposition 2.3.9]{Lallement1979}.
Suppose now that~$J$ is a stable $\J$-class.  Let $G$ be a maximal subgroup of~$S$ contained in $J$ (so $G$ is the~$\H$-class of some idempotent of $J$).  For any normal subgroup $N\normal G$, we define the relation
\[
\vt_N = (J\times J) \cap \bigset{(axb,ayb)}{x,y\in N,\ a,b\in S^1}.
\]
(This relation was denoted $\nu_N$ in \cite{EMRT2018}, but we use the $\vt_N$ notation here to avoid any ambiguity with our previous use of $\nu_N$, until we establish in Lemma \ref{lem:nu_N} that the two are essentially equivalent for the monoids under consideration in this paper.)
It was shown in \cite[Lemma 3.15]{EMRT2018} that the relations $\vt_N$ are independent of the choice of maximal subgroup $G\sub J$: namely, if $G_1$ and~$G_2$ are maximal subgroups contained in $J$, and if $N_1\normal G_1$, then there exists $N_2\normal G_2$ such that~${\vt_{N_1}=\vt_{N_2}}$.

Recall that the set $S/{\J}$ of all $\J$-classes of $S$ has a natural partial order $\leq$; see Subsection~\ref{subsect:S}.  Any ideal $I$ of $S$ is a union of $\J$-classes; so too, therefore, is the complement $S\sm I$, and we may speak of $\J$-classes that are minimal in $(S\sm I)/{\J}$; such minimal $\J$-classes need not exist in general.  An \emph{IN-pair} in $S$ is a pair $(I,N)$, where $I$ is an ideal of $S$, and $N$ is a normal subgroup of a maximal subgroup contained in a stable $\J$-class that is minimal in $(S\sm I)/{\J}$.  
We say that an IN-pair $(I,N)$ is \emph{retractable} if $I$ is a retractable ideal, and if all the elements of $N$ act the same way on $M$: i.e., if $|xN|=|Nx|=1$ for all $x\in M$.  The next result is a special case of \cite[Proposition 3.22]{EMRT2018}:

\begin{lemma}\label{lem:EMRT}
Let $S$ be a regular semigroup with a stable minimal ideal $M$, and let $(I,N)$ be an IN-pair in $S$.
\begin{thmenum}
\item \label{it:EMRTi} The relation $R_I \cup \vt_N$ is a congruence on $S$.
\item \label{it:EMRTii} If $(I,N)$ is retractable, and if $\si$ is a liftable congruence on $M$, then the relation $R_{I,\si} \cup \vt_N$ is a congruence on $S$.  \epfres
\end{thmenum}
\end{lemma}

\subsection[Regularity, stability and (retractable) IN-pairs in $\P_X$ and $\PB_X$]{\boldmath Regularity, stability and (retractable) IN-pairs in $\P_X$ and $\PB_X$}\label{subsect:stability}

We now relate the notions introduced in Subsection \ref{subsect:EMRT} to the monoid $\M$, which we recall stands for either $\P_X$ or $\PB_X$.

First we note that $\M$ is regular.  Indeed, if $\al=\partABCD\in\M$, then with $\al^*=\partn{B_i}{D_k}{A_i}{C_j}$, we have $\al=\al\al^*\al$.  In fact, we also have $(\al^*)^*=\al$ and $(\al\be)^*=\be^*\al^*$, so that $\M$ is a so-called \emph{regular $*$-semigroup} in the sense of Nordahl and Scheiblich \cite{NS1978}.
This leads to a natural symmetry/duality that will be repeatedly invoked to shorten arguments.  

By Corollary \ref{cor:ideals_M} \ref{it:idMi}, $\M$ has a minimal ideal, namely
\[
I_1=D_0=\set{\al\in\M}{\rank(\al)=0}.
\]
For a partition $\alpha\in\M$, let $\widehat{\alpha}$ denote the unique partition of rank~$0$ with the same kernel and cokernel as $\alpha$.  In other words, if $\al=\partABCD$, then~$\wh\al=\partXXV{A_i}{C_j}{B_i}{D_k}$.
The mapping $\al\mt\wh\al$ will be used frequently throughout the paper, including to describe the retractable ideals of $\M$.

The proof of \cite[Lemma 5.2]{EMRT2018} works virtually unmodified to prove the following (but we do note a slight shift in notation: in \cite{EMRT2018}, $I_k$ was used to denote the set of all partitions of rank up to \emph{and including}~$k$):

\begin{lemma}
\label{la36}
The mapping $I_2\rightarrow I_1\colon \alpha\mapsto\widehat{\alpha}$ is a retraction.  \epfres
\end{lemma}

Thus, the ideal $I_2$ is retractable.  It turns out that no ideal larger than $I_2$ is retractable; indeed, this can be shown directly, but also follows from Theorem \ref{thm-main} (since if any larger ideal of $\M$ was retractable, this would yield additional congruences on $\M$).  We now identify the stable ${\D}={\J}$-classes of $\M$.

\begin{lemma}\label{la37}
If $n\in[0,\aleph_0)$, then $D_n$ is a stable $\J$-class of $\M$. 
\end{lemma}

\pf
Let $\al\in D_n$ and $\be\in \M$ be arbitrary, and write $\al=\partABCD$, noting that $|I|=n<\aleph_0$.  We must show that 
\[
\al\be\J\al \implies \al\be\R\al \AND \be\al\J\al \implies \be\al\L\al.
\]
We just prove the first assertion, as the second is dual.  
Suppose $\al\be\J\al$: i.e, $\al\be\in D_n$.  Since 
$\ker(\al\be)\supseteq\ker(\al)$, each $\ker(\al\be)$-class is a union of $\ker(\al)$-classes.  Now, each~$C_j$ is a $\ker(\al\be)$-class.  
As $\rank(\alpha\beta)=n$, the $n$ sets $A_i$ must be the upper parts of distinct transversals of $\alpha\beta$.
Hence $\dom(\alpha\beta)=\dom(\alpha)$ and $\ker(\alpha\beta)=\ker(\alpha)$:
i.e., $\alpha\beta\R \alpha$ by Lemma~\ref{lem:Green_M} \ref{it:GMiv}.
\epf

It turns out that $D_\xi$ is not stable if $\xi$ is infinite; again, this can be shown directly, but also follows from Theorem \ref{thm-main}.

Next we identify the IN-pairs in $\M$.  By definition, and by Lemma \ref{la37}, these include all pairs of the form $(I_n,N)$, where $n\in [1,\aleph_0)$, and $N$ is a normal subgroup of some group $\H$-class contained in~$D_n$.  
(Once again, it will follow from Theorem \ref{thm-main} that these are \emph{all} the IN-pairs, but we do not need to know this here.)  It will be convenient to fix a particular such group $\H$-class for each $n\in [1,\aleph_0)$.  

To this end, fix any countable subset of $X$, and without loss of generality assume it is $[1,\aleph_0)=\{1,2,\ldots\}\subseteq X$.  For each $n\in [1,\aleph_0)$, we write $\ep_n=\ep_{\{1,\ldots,n\}}$ (the~$\ep_Y$ notation was defined in Subsection \ref{subsect:M}).  For any permutation $\pi\in\S_n$, we write $\permdec\pi =\binom i{i\pi}_{1\leq i\leq n}
\in\M$,
and for any $\Si\subseteq\S_n$ write $\permdec\Si=\set{\permdec\pi}{\pi\in\Si}$.  
So the $\H$-class of $\ep_n$ is precisely the set $\permdec\S_n$.
For any normal subgroup $N\unlhd \S_n$, the set $\permdec N$ is a normal subgroup of $\permdec\S_n$, and $(I_n,\permdec N)$ is an IN-pair.

Clearly the IN-pair $(I_1,\permdec\S_1)=(I_1,\{\permdec\id_1\})$ is retractable.  Beyond this obvious one, we have two more retractable IN-pairs, as the next lemma demonstrates; the proof is essentially identical to that of \cite[Lemma~5.3]{EMRT2018}.

\begin{lemma}
\label{la39}
If $N$ is either of $\{\id_2\}$ or $\S_2$, then $(I_2,\permdec N)$ is a retractable IN-pair.  \epfres
\end{lemma}

Each IN-pair $(I_n,\permdec N)$ leads to a congruence on $\M$, as in Lemma \ref{lem:EMRT} \ref{it:EMRTi}, each involving the relation~$\vt_{\permdec N}$ defined in Subsection~\ref{subsect:EMRT}.  The next lemma shows that this relation~$\vt_{\permdec N}$ is precisely the relation~$\nu_N$ defined in Subsection~\ref{subsect:statement}; its proof is essentially identical to that of \cite[Lemma~5.6]{EMRT2018}.

\begin{lemma}\label{lem:nu_N}
For any $n\in[1,\aleph_0)$, and for any normal subgroup $N\normal\S_n$, we have $\vt_{\permdec N}=\nu_N$.  \epfres
\end{lemma}

Recall that for any cardinal $1\leq\xi\leq|X|^+$, we have the Rees congruence
\[
R_\xi = \De_{\M}\cup(I_\xi\times I_\xi).
\]

\begin{lemma}\label{lem:R_N}
For any $n\in[1,\aleph_0)$ and $N\normal\S_n$, the relation $R_N$ defined by 
$
R_N = R_n\cup\nu_N
$
is a congruence on $\M$.  \epfres
\end{lemma}

\subsection{Inequalities}\label{subsect:inequalities}

Before we move on, we establish a number of inequalities involving the symmetric difference.

\begin{lemma}
\label{la32}
For arbitrary partitions $\alpha,\beta,\th\in \P_X$ we have
\begin{thmenum}
\item
\label{it:32i}
$|\overline{\alpha\th}\sd\overline{\beta\th}|\leq |\overline{\alpha}\sd\overline{\beta}| + 2\rank(\alpha) + 2\rank(\beta)$,
\item
\label{it:32ii}
$|\overline{\th\alpha}\sd\overline{\th\beta}| \leq
|\overline{\alpha}\sd\overline{\beta}|$,
\item
\label{it:32iii}
$|\underline{\alpha\th}\sd\underline{\beta\th}|\leq 
|\underline{\alpha}\sd\underline{\beta}|$,
\item
\label{it:32iv}
$|\underline{\th\alpha}\sd\underline{\th\beta}| \leq
|\underline{\alpha}\sd\underline{\beta}| + 2\rank(\alpha) + 2\rank(\beta)$,
\item
\label{it:32v}
$|\alpha\th\sd\beta\th|\leq |\alpha\sd\beta|$,
\item
\label{it:32vi}
$|\th\alpha\sd\th\beta| \leq |\alpha\sd\beta|$,
\item 
\label{it:32vii}
$|\ol\al\sd\ol\be|\leq|\al\sd\be|$,
\item 
\label{it:32viii}
$|\ul\al\sd\ul\be|\leq|\al\sd\be|$,
\item 
\label{it:32ix}
$|\al\sd\be|\leq|\ol\al\sd\ol\be|+|\ul\al\sd\ul\be|+3\rank(\al)+3\rank(\be)$. 
\end{thmenum}
\end{lemma}

\begin{proof}
By duality, it is enough to prove \ref{it:32i}, \ref{it:32ii}, \ref{it:32v}, \ref{it:32vii} and \ref{it:32ix}.  We treat these roughly in order of difficulty.

\pfitem{\ref{it:32vii}}   Consider a block $A\in\ol\al\sm\ol\be$.  Then $A=B\cap X$ for some block $B$ of $\al$ (possibly $B=A$).  If~$B$ was a block of $\be$, then $A=B\cap X$ would be a block of $\overline{\be}$, a contradiction.  So $B\in\al\sm\be$.  This shows that $|\ol\al\sm\ol\be|\leq|\al\sm\be|$.  A symmetrical argument gives $|\ol\be\sm\ol\al|\leq|\be\sm\al|$.  Adding these two inequalities gives the claimed result.

\pfitem{\ref{it:32v}}  Consider a block $A$ from ${\al\th\sm\be\th}$.  The product graph $\Pi(\al,\th)$ contains a connected component~$B$ such that ${A=B\cap(X\cup X')}$.  Now, $B$ is the union of some collection of blocks of~$\al_\downarrow$ and blocks of~$\th^\uparrow$.  All of these blocks from $\th^\uparrow$ are present in $\Pi(\be,\th)$.  Thus, since $A$ is not a block of $\be\th$, at least one of the blocks of $\al_\downarrow$ contained in $B$ must not be present in $\Pi(\be,\th)$; this corresponds to a block from $\al\sm\be$.  This shows that $|\al\th\sm\be\th|\leq|\al\sm\be|$, and the proof concludes as in the previous part, by adding this to the symmetrical statement.

\pfitem{\ref{it:32i}}
As in the previous cases, it is enough prove that $|\ol{\al\th}\sm\ol{\be\th}|\leq|\ol\al\sm\ol\be|+\rank(\al)+\rank(\be)$.  
Now, each block in $\overline{\alpha\th}$ is a union of blocks of $\overline{\alpha}$.
The upper non-transversals of $\alpha$ remain upper non-transversals in $\alpha\th$ too.
For such a block to belong to $\overline{\alpha\th}\sm\overline{\beta\th}$,
it must already belong to $\overline{\alpha}\sm\overline{\beta}$ or else be the upper part of a transversal of $\be$; there are no more than $|\overline{\alpha}\sm\overline{\beta}|$ and $\rank(\be)$ such blocks, respectively.
Every other block in $\ol{\alpha\th}$ must contain the upper part of at least one transversal of $\alpha$, so there are no more than $\rank(\alpha)$ of them.

\pfitem{\ref{it:32ii}}
Here it is enough to show that $|\overline{\th\alpha}\sm\overline{\th\beta}| \leq
|\overline{\alpha}\sm\overline{\beta}|$. 
Now, every block of $\overline{\th\alpha}$ is a union of blocks of $\overline{\th}$.
The upper non-transversals of $\th$ remain upper non-transversals of both $\th\alpha$ and $\th\beta$, so do not belong to $\overline{\th\alpha}\sd\overline{\th\beta}$.
Any other block of $\ol{\th\alpha}$ has the form $Y=\bigcup_{i\in I} A_i$ for some collection of transversals
$\set{A_i\cup B_i'}{i\in I}$ of $\th$; in this case, there must also be some (possibly empty) collection $\set{C_j'}{j\in J}$ of lower non-transversals of $\th$ such that $\bigcup_{i\in I}B_i\cup\bigcup_{j\in J}C_j$ is a union of some collection of blocks $\set{D_k}{k\in K}$ of $\ol\al$.  For such a block $Y$ to belong to $\overline{\th\alpha}\sm\overline{\th\beta}$, at least one of the $D_k$ must not belong to $\ol\be$; thus, there are at most $|\ol\al\sm\ol\be|$ such blocks $Y$.

\pfitem{\ref{it:32ix}} Consider a block $A\cup B'\in\al\sm\be$, where $A$ or $B$ (but not both) might be empty.  There are at most $\rank(\al)$ such blocks with $A$ and $B$ both non-empty.  If $B$ is empty, then either $A\in\ol\al\sm\ol\be$ or else $\be$ has a transversal $A\cup C'$ with $C\not=\emptyset$; thus, there are at most $|\ol\al\sm\ol\be|+\rank(\be)$ such blocks with $B$ empty.  Similarly, there are at most $|\ul\al\sm\ul\be|+\rank(\be)$ such blocks with $A$ empty.  This all shows that $|\al\sm\be|\leq|\ol\al\sm\ol\be|+|\ul\al\sm\ul\be|+\rank(\al)+2\rank(\be)$; the statement now follows in the usual way.
\end{proof}

\subsection{Congruences of type \ref{CT1}}\label{subsect:CT1}

We now embark on proving that the relations listed in Theorem \ref{thm-main} are congruences, starting with those of type \ref{CT1}. 
We begin with three lemmas that will also be useful in subsequent sections.
In the next proof, and in many subsequent ones, we make use of the following simple observation:

\begin{lemma}
\label{lem:card}
If $\ze=1$ or $\ze\geq\aleph_0$, then any finite sum of cardinals strictly less than $\ze$ is again strictly less than $\ze$.
\epfres
\end{lemma}

\begin{lemma}
\label{la31}
If $\zeta\in \{1\}\cup[\aleph_0,|X|^+]$, then each of the relations $\lambda_\zeta$, $\rho_\zeta$, $\mu_\zeta$ is an equivalence.
\end{lemma}

\begin{proof}
We prove the statement for $\lambda_\zeta$; the proof for $\rho_\zeta$ is dual, and for $\mu_\zeta$ analogous.
It is clear that $\lambda_\zeta$ is reflexive and symmetric.
Transitivity follows from $\overline{\alpha}\sd\overline{\gamma}\subseteq (\overline{\alpha}\sd\overline{\beta}) \cup (\overline{\beta}\sd\overline{\gamma})$
and Lemma \ref{lem:card}.
\end{proof}

\begin{lemma}
\label{la33}
If $\ze\in\{1\}\cup[\aleph_0,|X|^+]$ and $\eta\in[1,|X|^+]$ are such that $\eta\leq\zeta$, then the relations $\lambda_\zeta^\eta$ and $\rho_\zeta^\eta$ are congruences.  
\end{lemma}

\pf
By duality, it suffices to prove the statement for $\lambda_\zeta^\eta$.  By Lemma \ref{la31}, $\lam_\ze$ is an equivalence; since $R_\eta$ is as well, so too is $\lambda_\zeta\cap R_\eta=\lambda_\zeta^\eta$. It remains to show that $\lam_\ze^\eta$ is compatible.
To do so, suppose $(\alpha,\beta)\in \lambda_\zeta^\eta$ and $\th\in\M$.
We need to prove that $(\alpha\th,\beta\th),(\th\alpha,\th\beta)\in\lambda_\zeta^\eta$.
If $\alpha=\beta$ this is obvious, so suppose $\alpha\neq\beta$.  
Since $(\alpha,\beta)\in\lam_\ze^\eta=\lam_\ze\cap R_\eta$ and $\al\not=\be$, it follows that $\al,\be\in I_\eta$ and $|\ol\al\sd\ol\be|<\zeta$.  
Since~$I_\eta$ is an ideal, we have $\al\th,\be\th,\th\al,\th\be\in I_\eta$.
By Lemma~\ref{la32}~\ref{it:32ii}, we have $|\overline{\th\alpha}\sd\overline{\th\beta}|\leq |\overline{\alpha}\sd\overline{\beta}|<\ze$, completing the proof that $(\th\alpha,\th\beta)\in\lambda_\zeta^\eta$.  
Since $\al,\be\in I_\eta$, we have $\rank(\al),\rank(\be)<\eta\leq\zeta$.
Using Lemmas \ref{la32} \ref{it:32i} and \ref{lem:card}, it follows that
$|\overline{\alpha\th}\sd \overline{\beta\th}| \leq |\overline{\alpha}\sd\overline{\beta}|+2\rank(\alpha)+2\rank(\beta) < \ze$, 
completing the proof that $(\alpha\th,\beta\th)\in\lambda_\zeta^\eta$.
\epf

We now have all the pieces needed to prove that all the relations of type \ref{CT1} are congruences. We split the considerations into two results, depending on whether $n\leq 2$ or $n>2$.

\begin{prop}
\label{prop:CT1}
If $N$ is any of $\S_1$, $\{\id_2\}$ or $\S_2$, and if $\zeta_1,\zeta_2\in\{1\}\cup [\aleph_0,|X|^+]$, then the relation~$\lambda_{\zeta_1}^N\cap \rho_{\zeta_2}^N$ is a congruence.
\end{prop}

\begin{proof}
By duality, and since the intersection of two congruences is a congruence, it suffices to show that $\lambda_\zeta^N$ is a congruence, where $\ze=\ze_1$.
The case where $N=\S_1$ follows from Lemma~\ref{la33}, as $\lam_\ze^{\S_1}=\lam_\ze^1$, so we will assume that $n=2$ and $N$ is $\{\id_2\}$ or $\S_2$; we will also write $\si=\lam_{\ze}^1\restr_{I_1}$.  Since $\De_{\M}\cup\si=\lam_{\ze}^1$ is a congruence by Lemma \ref{la33}, it follows that $\si$ is a liftable congruence on~$I_1$ (in the language of Subsection \ref{subsect:EMRT}).
By Lemmas \ref{la39} and \ref{lem:EMRT} \ref{it:EMRTii}, it follows that the relation $R_{I_2,\si}\cup\vt_{\permdec N}$ is a congruence on~$\M$. On the other hand, we have
\begin{align*}
R_{I_2,\sigma}\cup\vt_{\permdec N} 
&= \set{(\alpha,\beta)\in I_2\times I_2}{(\widehat{\alpha},\widehat{\beta})\in\sigma}\cup\Delta_{\M}\cup\nu_N
&& \text{by definition of  $R_{I_2,\sigma}$; Lemma \ref{lem:nu_N}}
\\
&= \set{(\alpha,\beta)\in I_2\times I_2}{|\overline{\alpha}\sd\overline{\beta}|<\zeta}\cup\Delta_{\M}\cup\nu_N
&& \text{as } \sigma=\lambda_{\zeta}^1\restr_{I_1}; \ \ol{\;\!\wh\ga\;\!}=\overline{\gamma} \text{ for all } \gamma 
\\
&=\lambda_{\zeta}^2\cup\nu_N=\lambda_{\zeta}^N,
\end{align*}
implying that $\lambda_{\zeta}^N$ is a congruence, as claimed.
\end{proof}

\begin{rem}\label{rem:lam_1^2}
Taking $N=\{\id_2\}$, $\ze_1=1$ and $\ze_2=|X|^+$, Proposition \ref{prop:CT1} tells us that $\lam_1^2$ is a congruence, a fact that does not follow from Lemma \ref{la33}.  A similar statement holds for $\rho_1^2$.
\end{rem}


\begin{prop}
\label{prop:CT2}
For all $n\in [3,\aleph_0)$, $N\unlhd \S_n$ and $\zeta_1,\zeta_2\in [\aleph_0,|X|^+]$, the relation $\lambda_{\zeta_1}^N\cap\rho_{\zeta_2}^N$ is a congruence.
\end{prop}

\pf
Again, it suffices to prove that $\lam_\ze^N$ is a congruence, where $\ze=\ze_1$.  For this, first note that $\lam_\ze^{n+1}\cap R_N = \lam_\ze^{n+1} \cap(R_n\cup\nu_N) = (\lam_\ze^{n+1}\cap R_n)\cup(\lam_\ze\cap\nu_N) = \lam_\ze^n\cup\nu_N = \lam_\ze^N$.  Since $R_N$ and~$\lam_\ze^{n+1}$ are congruences (by Lemmas \ref{lem:R_N} and \ref{la33}), $\lam_\ze^N$ is a congruence.
\epf

\subsection{Congruences of type \ref{CT2}}\label{subsect:CT3}

We now start working towards proving that the relations of type \ref{CT2}
are congruences.

\begin{lemma}
\label{la314}
The relation $\mu_\xi$ is compatible for any $\xi\in [1,|X|^+]$.
\end{lemma}

\begin{proof}
This follows directly from the definition of $\mu_\xi$ and the inequalities of Lemma \ref{la32} \ref{it:32v} and~\ref{it:32vi}.
\end{proof}

\begin{lemma}
\label{la315}
The relation $\mu_\xi$ is a congruence for any $\xi\in \{1\}\cup [\aleph_0,|X|^+]$.
\end{lemma}

\begin{proof}
This follows from Lemmas \ref{la31} and \ref{la314}.
\end{proof}

It follows from Lemma \ref{la315} that for any cardinals $\eta,\xi\in\{1\}\cup[\aleph_0,|X|^+]$ with $\xi<\eta$, the relation $\mu_\xi^\eta=\mu_\xi\cap R_\eta$ is a congruence.  At this point it will be convenient for later use to prove a simple lemma showing how such congruences may be expressed in the notation of Theorem \ref{thm-main}.

\begin{lemma}\label{lem:mu}
For any $\eta,\xi\in\{1\}\cup[\aleph_0,|X|^+]$ with $\xi<\eta$, we have $\mu_\xi^\eta=(\lam_\xi^\xi\cap\rho_\xi^\xi)\cup\mu_\xi^\eta\cup\mu_1^{|X|^+}$.
\end{lemma}

\pf
Clearly we only need to show that $\lam_\xi^\xi\cap\rho_\xi^\xi\sub\mu_\xi^\eta$.  To do so, let $(\al,\be)\in\lam_\xi^\xi\cap\rho_\xi^\xi$.  If $\al=\be$, then of course $(\al,\be)\in\mu_\xi^\eta$, so suppose $\al\not=\be$.  Then $\al,\be\in I_\xi$ and $|\ol\al\sd\ol\be|,|\ul\al\sd\ul\be|<\xi$.  By Lemmas \ref{la32} \ref{it:32ix} and \ref{lem:card} we have $|\al\sd\be| \leq |\ol\al\sd\ol\be| + |\ul\al\sd\ul\be| + 3\rank(\al) + 3\rank(\be) < \xi$, so that $(\al,\be)\in\mu_\xi^\xi\sub\mu_\xi^\eta$.
\epf

\begin{lemma}
\label{la316}
If $\xi\in\{1\}\cup[\aleph_0,|X|^+]$, $(\alpha,\beta)\in\mu_\xi$ and $\rank(\alpha)\geq\xi$, then $\rank(\alpha)=\rank(\beta)$.
\end{lemma}

\begin{proof}
The result is trivial for $\xi=1$, since $\mu_1=\De_{\M}$, so we assume that $\xi\geq\aleph_0$.  Write ${\ka=\rank(\al)}$, noting that $\ka\geq\xi\geq\aleph_0$.  So $\alpha$ has $\kappa$ transversals; since $(\al,\be)\in\mu_\xi\sub\mu_\ka$, strictly fewer than $\kappa$ of these are not transversals of $\beta$. It follows that some $\kappa$ transversals of $\alpha$ are also transversals of $\beta$, and hence $\rank(\beta)\geq \kappa=\rank(\al)$.  Since this also implies that $\rank(\be)\geq\xi$, we may repeat the same reasoning, with $\al$ and $\be$ swapped, to obtain $\rank(\al)\geq\rank(\be)$.
\end{proof}

\begin{lemma}
\label{la317}
If $\xi_2\leq\xi_1\leq\eta_1\leq\eta_2$, then $\mu_{\xi_2}^{\eta_2}\restr_{I_{\eta_1}}\subseteq\mu_{\xi_1}^{\eta_1}$.
\end{lemma}

\begin{proof}
If $(\alpha,\beta)\in \mu_{\xi_2}^{\eta_2}\restr_{I_{\eta_1}}$, then $\al,\be\in I_{\eta_1}$ and $|\alpha\sd\beta|<\xi_2\leq\xi_1$, and hence $(\alpha,\beta)\in \mu_{\xi_1}^{\eta_1}$.
\end{proof}

\begin{lemma}
\label{la318}
If $\zeta_1,\zeta_2\geq\eta$ and $\xi_1\leq\eta\leq\eta_1$, then
$\mu_{\xi_1}^{\eta_1}\restr_{I_\eta}\subseteq \lambda_{\zeta_1}^\eta\cap \rho_{\zeta_2}^\eta$.
\end{lemma}

\begin{proof}
Suppose $(\alpha,\beta)\in \mu_{\xi_1}^{\eta_1}\restr_{I_\eta}$, 
so that $\al,\be\in I_\eta$ 
and $|\alpha\sd\beta|<\xi_1$.  
Then, using Lemma~\ref{la32}~\ref{it:32vii}, we have $|\overline{\alpha}\sd\overline{\beta}|\leq |\alpha\sd\beta|<\xi_1\leq\eta\leq\zeta_1$, so that $(\al,\be)\in\lam_{\ze_1}^\eta$, and similarly $(\al,\be)\in\rho_{\ze_2}^\eta$.
\end{proof}

We are now ready to show that the relations of type \ref{CT2} are congruences:

\begin{prop}\label{prop:CT3}
If
\bit
\item 
$k\geq1$, $\eta\in [\aleph_0,|X|]$, $\ze_1,\ze_2,\eta_1,\ldots,\eta_k\in [\eta,|X|^+]$, $\xi_1,\ldots,\xi_k\in \{1\}\cup[\aleph_0,\eta]$, and 
\item 
$\xi_k<\dots<\xi_1\leq \eta<\eta_1<\dots<\eta_k=|X|^+$,
\eit
then the relation $(\lambda_{\zeta_1}^\eta\cap \rho_{\zeta_2}^\eta) \cup \mu_{\xi_1}^{\eta_1}\cup \cdots \cup \mu_{\xi_k}^{\eta_k}$ is a congruence.
\end{prop}

\pf
Denote the relation in question by $\tau$.
By Lemma \ref{la33}, $\lambda_{\zeta_1}^\eta\cap \rho_{\zeta_2}^\eta$ is a congruence; so too is each  $\mu_{\xi_i}^{\eta_i}=\mu_{\xi_i}\cap R_{\eta_i}$, by Lemma \ref{la315}.  Thus $\tau$ is a union of congruences, and therefore is symmetric, reflexive and compatible.


It remains to prove transitivity of $\tau$.  To do so, suppose $(\alpha,\beta),(\beta,\gamma)\in\tau$.  If $\alpha=\beta$ or $\beta=\gamma$, then clearly $(\alpha,\gamma)\in \tau$, so we may assume that $\alpha\neq\beta$ and $\beta\neq\gamma$.
Now, each of the relations $\lambda_{\zeta_1}^\eta\cap\lambda_{\zeta_2}^\eta,\mu_{\xi_1}^{\eta_1},\dots,\mu_{\xi_k}^{\eta_k}$ is an equivalence (as noted above); thus, if both $(\alpha,\beta)$ and $(\beta,\gamma)$ belong to the same one of these relations, then so too does $(\al,\ga)$, completing the proof in this case.
Up to symmetry, the remaining cases to consider are: 
\begin{thmenum}
\item \label{it:tci} $(\alpha,\beta)\in\lambda_{\zeta_1}^\eta\cap\rho_{\zeta_2}^\eta$ and $(\beta,\gamma)\in\mu_{\xi_i}^{\eta_i}$, for some $1\leq i\leq k$, and 
\item \label{it:tcii} $(\alpha,\beta)\in \mu_{\xi_i}^{\eta_i}$ and $(\beta,\gamma)\in\mu_{\xi_j}^{\eta_j}$, for some $1\leq i<j\leq k$.
\end{thmenum}
We consider these separately.  In both cases, recall that $\al\not=\be$ and $\be\not=\ga$.

\pfitem{\ref{it:tci}}  Here we have $\al,\be\in I_\eta$, $|\ol\al\sd\ol\be|<\ze_1$, $|\ul\al\sd\ul\be|<\ze_2$, $\be,\ga\in I_{\eta_i}$ and $|\be\sd\ga|<\xi_i$.  If we had $\rank(\ga)\geq\eta$, then we would also have $\rank(\ga)\geq\xi_i$; since $(\be,\ga)\in\mu_{\xi_i}$, Lemma \ref{la316} would then give $\rank(\be)=\rank(\ga)\geq\eta$, contradicting $\be\in I_\eta$.  So we must in fact have $\rank(\ga)<\eta$: i.e., $\ga\in I_\eta$.  But then $(\be,\ga)\in\mu_{\xi_i}^{\eta_i}\restr_{I_\eta}$, and so Lemma \ref{la318} gives $(\be,\ga)\in\lam_{\ze_1}^\eta\cap\rho_{\ze_2}^\eta$.  It now follows that $(\alpha,\gamma)\in\tau$ by transitivity of $\lambda_{\zeta_1}^\eta\cap\rho_{\zeta_2}^\eta$.

\pfitem{\ref{it:tcii}}  Here we have $\al,\be\in I_{\eta_i}$, $|\al\sd\be|<\xi_i$, $\be,\ga\in I_{\eta_j}$ and $|\be\sd\ga|<\xi_j$.  If we had ${\rank(\ga)\geq\eta_i}$, then we would also have $\rank(\ga)\geq\xi_j$; since $(\be,\ga)\in\mu_{\xi_j}$, Lemma \ref{la316} would then give ${\rank(\be)=\rank(\ga)\geq\eta_i}$, contradicting $\be\in I_{\eta_i}$.  So it follows that $\rank(\ga)<\eta_i$, and so $(\be,\ga)\in\mu_{\xi_j}^{\eta_j}\restr_{I_{\eta_i}}$; Lemma~\ref{la317} then gives $(\be,\ga)\in\mu_{\xi_i}^{\eta_i}$.  Thus, $(\alpha,\gamma)\in\tau$ by transitivity of~$\mu_{\xi_i}^{\eta_i}$.
\epf

\section{Second stage of the proof: any congruence has one of the stated forms}\label{sect:stage2}

We now move on to the second stage of our proof of Theorem \ref{thm-main}, which involves showing that any congruence on $\M$ (again standing for $\P_X$ or $\PB_X$ with $X$ infinite) is of one of the forms listed in the theorem.

Throughout this section, $\si$ denotes an arbitrary congruence on~$\M$.  As outlined in Subsection~\ref{subsect:strategy}, we will proceed by first identifying a number of parameters ($\eta$, $\ze_1$, $\ze_2$, etc.)~associated to~$\si$, then showing that the permissible values of these parameters are as stated in Theorem~\ref{thm-main}, and finally showing that $\sigma$ is equal to the congruence from the theorem thus singled out.  The main results of this section are summarised in Propositions \ref{prop:CT1/2} and \ref{prop:CT3b}.

Before we begin, we introduce a piece of notation relating to an arbitrary set of cardinals~$\Xi$.  It is well known that $\Xi$, being a set, has a strict upper bound: e.g., $\sum_{\xi\in\Xi}\xi^+$.  Since the cardinals are well-ordered, there exists a least such bound;  we call it the \emph{least strict upper bound} of $\Xi$, and denote it by
\[
\LSUB( \Xi) = \min\set{\ka}{\xi<\ka \text{ for all } \xi\in\Xi}.
\]

\subsection[The parameter $\eta$]{\boldmath The parameter $\eta$}\label{subsect:eta}

We begin with the observation that the congruence $\si$ might identify partitions of unequal ranks.  That is, there may exist some $(\al,\be)\in\si$ with $\rank(\al)>\rank(\be)$; if we write $\ka=\rank(\al)$, then this says that $(\al,\be)\in\si\cap(D_\ka\times I_\ka)$.  Roughly speaking, our first parameter, $\eta(\si)$, measures how high up (in the ordering of $\J$-classes of $\M$) this phenomenon occurs.
Specifically, we define
\[
\eta=\eta(\si) = {\LSUB} \bigset{\ka}{\si\cap(D_\ka\times I_\ka)\not=\emptyset}.
\]
Note for example that $\eta(\De_{\M})=0$ and $\eta(\nabla_{\M})=|X|^+$.
More generally, for a Rees congruence~$R_\kappa$ with $\ka\geq2$ we have $\eta(R_\kappa)=\kappa$.
We begin with a simple lemma (in which for convenience we additionally define~$I_0=\emptyset$ to cover the $\eta(\si)=0$ case):

\begin{lemma}
\label{la323a}
With $\eta=\eta(\si)$, we have $\sigma=\sigma\restr_{I_\eta}\cup  \displaystyle{\bigcup_{\kappa\in[\eta,|X|]} \sigma\restr_{D_\kappa}} $.  
\end{lemma}

\pf
Clearly only the forward inclusion requires a proof, so suppose $(\al,\be)\in\si$.  By symmetry, we may assume that $\ka=\rank(\al)\geq\rank(\be)$.  If $\ka<\eta$, then $(\al,\be)\in\si\restr_{I_\eta}$.  If $\ka\geq\eta$, then we must have $\rank(\be)\geq\ka$ by definition of $\eta$, and so $\rank(\be)=\ka$, giving $(\alpha,\beta)\in\sigma\restr_{D_\kappa}$, as required.
\epf

Because the set $\bigset{\ka}{\si\cap(D_\ka\times I_\ka)\not=\emptyset}$ never contains $0$, the next lemma follows immediately from the definition of $\eta(\si)$:

\begin{lemma}
\label{la323}
We have $\eta(\sigma)\in \{0\}\cup [2,|X|^+]$.  \epfres
\end{lemma}

The parameter $\eta(\sigma)$ is the main classifying parameter in our theorem: 
the congruences of type~\ref{CT1} are those with $\eta(\sigma)$ finite,
while \ref{CT2} consists of all the congruences for which $\eta(\sigma)$ is infinite.

The remainder of this subsection is devoted to establishing a key property of $\eta$: namely, that~$(\alpha,\widehat{\alpha})\in \sigma$ for all $\alpha$ of rank smaller than~$\eta$ (the $\al\mt\wh\al$ map was defined in Subsection \ref{subsect:stability}). 
This will be achieved in Lemma \ref{la325}, the proof of which requires several intermediate lemmas.

\begin{lemma}
\label{la-ty1}
If $(\alpha,\beta)\in\sigma$, then for every $\gamma\in \M$ with ${\rank(\gamma)\leq\rank(\alpha)}$, there exists $\delta\in \M$ with $\rank(\delta)\leq\rank(\beta)$ such that $(\gamma,\delta)\in\sigma$.   
\end{lemma}

\begin{proof}
From $\rank(\ga)\leq\rank(\al)$, we have $\ga\leqJ\al$ by Lemma \ref{lem:Green_M} \ref{it:GMiii}, and hence $\gamma=\theta_1\alpha\theta_2$ for some $\theta_1,\theta_2\in\M$.
Setting $\delta=\theta_1\beta\theta_2$, we have $(\gamma,\delta)=(\th_1\al\th_2,\th_1\be\th_2)\in\sigma$ because $\sigma$ is a congruence.  Since $\de\leqJ\be$, another application of Lemma \ref{lem:Green_M} \ref{it:GMiii} gives $\rank(\de)\leq\rank(\be)$.
\end{proof}

\begin{lemma}
\label{la-ty2}
If $(\epsilon_Y,\alpha)\in\sigma$ where $Y\subseteq X$ is finite and $\rank(\alpha)<|Y|$,
then there exists $\alpha_0\in D_0$ such that $(\epsilon_Y,\alpha_0)\in \sigma$.
\end{lemma}

\begin{proof}
Suppose $Y=\{a_1,\dots,a_k\}$, and let $\alpha_0$ be a partition of the smallest possible rank such that $(\epsilon_Y,\alpha_0)\in\sigma$, noting that $\rank(\al_0)\leq\rank(\al)<|Y|$.
We must prove that $\rank(\alpha_0)=0$.
To do so, suppose to the contrary that $\rank(\alpha_0)=l>0$, and let the transversals of $\alpha_0$ be $\set{A_i\cup B_i'}{ i=1,\dots,l}$.
Since $|Y|>\rank(\al_0)=l$, either some element of~$Y$ does not belong to $A_1\cup\cdots\cup A_l$, or else there exists some $A_i$ that contains two distinct elements of~$Y$.  In any case, there exist $l$ distinct elements of $Y$, say $a_1,\dots,a_l$, that do not all belong to distinct~$A_i$.  Let $Z=\{a_1,\dots,a_l\}$.  Then $(\epsilon_Z, \epsilon_Z\alpha_0)= (\epsilon_Z\epsilon_Y,\epsilon_Z\alpha_0)\in \sigma$, and we have $\rank(\ep_Z)=l>\rank(\epsilon_Z\alpha_0)$.  By Lemma \ref{la-ty1}, since ${\rank(\al_0)=l=\rank(\ep_Z)}$, we have $(\alpha_0,\alpha_1)\in \sigma$ for some~$\alpha_1$ with ${\rank(\alpha_1)\leq\rank(\ep_Z\al_0)< l=\rank(\alpha_0)}$. By transitivity we also have ${(\epsilon_Y,\alpha_1)\in\sigma}$, contradicting the minimality of $\rank(\alpha_0)$, and completing the proof.
\end{proof}

\begin{lemma}
\label{la-ty3}
If $(\epsilon_Y,\alpha)\in\sigma$ where $Y\sub X$ is infinite and $\rank(\alpha)<|Y|=|Y\setminus\dom(\alpha)|$, then there exists $\alpha_0\in D_0$ such that $(\epsilon_Y,\alpha_0)\in\sigma$.
\end{lemma}

\begin{proof}
Let $Z=Y\setminus\dom(\alpha)$.  Then $(\epsilon_Z,\epsilon_Z\alpha)=(\epsilon_Z\epsilon_Y,\epsilon_Z\alpha)\in\sigma$.  Since $\rank(\ep_Y)=\rank(\ep_Z)$, Lemma \ref{la-ty1} says that $(\ep_Y,\al_0)\in\si$ for some $\al_0\in\M$ with $\rank(\al_0)\leq\rank(\ep_Z\al)=0$.
\end{proof}

\begin{lemma}
\label{la-ty4}
If $(\epsilon_Y,\alpha)\in\sigma$ where $Y\subseteq X$ and $\rank(\alpha)<|Y|$, then there exists $\alpha_0\in D_0$ such that
$(\epsilon_Y,\alpha_0)\in\sigma$.  
\end{lemma}

\begin{proof}
This follows from Lemma \ref{la-ty2} if $Y$ is finite, or from Lemma~\ref{la-ty3} if $Y$ is infinite and $|Y\setminus\dom(\alpha)|=|Y|$. As these are the only possibilities for $\M=\PB_X$, the lemma is proved for this monoid.  So for remainder of the proof we assume that $\M=\P_X$, that $Y$ is infinite, and that $|Y\setminus\dom(\alpha)|<|Y|$.  We may also assume that $\dom(\al)\sub Y$; indeed, if this were not the case, then we could replace $\al$ with $\al_1=\ep_Y\al$, noting that $(\ep_Y,\al_1)=(\ep_Y\ep_Y,\ep_Y\al)\in\si$, ${\rank(\al_1)\leq\rank(\al)<|Y|}$, $\dom(\al_1)\sub Y$ and $|Y\sm\dom(\al_1)|=|Y\sm\dom(\al)|<|Y|$.  Since $Y$ is infinite, the assumptions $|Y\setminus\dom(\alpha)|<|Y|$ and $\dom(\al)\sub Y$ together imply $|{\dom(\al)}|=|Y|$.

Suppose the transversals of $\alpha$ are $\set{A_i\cup B_i'}{i\in I}$, noting that $|I|<|Y|$.  Pick arbitrary $a_i\in A_i$ for each $i\in I$, let $Z_1=\set{a_i}{i\in I}$, and put $Z_2=\dom(\alpha)\setminus Z_1$.  Since $|Z_1|=|I|$ and $|{\dom(\alpha)}|=|Y|>|I|$, it follows that $|Z_2|=|Y|$.  Let $\theta=\partn z{}z{Z_1}_{z\in Z_2}$, and put $\al_2=\th\al$.  Then $(\th,\al_2)=(\th\ep_Y,\th\al)\in\sigma$, and we note that $\rank(\th)=|Z_2|=|Y|$ and $\rank(\al_2)=1$, with the single transversal of $\al_2$ being $Z_2\cup\codom(\al)'$.  Now pick any $u\in\codom(\alpha)$ and any $v\in Z_2\sm\{u\}$, and note that $(\binom vu,\emptypart)=(\binom vv \al_2\binom uu,\binom vv\th\binom uu)\in\si$.
Since $\rank(\al_2)=1=\rank(\binom vu)$, Lemma~\ref{la-ty1} says that there exists $\alpha_3\in D_0$ such that $(\al_2,\al_3)\in\sigma$, and then $(\th,\alpha_3)\in\sigma$ by transitivity.  Since $\rank(\epsilon_Y)=\rank(\th)$, another application of Lemma~\ref{la-ty1} shows that there exists $\alpha_0\in D_0$ with $(\epsilon_Y,\alpha_0)\in\sigma$, as required.
\end{proof}

\begin{lemma}
\label{la324}
If $(\alpha,\beta)\in\sigma$ with $\rank(\alpha)>\rank(\beta)$, then
$(\gamma,\widehat{\gamma})\in\sigma$ for all $\gamma\in\M$ with $\rank(\gamma)\leq\rank(\alpha)$.
\end{lemma}

\begin{proof}
Let $Y\subseteq X$ be any subset of cardinality $\rank(\alpha)$.  By Lemma \ref{la-ty1} we have $(\epsilon_Y,\de)\in\si$ for some $\de\in\M$ with  $\rank(\de)\leq\rank(\beta)< |Y|$.
By Lemma \ref{la-ty4} it follows that there exists $\de_0\in D_0$ such that
$(\epsilon_Y,\de_0)\in\sigma$.
Now let $\ga\in\M$ with $\rank(\ga)\leq\rank(\al)=|Y|$.
Using Lemma~\ref{la-ty1} again, we have $(\gamma,\gamma_0)\in\sigma$ for some $\gamma_0\in D_0$.  But then $(\ga_0,\widehat{\ga})=(\ga_0\emptypart\ga_0,\ga\emptypart\ga)\in\sigma$, and hence $(\ga,\widehat{\ga})\in\sigma$ by transitivity.
\end{proof}

\begin{lemma}
\label{la325}
For every $\alpha\in\M$ with $\rank(\alpha)<\eta(\sigma)$, we have $(\alpha,\widehat{\alpha})\in\sigma$.
\end{lemma}

\begin{proof}
Suppose $\rank(\alpha)=\kappa<\eta(\si)$.  By definition of $\eta(\si)$, there exists a pair $(\beta,\gamma)\in\sigma$ with
$\rank(\beta)\geq\kappa$ and $\rank(\beta)>\rank(\gamma)$.
It now follows from Lemma \ref{la324} that $(\alpha,\widehat{\alpha})\in\sigma$.
\end{proof}

One consequence of Lemma \ref{la325} is that the set $\set{\ka}{\si\cap(D_\ka\times I_\ka)\not=\emptyset}$ used above to define~${\eta=\eta(\si)}$ is in fact the entire interval $[1,\eta)$.

\subsection[The parameters $\zeta_1$ and $\zeta_2$]{\boldmath The parameters $\zeta_1$ and $\zeta_2$}\label{subsect:zeta}

The other two parameters that apply to an arbitrary congruence $\sigma$ are denoted $\ze_1$ and $\ze_2$.  Roughly speaking, they measure by how much the kernels and cokernels of $\sigma$-related pairs in~$D_0$ can differ.  Formally, we define
\begin{align*}
\zeta_1 = \zeta_1(\sigma) &= {\LSUB} \bigset{|\ol{\alpha}\sd\ol{\beta}|}{(\al,\be)\in\si\restr_{D_0}} , \\[2truemm]
\zeta_2 = \zeta_2(\sigma) &= {\LSUB} \bigset{|\ul{\alpha}\sd\ul{\beta}|}{(\al,\be)\in\si\restr_{D_0}}.
\end{align*}

Again, we proceed to gather some important facts about $\zeta_1$ and $\zeta_2$.  The main result here is that these parameters tell us everything about the way $\si$ identifies partitions of rank below $\eta=\eta(\si)$; see Lemma \ref{la327}.  We also identify some restrictions on the possible values of $\ze_1$ and~$\ze_2$; see Lemmas \ref{la328} and \ref{la328b}.
We begin with the following obvious observation:

\begin{lemma}\label{lem:D0_al_be}
If $\al,\be\in D_0$, then $\ol{\al\be}=\ol\al$ and $\ul{\al\be}=\ul\be$.  \epfres
\end{lemma}

We will also require the following two technical lemmas; the proofs diverge significantly for~$\P_X$ and $\PB_X$, and will be postponed until Section \ref{sect:tech}.  There are obvious dual versions, but we will not state these.

\begin{lemma}
\label{la328a}
Suppose $(\alpha,\beta)\in \sigma\restr_{D_0}$ with $\overline{\alpha}\neq\overline{\beta}$.
Then for any $\gamma,\delta\in D_0$ with $|\overline{\gamma}\sd\overline{\delta}|<\aleph_0$ and $\underline{\gamma}=\underline{\delta}$, we have $(\gamma,\delta)\in\sigma$.
\end{lemma}

\begin{proof}
This is proved in Lemma \ref{la-tt5} \ref{it:tt5ii} for $\PB_X$, and in Lemma \ref{la-tt1} \ref{it:tt1v} for $\P_X$.
\end{proof}

\begin{lemma}
\label{la328c}
Suppose $(\alpha,\beta)\in \sigma\restr_{D_0}$ with $|\ol\al\sd\ol\be|\geq\aleph_0$.
Then for any $\gamma,\delta\in D_0$ with $|\overline{\gamma}\sd\overline{\delta}|\leq|\ol\al\sd\ol\be|$ and $\underline{\gamma}=\underline{\delta}$, we have $(\gamma,\delta)\in\sigma$.
\end{lemma}

\begin{proof}
This is proved in Lemma \ref{la-tt6} \ref{it:tt6ii} for $\PB_X$, and in Lemma~\ref{la-tt2} \ref{it:tt2v} for $\P_X$.
\end{proof}

\newpage

\begin{lemma}
\label{la326a}
Suppose $\alpha,\beta,\gamma,\delta\in D_0$ and  $(\alpha,\beta)\in\sigma$. 
\begin{thmenum}
\item
\label{it:326i}
If $|\overline{\gamma}\sd\overline{\delta}|\leq |\overline{\alpha}\sd\overline{\beta}|$ and
$\underline{\gamma}=\underline{\delta}$,
then $(\gamma,\delta)\in\sigma$.
\item
\label{it:326ii}
If
$|\overline{\gamma}\sd\overline{\delta}|\leq |\overline{\alpha}\sd\overline{\beta}|$ and
$|\underline{\gamma}\sd\underline{\delta}|\leq |\underline{\alpha}\sd\underline{\beta}|$,
then $(\gamma,\delta)\in\sigma$.
\end{thmenum}
\end{lemma}

\begin{proof}
\ref{it:326i} This follows from Lemmas \ref{la328a} and \ref{la328c} for finite and infinite $|\ol\al\sd\ol\be|$, respectively.

\pfitem{\ref{it:326ii}}
Fix any $\th\in D_0$.
Then, using Lemma \ref{lem:D0_al_be}, we have
$
|\ol{\ga\th}\sd\ol{\de\th}| = |\ol\ga\sd\ol\de| \leq |\ol\al\sd\ol\be| $ and $ \ul{\ga\th}=\ul\th=\ul{\de\th}
$.
Thus, $(\ga\th,\de\th)\in\si$ by part \ref{it:326i}.  By duality we have $(\th\ga,\th\de)\in\si$; together, these then give
$
(\ga,\de) = \big( (\ga\th)(\th\ga), (\de\th)(\th\de) \big) \in\si
$,
as required.
\end{proof}

The parameters $\zeta_1$ and $\zeta_2$ completely determine the restriction of $\sigma$ to the bottom $\D$-class~${D_0=I_1}$, as we now show (recall that $R_\xi$ is the Rees congruence associated to the ideal~$I_\xi$):

\begin{lemma}
\label{la326}
We have $\sigma\cap R_1=\lambda_{\zeta_1}^1\cap\rho_{\zeta_2}^1$. 
\end{lemma}

\begin{proof}
The forward inclusion follows directly from the definition of $\zeta_1$ and $\zeta_2$.  For the reverse, suppose $(\al,\be)\in\lam_{\ze_1}^1\cap\rho_{\ze_2}^1$.  Clearly $(\al,\be)\in\si\cap R_1$ if $\al=\be$, so suppose $\al\not=\be$, noting that then $\al,\be\in D_0$.  Put $\ka_1=|\ol\al\sd\ol\be|$ and $\ka_2=|\ul\al\sd\ul\be|$.  
Since $\ka_1<\ze_1$ and $\ka_2<\ze_2$, there exists $(\ga_1,\de_1),(\ga_2,\de_2)\in\si\restr_{D_0}$ such that $|\ol\ga_1\sd\ol\de_1|\geq\ka_1$ and $|\ul\ga_2\sd\ul\de_2|\geq\ka_2$.  
Put $\ga=\ga_1\ga_2$ and $\de=\de_1\de_2$.  So $(\ga,\de)\in\si$ and, using Lemma \ref{lem:D0_al_be}, we have
\[
|\ol\al\sd\ol\be| = \ka_1 \leq |\ol\ga_1\sd\ol\de_1| = |\ol{\ga_1\ga_2}\sd\ol{\de_1\de_2}| = |\ol\ga\sd\ol\de|
\ANDSIM |\ul\al\sd\ul\be| \leq |\ul\ga\sd\ul\de|.
\]
It then follows from Lemma \ref{la326a} \ref{it:326ii} that $(\al,\be)\in\si$.
\end{proof}

In fact, $\zeta_1$ and $\zeta_2$ completely determine the behaviour of $\sigma$ on the entire ideal $I_\eta$ (for the next statement, recall that we define $I_0=\emptyset$, so that $R_0=\De_{\M}=\lam_\ze^0=\rho_\ze^0$ for any $\ze$):

\begin{lemma}
\label{la327}
We have $\si\cap R_\eta = \lambda_{\zeta_1}^\eta\cap\rho_{\zeta_2}^\eta$.
\end{lemma}

\begin{proof}
Since neither $\si\cap R_\eta$ nor $\lambda_{\zeta_1}^\eta\cap\rho_{\zeta_2}^\eta=(\lam_{\ze_1}\cap\rho_{\ze_2})\cap R_\eta$ identify any partition of rank $\geq\eta$ with any other distinct partition, we may prove the lemma by showing that
\[
(\al,\be)\in\si \iff (\al,\be)\in\lambda_{\zeta_1}\cap\rho_{\zeta_2} \qquad\text{for all $\al,\be\in I_\eta$.}
\]
But for any such $\al,\be$, we have
\begin{align*}
(\al,\be)\in\si &\iff (\wh\al,\wh\be)\in\si  &&\text{since $(\alpha,\widehat{\alpha}),(\beta,\widehat{\beta})\in\sigma$, by Lemma \ref{la325}}\\
&\iff (\wh\al,\wh\be)\in\si\cap R_1 &&\text{since $\wh\al,\wh\be\in D_0=I_1$}\\
&\iff (\wh\al,\wh\be)\in\lambda_{\zeta_1}^1\cap\rho_{\zeta_2}^1 &&\text{by Lemma \ref{la326}}\\
&\iff (\al,\be)\in\lambda_{\zeta_1}\cap\rho_{\zeta_2} &&\text{since $\ol{\;\!\wh\ga\;\!} = \ol{\ga}$ and $\ul{\wh\ga} = \ul{\ga}$ for any $\ga\in\M$.}
\qedhere
\end{align*}
\end{proof}

We conclude with two lemmas discussing the possible values of $\zeta_1$ and $\zeta_2$.

\begin{lemma}
\label{la328}
The only possible finite value for the parameters $\zeta_1$ and $\zeta_2$ is $1$.
\end{lemma}

\begin{proof}
We prove the assertion for $\zeta_1$, as the one for $\zeta_2$ is dual.  That $\zeta_1\neq 0$ follows straight from the definition, since $\si$ is reflexive.  Suppose now that $\zeta_1$ is finite and greater than $1$.  This means that there exist $\alpha,\beta\in D_0$ such that $(\al,\be)\in\si$ and $|\overline{\alpha}\sd\overline{\beta}|= \zeta_1-1$; in particular $\overline{\alpha}\neq\overline{\beta}$.  Since~$X$ is infinite, we can pick $\gamma,\delta\in D_0$ such that $\zeta_1\leq |\overline{\gamma}\sd\overline{\delta}| <\aleph_0$ and $\underline{\gamma}=\underline{\delta}$.  But then $(\gamma,\delta)\in\sigma $ by Lemma \ref{la328a}, which contradicts the definition of $\zeta_1$.
\end{proof}

\begin{lemma}
\label{la328b}
If $\eta=\eta(\sigma)>2$, then $\zeta_1,\zeta_2\geq \eta$.
\end{lemma}

\begin{proof}
We just prove $\zeta_1\geq \eta$, as $\ze_2\geq\eta$ is dual.  Suppose, aiming for contradiction, that $\zeta_1<\eta$.

First let $x,y,z\in X$ be three distinct elements, and define $\al = \partXI xyxy$, $\be=\partXI yzyz$ and~${\th = \partXX {x,y}{}}$.
Since $\widehat{\alpha}=\widehat{\beta}= \emptypart$ and $\eta>2=\rank(\alpha)=\rank(\beta)$, Lemma \ref{la325} implies that $(\alpha,\beta)\in\sigma$.
It then follows that $(\th,\emptypart)=(\alpha\theta,\beta\theta)\in \sigma$.
Since $\th,\emptypart\in D_0$ and $\ol\th\not=\ol\emptypart$, it follows that $\zeta_1\neq 1$, and so by Lemma \ref{la328}, $\ze_1\geq\aleph_0$.

Now pick pairwise distinct elements $a_i,b_i\in X$ ($i\in I$) where $|I|=\zeta_1$.
Let $\ga=\partXI{a_i}{b_i}{a_i}{b_i}_{i\in I}$ and $\de=\partXX{a_i,b_i}{}_{i\in I}$.
Since $\rank(\ga)=2|I|=\zeta_1<\eta$, we have $(\ga,\widehat{\ga})\in \sigma$ by
Lemma \ref{la325}.
Clearly $\widehat{\ga}=\emptypart$, and hence $(\ga,\emptypart)\in\sigma$,
implying $(\de,\emptypart)=(\ga\de,\emptypart\de)\in\sigma$.
But $|\overline{\de}\sd\overline{\emptypart}|=3|I|=\zeta_1$,
which contradicts the definition of $\zeta_1$.
\end{proof}

Clearly Lemma \ref{la328b} holds for $\eta=0$ as well, but it does not for $\eta=2$; consider $\lam_1^2$.

\subsection[Two technical lemmas concerning $\sigma$-related elements of unequal ranks]{\boldmath Two technical lemmas concerning $\sigma$-related elements of unequal ranks}\label{subsect:unequal}

A key feature of congruences with finite $\eta(\sigma)$ is that they restrict to the diagonal relation for all ranks greater than $\eta$; indeed, this will be shown in Lemma \ref{la331}, and is a consequence of the following two lemmas, which essentially show how $\sigma$-related pairs of equal rank give rise to~$\sigma$-related pairs of unequal (possibly smaller) finite ranks.

\begin{lemma}
\label{la329}
Suppose $\alpha,\beta\in D_\kappa$ where $\kappa\geq 1$ and $(\alpha,\beta)\in\sigma\setminus{\H}$.
Then for every finite cardinal~$q$ with $1\leq q\leq\kappa$, there exist $\gamma,\delta\in\M$ such that $(\gamma,\delta)\in\sigma$, $\rank(\gamma)=q$ and $\rank(\delta)<q$.
\end{lemma}

\begin{proof}
Since $(\alpha,\beta)\not\in{\H}$, we have $(\alpha,\beta)\not\in{\R}$ or $(\alpha,\beta)\not\in{\L}$.
Without loss of generality we may assume that the former is the case, which means (by Lemma \ref{lem:Green_M} \ref{it:GMiv}) that $\dom(\alpha)\neq\dom(\beta)$ or $\ker(\al)\not=\ker(\be)$.  

\pfcase{1}
Suppose first that $\dom(\alpha)\neq\dom(\beta)$.  Without loss of generality, we may assume that $\dom(\al)\not\sub\dom(\be)$, so there exists a transversal $A_1\cup B_1'$ of $\alpha$ such that $A_1\setminus \dom(\beta)\neq \emptyset$; let $a_1\in A_1\setminus\dom(\beta)$.
Let $A_i\cup B_i'$ ($i=2,\dots,q$) be any other $q-1$ transversals of $\alpha$; they exist because $q\leq\kappa=\rank(\alpha)$.
Pick arbitrary $a_i\in A_i$ ($i=2,\dots,q$).
Let $\th=\partpermIII{a_1}\cdots{a_q}{a_1}\cdots{a_q}$, and put
$(\gamma,\de)=(\theta\alpha,\theta\beta)\in\sigma$.
Note that $\rank(\gamma)=q$, with the transversals $\{a_i\}\cup B_i'$ ($i=1,\dots, q$).
On the other hand, $\rank(\delta)<q$, because $\dom(\delta)\subseteq\dom(\th)= \{a_1,\dots, a_q\}$ and $a_1\not\in\dom(\delta)$.

\pfcase{2}
Suppose now that $\dom(\al)=\dom(\be)$ but $\ker(\al)\not=\ker(\be)$.
Without loss of generality assume that $\ker(\al)\not\sub\ker(\be)$, so there exists $(x_1,x_2)\in\ker(\al)\sm\ker(\be)$.
Note that $x_1\not=x_2$, and that either $x_1,x_2$ both belong to $\dom(\alpha)=\dom(\beta)$ or else neither does.

\pfsubcase{2.1}  Suppose first that $x_1,x_2\in\dom(\alpha)=\dom(\beta)$.
Let $C_1\cup D_1'$ and $C_2\cup D_2'$ be the transversals of $\beta$ containing $x_1$ and $x_2$, respectively.
Let $C_i\cup D_i'$ ($i=3,\dots,q$) be an arbitrary further $q-2$ transversals of $\beta$,
and let $x_i\in C_i$ ($i=3,\dots,q$).
Define $\th=\partpermIII{x_1}\cdots{x_q}{x_1}\cdots{x_q}$,
and put $(\gamma,\de)=(\theta\be,\theta\al)\in\sigma$.
Then $\rank(\gamma)=q$, with the transversals $\{x_i\}\cup D_i'$ ($i=1,\dots,q$).
On the other hand, $\rank(\delta)<q$ because $\dom(\delta)\subseteq\dom(\th)=\{x_1,\dots,x_q\}$,
and $x_1,x_2$ are in the same transversal of $\delta$, as they already are in the same transversal of $\alpha$.

\pfsubcase{2.2}
Finally, suppose $x_1,x_2\not\in\dom(\alpha)=\dom(\beta)$.
Pick arbitrary transversals $A_i\cup B_i'$ ($i=1,\dots,q$) of $\alpha$,
and arbitrary elements $a_i\in A_i$ ($i=1,\dots,q$).
Let $\th=\partXXI{a_1}{a_{q-1}}{a_q}{}{a_1}{a_{q-1}}{x_1}{x_2,a_q}$,
and put $(\gamma,\de)=(\theta\alpha,\theta\beta)\in\sigma$.
Then $\rank(\gamma)=q$, with the transversals $\{a_i\}\cup B_i'$ (${i=1,\dots,q}$).
However, $\dom(\de)\sub\dom(\th)=\{a_1,\ldots,a_q\}$, yet $\{a_q\}$ is a singleton block of $\delta$ (to see this, recall that $\dom(\be)=\dom(\al)$ and $(x_1,x_2)\not\in\ker(\be)$), and hence $\rank(\delta)<q$,
completing the proof of this subcase and of the lemma.
\end{proof}

\begin{lemma}
\label{la330}
Suppose $\alpha,\beta\in D_\kappa$ where $\kappa\geq 2$, $(\alpha,\beta)\in\sigma\cap{\H}$ and $\alpha\neq\beta$.
Then for every finite cardinal $q$ with $1\leq q <\kappa$ there exist $\gamma,\delta\in D_q$ such that $(\gamma,\delta)\in \sigma\setminus{\H}$.
\end{lemma}

\begin{proof}
Let the transversals of $\alpha$ be $A_i\cup B_i'$ ($i\in I$), noting that $|I|=\kappa$.
Since $(\alpha,\beta)\in{\H}$, the transversals of $\beta$ are $A_i\cup B_{i\pi}'$, for some permutation $\pi\in\S_I$.
Furthermore, $(\al,\be)\in{\H}$ and $\alpha\neq\beta$ imply $\pi\neq\id_I$; say $i_1\pi=k\neq i_1$.
Pick a further $q-1$ transversals $A_{i_j}\cup  B_{i_j}'$ ($j=2,\dots, q$)
of $\alpha$,
making sure that $k\not\in \{i_1,\dots, i_q\}$;
this is possible because $q<\kappa=\rank(\alpha)$.  Define $\th=\partpermIII{A_{i_1}}\cdots{A_{i_q}}{A_{i_1}}\cdots{A_{i_q}}$, $\gamma=\theta\alpha$ and $\delta=\theta\beta$; clearly $(\gamma,\delta)\in\sigma$.
The transversals of $\gamma$ and $\delta$ are $A_{i_j}\cup B_{i_j}'$ and $A_{i_j}\cup B_{i_j\pi}'$ ($j=1,\dots,q$), respectively.
It follows that $\rank(\gamma)=\rank(\delta)=q$, but that $(\gamma,\delta)\not\in{\L}$ because $B_{k}\subseteq \codom(\de)$
and $B_{k}\not\subseteq \codom(\ga)$; cf.~Lemma \ref{lem:Green_M} \ref{it:GMv}.  Since ${\L}\supseteq{\H}$, it follows that $(\ga,\de)\in\si\sm{\H}$, as required.
\end{proof}

\subsection[Congruences with finite $\eta(\sigma)$: type \ref{CT1}]{\boldmath Congruences with finite $\eta(\sigma)$: type \ref{CT1}}\label{subsect:finite_eta}

We are now almost ready to deal with the congruences with $\eta(\si)$ finite; we will show in Proposition \ref{prop:CT1/2} that these are precisely the congruences of type \ref{CT1}, as enumerated in Theorem~\ref{thm-main}. 
Note that the description of \ref{CT1} congruences in Theorem \ref{thm-main} actually does not feature the parameter $\eta(\sigma)$, but
instead a closely related one:
\[
n=n(\sigma)= \begin{cases} 1 & \text{if } \eta(\sigma)=0\\
\eta(\sigma) & \text{if $2\leq\eta(\si)<\aleph_0$.} \end{cases}
\]
(Recall that $\eta(\si)\not=1$; cf.~Lemma \ref{la323}.)  We begin with the lemma promised at the beginning of Subsection \ref{subsect:unequal}.

\begin{lemma}
\label{la331}
If $\eta=\eta(\sigma)$ is finite, then $\sigma\subseteq R_{\eta+1}$.
\end{lemma}

\begin{proof}
We need to prove that if $(\alpha,\beta)\in\sigma$ with $\alpha\neq\beta$ then $\rank(\alpha),\rank(\beta)\leq \eta$.
Suppose, aiming for contradiction, that there is such a pair, but with $\kappa=\rank(\alpha)\geq \eta+1$.
Clearly, we must have $\rank(\beta)=\kappa$ as well, by the definition of $ \eta=\eta(\si)$.
Now, if $(\alpha,\beta)\not\in{\H}$ then by Lemma~\ref{la329} (keeping in mind that $\ka\geq1$), there exist $\gamma,\delta\in\M$ with $(\gamma,\delta)\in\sigma$, $\rank(\gamma)=\eta+1$ and $\rank(\delta)\leq \eta$; but this contradicts the definition of $ \eta=\eta(\si)$.
So suppose now that $(\alpha,\beta)\in{\H}$.
Since the $\H$-classes in $D_1$ are trivial, we cannot have $\kappa=1$, and hence $\kappa\geq 2$.
Put $q=\max(1,\eta)$ and note that $1\leq q<\ka$.
By Lemma~\ref{la330}, there exist $\gamma,\delta\in D_q$  such that $(\gamma,\delta)\in\sigma\setminus{\H}$.
Now using Lemma \ref{la329}, we see that there exist~$\gamma_1,\delta_1\in\M$ with $(\gamma_1,\delta_1)\in\sigma$, $\rank(\gamma_1)=q$ and $\rank(\delta_1)<q$, and this yet again contradicts the definition of $\eta=\eta(\si)$ since $q\geq \eta$.
\end{proof}

It will be convenient to single out the case in which $\eta(\si)=0$:

\begin{lemma}
\label{la332}
If $\eta(\sigma)=0$, then $\sigma=\lambda_{\zeta_1}^1\cap\rho_{\zeta_2}^1$.
\end{lemma}

\begin{proof}
By Lemma \ref{la331} we have $\sigma\subseteq R_1$, and by Lemma \ref{la326} we have $\sigma\cap R_1=\lambda_{\zeta_1}^1\cap\rho_{\zeta_2}^1$.
\end{proof}

For the rest of this subsection, we assume that $n=\eta(\sigma)\in [2,\aleph_0)$.  (By Lemma~\ref{la323}, it is impossible to have $\eta(\si)=1$.)

\begin{lemma}
\label{la333}
If $n=\eta(\si)\in[2,\aleph_0)$, then $\sigma=(\lambda_{\zeta_1}^n \cap\rho_{\zeta_2}^n )\cup \sigma\restr_{D_n }$.
\end{lemma}

\begin{proof}
This follows from Lemmas \ref{la323a}, \ref{la327} and \ref{la331}.
\end{proof}

Next, we proceed to associate a normal subgroup $N=N(\sigma)\unlhd \S_n $ to the congruence $\sigma$ with $n=\eta(\si)\in[2,\aleph_0)$.  To do so, recall that we assume $X$ contains $[1,\aleph_0)=\{1,2,\ldots\}$, and that we have defined $\permdec\pi=\partpermIII1\cdots n{1\pi}\cdots{n\pi}\in\M$ for each permutation $\pi\in\S_n$.  In this way, the $\H$-class of the idempotent $\ep_n=\permdec\id_n$ is the group $\permdec\S_n=\set{\permdec\pi}{\pi\in\S_n}$, the identity of which is $\ep_n$.  Because $\si$ is a congruence, it is clear that the set $\set{\al\in\permdec\S_n}{(\ep_n,\al)\in\si}$ is a normal subgroup of $\permdec\S_n$; it is therefore of the form $\permdec N$ for some normal subgroup $N=N(\si)$ of~$\S_n$.

\begin{lemma}
\label{la334}
If $n=\eta(\si)\in[2,\aleph_0)$, then with $N=N(\si)$ as above, we have $\sigma\restr_{D_n }=\nu_N$.
\end{lemma}

\begin{proof} 
($\subseteq$)
Suppose $(\alpha,\beta)\in\sigma\restr_{D_n }$.
By Lemma \ref{la329} (with $\ka=q=n$), and the definition of $n=\eta(\si)$, we must have $(\alpha,\beta)\in{\H}$,
so we may write
\begin{equation}
\label{eq:albe}
\al = \partXXII{A_1}{A_n}{C_i}{B_1}{B_n}{D_j} \AND \be = \partXXII{A_1}{A_n}{C_i}{B_{1\pi}}{B_{n\pi}}{D_j},
\end{equation}
where $\pi=\phi(\alpha,\beta)\in \S_n $; to show that $(\al,\be)\in\nu_N$, we must show that $\pi\in N$.  (The permutation $\phi(\al,\be)$ was defined before Theorem \ref{thm-main}; it is well defined up to conjugation.)
Let ${\th_1=\partpermIII1\cdots n{A_1}\cdots{A_n}}$ and $\th_2=\partpermIII{B_1}\cdots{B_n}1\cdots n$.
Then $(\ep_n,\permdec\pi)=(\th_1\al\th_2,\th_1\be\th_2)\in\si$, so that $\permdec\pi\in\permdec N$ by definition, whence~$\pi\in N$.

\pfitem{($\supseteq$)}
Suppose now that $(\alpha,\beta)\in\nu_N$.  
Recall that $\nu_N\subseteq{\H}$, and write $\alpha,\beta$ as in \eqref{eq:albe}, with
${\pi=\phi(\alpha,\beta)\in N}$.  Then $(\ep_n,\permdec\pi)\in\si$,  so with $\th_3=\partXXII{A_1}{A_n}{C_i}1 n{}$ and $\th_4=\partXXII1 n{}{B_1}{B_n}{D_j}$, we have $(\al,\be)=(\th_3\ep_n\th_4,\th_3\permdec\pi\th_4)\in\si$; moreover, $\rank(\al)=\rank(\be)=n$, so ${(\al,\be)\in\si\restr_{D_n}}$.
\end{proof}

\begin{lemma}
\label{la335}
If $\eta(\si)\in[2,\aleph_0)$, then with $N=N(\si)$ as above, we have $\sigma=\lambda_{\zeta_1}^N\cap \rho_{\zeta_2}^N$.
\end{lemma}

\begin{proof}
This follows from Lemmas \ref{la333} and \ref{la334}.
\end{proof}

Lemma \ref{la335} (together with Lemma \ref{la328}) constitutes a complete classification of congruences with $\eta (\sigma)=2$.
For the remaining finite values of $\eta(\si)$ we must also rule out the cases where~$\zeta_1=1$ or $\zeta_2=1$:

\begin{lemma}
\label{la336}
If $\eta(\sigma)\in [3,\aleph_0)$, then $\zeta_1,\zeta_2\neq 1$.
\end{lemma}

\begin{proof}
This follows directly from Lemmas \ref{la328} and \ref{la328b}.
\end{proof}

For convenience, we summarise the main conclusions of this subsection:

\begin{prop}\label{prop:CT1/2}
Any congruence $\si$ on $\M$ with $\eta(\si)<\aleph_0$ is of type \ref{CT1}, as listed in Theorem \ref{thm-main}.
\end{prop}

\pf This follows from Lemmas \ref{la323}, \ref{la328}, \ref{la332}, \ref{la335} and~\ref{la336}. \epf

\subsection[Congruences with infinite $\eta(\sigma)$: type \ref{CT2}]{\boldmath Congruences with infinite $\eta(\sigma)$: type \ref{CT2}}\label{subsect:infinite_eta}

Now we move to considering a congruence $\sigma$ with $\eta(\sigma)$ infinite.  We begin by isolating the case in which $\eta(\si)=|X|^+$.

\begin{lemma}\label{lem:nabla}
We have $\eta(\si)=|X|^+$ if and only if $\si=\nabla_{\M}$.
\end{lemma}

\pf
We have already observed that $\eta(\nabla_{\M})=|X|^+$.  Conversely, if $\eta(\si)=|X|^+$, then  ${(\ep_X,\emptypart)=(\ep_X,\wh\ep_X)\in\si}$ by Lemma \ref{la325}; thus, for any $\al\in\M$, ${(\al,\emptypart)=(\ep_X\al\ep_X,\emptypart\al\emptypart)\in\si}$, so that all elements of $\M$ are $\si$-related.
\epf

The forward implication in Lemma \ref{lem:nabla} also follows quickly from Lemmas \ref{la327} and \ref{la328b}.

For the remainder of the section, we assume that $\eta=\eta(\si)\in[\aleph_0,|X|]$, with the ultimate aim being to show that $\si$ is of type \ref{CT2}; see Lemmas \ref{la341} and \ref{la344}.  
By Lemma \ref{la327}, we already know that $\si\cap R_\eta=\lam_{\ze_1}^\eta\cap\rho_{\ze_1}^\eta$.  Thus, in light of Lemma \ref{la323a}, it remains to describe $\si\restr_{D_\ka}$ for each $\ka\in[\eta,|X|]$.  In doing so, we will also see how to determine the parameters $k=k(\sigma)$, and $\eta_i=\eta_i(\sigma)$ and $\xi_i=\xi_i(\sigma)$ for $i=1,\dots,k$.

To this end we define a map
\[
\Psi \colon  [\eta,|X|]\rightarrow [0,|X|^+]  \colon   \kappa\mapsto\kappa^\ast= {\LSUB} \bigset{|\al\sd\be|}{(\al,\be)\in\si\restr_{D_\ka}}.
\]
Roughly speaking, $\ka^*$ represents the boundary that the values of $|\al\sd\be|$ may approach but not attain (or exceed), as $(\al,\be)$ ranges over all $\si$-related pairs from $D_\ka$.

The two most important properties of $\Psi$ are recorded in Lemmas \ref{la338} and~\ref{la337} below.  The proofs of these two lemmas rely on the following technical lemma, whose proof will be given in Section~\ref{sect:tech}.

\begin{lemma}\label{lem:YZ}
If $(\al,\be)\in\si\restr_{D_\ka}$ where $\ka\geq\eta\geq\aleph_0$ and $|\al\sd\be|\not=0$, then for any disjoint subsets $Y,Z\sub X$ with $|Y|\leq\ka$ and $|Z|\leq|\al\sd\be|$, we have $(\ep_{Y\cup Z},\ep_Y)\in\si$.
\end{lemma}

\pf
This is proved in Lemma \ref{lem:tech3} \ref{it:h3ii}.
\epf

\begin{lemma}
\label{la338}
For any $\kappa\in[\eta,|X|]$, we have $\kappa^\ast\leq \eta$.
\end{lemma}

\begin{proof}
Aiming for a contradiction, suppose $\kappa^\ast>\eta$.   So there exists a pair ${(\alpha,\beta)\in\si\restr_{D_\kappa}}$ such that $\xi=|\alpha\sd\beta |\geq\eta$.  But then for any $Z\sub X$ with $|Z|=\xi$, Lemma~\ref{lem:YZ} (with $Y=\emptyset$) gives $(\ep_Z,\emptypart)\in\si$.  Since $\ep_Z\in D_\xi$ and $\emptypart\in D_0$ with $\xi\geq\eta$, this contradicts the definition of~$\eta=\eta(\si)$.
\end{proof}

\begin{lemma}
\label{la337}
The map $\Psi$ is order-reversing: i.e., $\kappa_1\leq\kappa_2 \implies \kappa_1^\ast\geq \kappa_2^\ast$ for all~${\ka_1,\ka_2\in[\eta,|X|]}$.
\end{lemma}

\begin{proof}
Write  ${\xi_1=\kappa_1^\ast}$ and $\xi_2=\kappa_2^\ast$ and suppose, aiming for a contradiction, that $\xi_1<\xi_2$.  This means that there exists a pair $(\alpha,\beta)\in \si\restr_{D_{\kappa_2}}$ with $|\alpha\sd\beta|\geq\xi_1$ (and note that $\xi_1>0$ as~$\si$ is reflexive).  Pick disjoint $Y,Z\sub X$ with $|Y|=\ka_1$ and $|Z|=\xi_1$.  Since $|Y|=\ka_1\leq\ka_2$, Lemma~\ref{lem:YZ} then gives $(\ep_{Y\cup Z},\ep_Y)\in\si$.  But $\ep_{Y\cup Z},\ep_Y\in D_{\ka_1}$ (since ${|Z|=\xi_1=\ka_1^*\leq\eta\leq\ka_1=|Y|}$, using~Lemma \ref{la338} for the first inequality) and ${|\ep_{Y\cup Z}\sd\ep_Y|=3|Z|\geq\xi_1}$, contradicting the definition of $\xi_1=\ka_1^*$.
\end{proof}

By Lemmas \ref{la338} and \ref{la337}, $\Psi$ maps the interval $[\eta,|X|]$ in an order-reversing fashion into the interval $[0,\eta]$.
Since the cardinals are well-ordered, it follows that the image of $\Psi$ is finite.  (Otherwise it would contain an infinite chain $\ka_1^*<\ka_2^*<\cdots$, in which case $\ka_1>\ka_2>\cdots$, a contradiction; cf.~\cite[p234]{CPbook2}.)  We write 
\begin{equation}
\label{eq:xi}
\im(\Psi) = \{\xi_1,\dots,\xi_k\}=\{\xi_1(\sigma),\dots,\xi_k(\sigma)\} , \quad\text{where $k=k(\si)\geq1$, and where $\xi_1>\cdots>\xi_k$.}
\end{equation}
Now let
\[
\label{eq:eta}
\eta_i =\eta_i(\sigma)= {\min}\set{ \kappa }{ \kappa^\ast=\xi_{i+1}} \qquad\text{for each $0\leq i\leq k-1$.}
\]
By definition, and since $\Psi$ is order-reversing,  we have $\eta^*=\xi_1$, and so $\eta_0=\eta=\eta(\si)$.  Also, let us define ${\eta_k=\eta_k(\sigma)=|X|^+}$.  Note that
\begin{equation}
\label{eq:eta}
\eta_0<\eta_1<\cdots<\eta_k=|X|^+.
\end{equation}

\begin{rem}
\label{rem:PsiPar} 
The parameters $k(\sigma)$, $\xi_i(\sigma)$ and $\eta_i(\sigma)$ have been defined from the mapping $\Psi$ and are uniquely determined by this mapping.
Conversely, it is easy to see that $\Psi$ itself is uniquely determined by the values of these parameters. Thus, $\Psi$ can also be regarded as a parameter of~$\sigma$, and in that case we shall write $\Psi=\Psi(\sigma)$; the value of this mapping at $\kappa$ will then be denoted~$\Psi(\sigma)(\kappa)$.
This point of view will be particularly useful in Part \ref{part:II} where we analyse the structure of the lattice $\Cong(\M)$.
Recall that we consider the universal congruence $\nabla_{\M}$ to be of type \ref{CT2} with $\eta=|X|^+$, in which case the interval $[\eta,|X|]$ is empty; we therefore consider~$\Psi(\nabla_{\M})$ to be the empty mapping.  
\end{rem}

We have now defined all of the relevant parameters associated to the congruence $\si$.  In Lemma~\ref{la341} below, we will show that they are constrained in the way stated in Theorem \ref{thm-main}, and in Lemma~\ref{la344} that $\si$ is precisely the congruence from the theorem with these parameter values.  First we need two technical lemmas, the proofs of which will be given in Section \ref{sect:tech}.

\begin{lemma}
\label{la340b}
If $(\alpha,\beta)\in\sigma\restr_{D_\kappa}$ where $\kappa\geq\eta\geq\aleph_0$ and $\al\not=\be$, then for any $\gamma,\delta\in D_\kappa$ with $|\gamma\sd\delta|<\aleph_0$, we have $(\gamma,\delta)\in \sigma$.
\end{lemma}

\begin{proof}
This is proved in Lemma \ref{la-sw4} for $\PB_X$, and in Lemma \ref{la-rg1} \ref{it:rg1iv} for $\P_X$.
\end{proof}

\begin{lemma}
\label{la340}
If $(\alpha,\beta)\in\sigma\restr_{D_\kappa}$ where $\kappa\geq\eta\geq\aleph_0$, then for any $\gamma,\delta\in D_\kappa$ with ${|\gamma\sd\delta|\leq |\alpha\sd\beta|}$, we have $(\gamma,\delta)\in \sigma$.
\end{lemma}

\begin{proof}
This is proved in Lemma \ref{la-sw5} for $\PB_X$, and in Lemma \ref{la-hd5} \ref{it:hd5iv} for $\P_X$.
\end{proof}

Next we show that certain values from $[0,\eta]$ are never in the image of $\Psi$.

\begin{lemma}
\label{la339}
For any $\kappa\in[\eta,|X|]$, the only possible finite value for $\ka^*$ is $1$.
\end{lemma}

\begin{proof}
That $\kappa^\ast\neq 0$ is clear from the definition, since $\si$ is reflexive.
Suppose $\kappa^\ast\in [2,\aleph_0)$, and write $m=\ka^*$.
By definition, there exist $\alpha,\beta\in D_\kappa$ such that $(\alpha,\beta)\in\sigma$ and ${|\alpha\sd\beta|=m-1\geq 1}$: i.e., $\al\not=\be$.
Let $\gamma,\delta\in D_\kappa$ be any two partitions with $m\leq|\gamma\sd\delta|<\aleph_0$.
By Lemma \ref{la340b},
we have $(\gamma,\delta)\in\sigma$, and this contradicts the definition of $\kappa^\ast=m$.
\end{proof}

The next result is immediate from the above definitions, Lemmas~\ref{la328b},~\ref{la338} and~\ref{la339}, and from \eqref{eq:xi} and \eqref{eq:eta}:

\begin{lemma}
\label{la341}
If $\eta=\eta(\si)\in[\aleph_0,|X|]$, then
\bit
\item 
$k\geq1$, $\ze_1,\ze_2,\eta_1,\ldots,\eta_k\in [\eta,|X|^+]$, $\xi_1,\ldots,\xi_k\in \{1\}\cup[\aleph_0,\eta]$, and 
\item 
$\xi_k<\dots<\xi_1\leq \eta<\eta_1<\dots<\eta_k=|X|^+$. \epfres
\eit
\end{lemma}

We are now ready to complete the last major step.

\begin{lemma}
\label{la344}
If $\eta=\eta(\sigma)\in[\aleph_0,|X|]$, then with the parameters as defined above, we have 
\[
{\sigma=(\lambda_{\zeta_1}^\eta\cap\rho_{\zeta_2}^\eta)\cup\mu_{\xi_1}^{\eta_1}\cup\dots\cup\mu_{\xi_k}^{\eta_k}}.
\]
\end{lemma}

\begin{proof}
Denote the relation on the right-hand side by $\tau$.  

\pfitem{($\subseteq$)}
First suppose $(\alpha,\beta)\in\sigma$, and write $\ka=\rank(\al)$.  If $\ka<\eta$, then Lemma \ref{la323a} gives $\rank(\beta)<\eta$ as well; together with Lemma \ref{la327}, it follows that $(\al,\be)\in\si\cap R_\eta=\lam_{\ze_1}^\eta\cap\rho_{\ze_2}^\eta\sub\tau$.
Now suppose $\kappa\geq \eta$, so that $\rank(\beta)=\kappa$ as well, by definition of $\eta$.
Let $1\leq i\leq k$ be such that $\ka\in[\eta_{i-1},\eta_i)$.  Then $\al,\be\in D_\ka\sub I_{\eta_i}$, and also $|\alpha\sd\beta|<\kappa^\ast=\xi_i$.  Thus, $(\alpha,\beta)\in\mu_{\xi_i}^{\eta_i}\subseteq\tau$.

\pfitem{($\supseteq$)}
Now suppose $(\alpha,\beta)\in\tau$.  If $(\alpha,\beta)\in\lambda_{\zeta_1}^\eta\cap\rho_{\zeta_2}^\eta$, then $(\alpha,\beta)\in \sigma$ by Lemma \ref{la327}.  So suppose instead that $(\alpha,\beta)\in \mu_{\xi_i}^{\eta_i}\setminus (\lambda_{\zeta_1}^\eta\cap\rho_{\zeta_2}^\eta)$ for some $1\leq i\leq k$.  
Since then $\al\not=\be$, we must have $\al,\be\in I_{\eta_i}$ and $|\al\sd\be|<\xi_i$.  By Lemmas \ref{la32} \ref{it:32vii} and \ref{la341}, we have
\[
|\ol\al\sd\ol\be| \leq |\al\sd\be| <\xi_i \leq\eta \leq \ze_1
\ANDSIM
|\ul\al\sd\ul\be|<\ze_2,
\]
so that $(\al,\be)\in\lam_{\ze_1}\cap\rho_{\ze_2}$.  Since $(\al,\be)\not\in\lam_{\ze_1}^\eta\cap\rho_{\ze_2}^\eta$, it follows (renaming $\al,\be$ if necessary) that ${\ka=\rank(\al)\geq\eta\geq\xi_i}$.  By Lemma~\ref{la316}, we have $\rank(\be)=\ka$ as well.  Since ${(\al,\be)\in\mu_{\xi_i}^{\eta_i}\sub R_{\eta_i}}$ and $\al\not=\be$, we have $\ka<\eta_i$, and so $\ka\in[\eta,\eta_i)=[\eta_0,\eta_i)$.  Let $1\leq j\leq i$ be such that $\ka\in[\eta_{j-1},\eta_j)$; it then follows that ${\kappa^\ast=\xi_j\geq \xi_i>|\al\sd\be|}$.  By definition of $\kappa^\ast$, it follows that there exists $(\gamma,\delta)\in \si\restr_{D_\kappa}$ with ${|\gamma\sd\delta|\geq|\al\sd\be|}$.  But then $(\alpha,\beta)\in\sigma$ by Lemma \ref{la340}.
\end{proof}

Again, we give a summary for convenience; it follows immediately from Lemmas~\ref{lem:nabla},~\ref{la341} and~\ref{la344}.

\begin{prop}\label{prop:CT3b}
Any non-universal congruence $\si$ on $\M$ with $\eta(\si)\geq\aleph_0$ is of type \ref{CT2}, as listed in Theorem \ref{thm-main}. \epfres
\end{prop}

Save for the lemmas whose proofs have been deferred to the next section, this completes the proof of Theorem \ref{thm-main}.

\section{Technical lemmas}\label{sect:tech}

A number of lemmas from Section \ref{sect:stage2} are as yet unproved, and the goal of this section is to provide the proofs.  The lemmas in question naturally fall into three categories:
\bit
\item Lemma \ref{lem:YZ} concerns equivalence of partitions of the form $\ep_Y$ and $\ep_{Y\cup Z}$ under certain conditions, and is proved in Subsection \ref{subsect:YZ} for both monoids.
\item Lemmas \ref{la328a} and \ref{la328c} concern restrictions of congruences to the bottom $\D$-class; they are proved in Subsection \ref{subsect:PBX1} for $\PB_X$ and Subsection \ref{subsect:PX1} for $\P_X$.
\item Lemmas \ref{la340b} and \ref{la340} concern restrictions of congruences to $\D$-classes at or above~$D_\eta$, where $\eta=\eta(\si)$ is infinite; they are proved in Subsection \ref{subsect:PBX2} for $\PB_X$ and Subsection \ref{subsect:PX2} for $\P_X$.  
\eit
Subsection \ref{subsect:part_lat} contains some preliminary discussion on meets, joins and refinement of partitions, which will be relevant to the calculations of Subsections \ref{subsect:PBX1}--\ref{subsect:PX2}.  

Before we begin, we prove a lemma about infinite graphs that will be used in Subsections~\ref{subsect:YZ} and \ref{subsect:PBX1}.  Recall that an \emph{independent set} in a graph $\Ga$ is a subset $A$ of the vertex set of~$\Ga$ such that there are no edges between the vertices from $A$.  

\begin{lemma}
\label{la-tt4}
Let $\Ga$ be a (simple, undirected) graph with $\kappa\geq\aleph_0$ vertices in which every vertex has finite degree.  Then $\Ga$ contains an independent set of size $\kappa$.
\end{lemma}

\begin{proof}
The finite degree assumption implies that all connected components of $\Ga$ have cardinality at most $\aleph_0$.
If $\Ga$ has $\kappa$ connected components then picking one representative from each component yields the desired independent set.
Otherwise we have $\kappa=\aleph_0$, and at least one connected component, say $C$, is infinite.
Define a sequence of elements $c_1,c_2,c_3,\ldots\in C$ recursively as follows.
First, let $c_1\in C$ be arbitrary.  Now suppose $k\geq1$ and that we have already defined $c_1,\ldots, c_k\in C$ so that $\{c_1,\ldots,c_k\}$ is an independent set.  Let $N$ be the set consisting of $c_1,\ldots,c_k$ and all neighbours of these vertices; since $N$ is finite, we may pick any $c_{k+1}\in C\sm N$.  Clearly $\{c_1,c_2,c_3,\ldots\}$ is an independent set of size $\aleph_0=\kappa$ in $\Ga$.
\end{proof}

\subsection{Lemma \ref{lem:YZ}}\label{subsect:YZ}

Throughout this subsection, $\M$ denotes either $\PB_X$ or $\P_X$, where $X$ is an infinite set, and $\si$ is an arbitrary congruence on $\M$.
Our main goal here is to prove Lemma \ref{lem:YZ}, which gives conditions for $\si$ to contain a pair of the form $(\ep_{Y\cup Z},\ep_Y)$, where $Y$ and $Z$ are subsets of $X$ with certain prescribed sizes.  This will be achieved in Lemma \ref{lem:tech3}, after a series of preliminary lemmas, some of which will also be of use in subsequent subsections.

We begin with a lemma about products of the form $\al\be\ga$; it concerns a certain scenario in which the upper and lower parts of a two-element transversal of $\al\be\ga$ are upper and lower parts of transversals of $\al$ and $\ga$, respectively.

\begin{lemma}\label{lem:trans1}
Suppose $\al,\be,\ga\in\M$ are such that every lower non-transversal of~$\al$ and every upper non-transversal of $\ga$ are singletons.  Suppose also that $\al$ and $\ga$ contain the transversals $\{x,y'\}$ and $\{u,v'\}$, respectively.
If $\{x,v'\}$ is a transversal of $\al\be\ga$, then there exists a transversal $A\cup B'$ of $\be$ such that ${A\cap\codom(\al)=\{y\}}$ and $B\cap\dom(\ga)=\{u\}$.
\end{lemma}

\pf
Let the block of $\be$ containing $y$ be $A\cup B'$, where $B\sub X$ is possibly empty.  If any transversal $C\cup D'\not=\{x,y'\}$ of $\al$ satisfied $D\cap A\not=\emptyset$, then $C\cup\{x\}$ would be contained in a block of $\al\be$, and hence also of $(\al\be)\ga$, a contradiction.  It follows that $A\cap\codom(\al)=\{y\}$, and so also (by the assumption on lower non-transversals of $\al$) that $\{z'\}$ is a block of $\al$ for all $z\in A\sm\{y\}$.  Thus, $\{x\}\cup A''\cup B'$ is a connected component of the product graph $\Pi(\al,\be)$, and so $B\not=\emptyset$ (or else $\{x\}$ would be a block of $\al\be$, and hence also $(\al\be)\ga$, a contradiction).  

To summarise the previous paragraph:  the block of $\al\be$ containing $x$ is of the form $\{x\}\cup B'$ for some transversal $A\cup B'$ of $\be$ with $A\cap\codom(\alpha)=\{y\}$.  By a dual argument applied to the product $(\al\be)\ga$, the block of $\al\be\ga$ containing $v'$ is of the form $E\cup\{v'\}$ for some transversal $E\cup F'$ of $\al\be$ with $F\cap\dom(\ga)=\{u\}$.  But the block of $\al\be\ga$ containing $v'$ is $\{x,v'\}$, so we must have $E=\{x\}$, so that $\{x\}\cup F'$ is a transversal of $\al\be$.
Since we have already seen that $\{x\}\cup B'$ is a transversal of $\al\be$, it follows that $B=F$, and so $B\cap\dom(\ga)=F\cap\dom(\ga)=\{u\}$.
\epf

Recall that we wish to prove that under certain conditions,  $\si$ contains pairs of the form $(\epsilon_{Y\cup Z},\epsilon_Y)$.  We begin building towards this by first considering the special (and technical) cases where $\si$ is known to contain a pair of the form $(\epsilon_Y,\al)$ or $(\epsilon_{Y\cup Z},\al)$.

\begin{lemma}\label{lem:tech01}
If $\ka\geq\aleph_0$ and $(\ep_Y,\al)\in\si$, where $|Y|=\ka$ and $\al$ contains fewer than $\ka$ of the transversals of $\ep_Y$, then we have $(\ep_Y,\emptypart)\in\si$.
\end{lemma}

\pf
Let $\al_1=\ep_Y\al\ep_Y$, noting that $(\ep_Y,\al_1)=(\ep_Y\ep_Y\ep_Y,\ep_Y\al\ep_Y)\in\si$, and that every transversal of $\al_1$ is contained in $Y\cup Y'$.

\pfcase{1}
If $\al_1$ has fewer than $\ka=|Y|$ transversals, then Lemma \ref{la324} (the proof of which did not rely on any of the technical lemmas proved in this section) gives ${(\ep_Y,\emptypart)=(\ep_Y,\wh\ep_Y)\in\si}$.

\pfcase{2}
Next, suppose $\al_1$ has $\ka$ transversals $\bigset{\{u_i,u_i'\}}{i\in I}$.  For each $i\in I$, there is a transversal $A_i\cup B_i'$ of $\al$ such that $A_i\cap Y=\{u_i\}=B_i\cap Y$.  By the assumption on the transversals of $\al$, some $\ka$ of these transversals $A_i\cup B_i'$ have size at least $3$.  By symmetry, we may assume there is a subset $J\sub I$ of size $\ka$ such that $|A_j|\geq2$ for all $j\in J$.  For each $j\in J$, fix some ${v_j\in A_j\sm\{u_j\}}$; note that $v_j\in X\sm Y$ for each $j$.  Also, since ${|Y|=\ka=|J|}$, we may write ${Y=\set{y_j}{j\in J}}$.  Then with $\th_1=\binom{y_j}{v_j}_{j\in J}$ and $\th_2=\binom{u_j}{y_j}_{j\in J}$, we have ${(\ep_Y,\emptypart)=(\th_1\al\th_2,\th_1\ep_Y\th_2)\in\si}$.

\pfcase{3}
Next, suppose $\al_1$ contains a set of $\ka$ transversals $\TT = \set{T_i}{i\in I}$, where $T_i=\{u_i,v_i'\}$ with $u_i\not=v_i$ (and $u_i,v_i\in Y$) for each $i$.  Define a graph $\Ga$ with vertex set $\TT$, and with an edge between distinct $T_i$ and $T_j$ if some transversal of $\ep_Y$ has non-trivial intersection with both $T_i$ and $T_j$ (meaning that $u_i=v_j$ or $u_j=v_i$).  Then $\Ga$ has $\ka$ vertices, each of degree at most $2$.  Lemma \ref{la-tt4} guarantees the existence of an independent set in $\Ga$ of size $\ka$, say $\set{T_j}{j\in J}$, where $J\sub I$.  Note that the independence condition says that the sets $\set{u_j}{j\in J}$ and $\set{v_j}{j\in J}$ are disjoint.  Since $|Y|=\ka=|J|$, we may write $Y=\set{y_j}{j\in J}$.  Then with $\th_1=\binom{y_j}{u_j}_{j\in J}$ and $\th_2=\binom{v_j}{y_j}_{j\in J}$, we have $(\ep_Y,\emptypart)=(\th_1\al_1\th_2,\th_1\ep_Y\th_2)\in\si$.

\pfcase{4}
Finally, suppose $\al_1$ contains a set of $\ka$ transversals, each with at least three elements, say $\set{A_i\cup B_i'}{i\in I}$.  Then for each $i$, we may fix some $a_i\in A_i$ and $b_i\in B_i$ with $a_i\not=b_i$.  Since $|Y|=\ka=|I|$, we may write $Y=\set{y_i}{i\in I}$.  Let $\th_3=\binom{y_i}{a_i}_{i\in I}$ and $\th_4=\binom{b_i}{y_i}_{i\in I}$, and put $\al_2=\th_3\ep_Y\th_4$.  Since $\th_3\al_1\th_4=\ep_Y$, it follows that $(\ep_Y,\al_2)\in\si$.  
Since any transversal of $\al_2$ is of the form $\{y_i,y_j'\}$ for distinct $i,j\in I$, $\al_2$ satisfies the assumptions of either Case 1 or Case 3, and so we are done.
\epf

For the proof of the next lemma, it is convenient to introduce some extra terminology.  
Suppose~${X=Y\cup W}$ where $Y$ and $W$ are disjoint.  For $\al\in\MY$ and $\be\in\MW$, the union of $\al$ and~$\be$ belongs to $\M$; we will denote it by $\al\sqcup\be$.  The set $\set{\al\sqcup\be}{\al\in\MY,\ \be\in\MW}$ of all partitions created in this way is a submonoid of $\M$ and is isomorphic to the direct product of~$\MY$ and~$\MW$.  
This has an important consequence for congruences on $\M$.  Suppose we knew that $(\al,\be)\in\si$ (where as usual $\si$ is a congruence on $\M$), and that $\al=\th\sqcup\al_1$ and $\be=\th\sqcup\be_1$, for some $\th\in\MY$ and $\al_1,\be_1\in\MW$.  Now let $\tau$ be the congruence on $\MW$ generated by the pair $(\al_1,\be_1)$.  Then for any $(\al_2,\be_2)\in\tau$, we have $(\th\sqcup\al_2,\th\sqcup\be_2)\in\si$.

\begin{lemma}\label{lem:tech08}
If $\aleph_0\leq\xi\leq\ka$ and $(\ep_{Y\cup Z},\al)\in\si$, where $Y$ and $Z$ are disjoint subsets of $X$ with $|Y|=\ka$ and $|Z|=\xi$, and if $\al$ contains the transversals of $\ep_Y$
but fewer than $\xi$ of the transversals of $\ep_Z$, then we have $(\ep_{Y\cup Z},\ep_Y)\in\si$.
\end{lemma}

\pf
Write $W=X\sm Y$.  
During this proof, in any expression $\be\sqcup\ga$, it is assumed that $\be\in\MY$ and $\ga\in\MW$; thus, for example, if we write $\ep_Y=\ep_Y\sqcup\emptypart$, the ``$\ep_Y$'' on the left is the usual element of $\M$, but the one on the right is the corresponding element of~$\MY$ (indeed the identity of~$\MY$), and ``$\emptypart$'' denotes the element of $\MW$ all of whose blocks are singletons.

Beginning the proof now, note that by the form of $\al$, we have $\al=\ep_Y\sqcup\al_1$ for some ${\al_1\in\MW}$.  Also $\ep_{Y\cup Z}=\ep_Y\sqcup\ep_Z$, so it follows that ${(\ep_Y\sqcup\ep_Z,\ep_Y\sqcup\al_1)\in\si}$.  Let~$\tau$ be the congruence on $\MW$ generated by $(\ep_Z,\al_1)$.  By Lemma \ref{lem:tech01} (in the monoid~$\MW$), we have $(\ep_Z,\emptypart)\in\tau$.  As noted before the statement of the lemma, it follows that~$\si$ contains ${(\ep_Y\sqcup\ep_Z,\ep_Y\sqcup\emptypart)=(\ep_{Y\cup Z},\ep_Y)}$.
\epf

\begin{lemma}\label{lem:tech02}
If $\aleph_0\leq\xi\leq\ka$ and $(\ep_Y,\al)\in\si$, where $|Y|=\ka$, $\al$ contains the transversals of $\ep_Y$, and the union of the non-singleton non-transversals of $\al$ has size at least $\xi$, then for any subset $Z\sub X\sm Y$ of size $\xi$, we have ${(\ep_{Y\cup Z},\ep_Y)\in\si}$.
\end{lemma}

\pf
Fix an arbitrary subset $Z\sub X\sm Y$ of size $\xi$.  Write $Y=\set{y_i}{i\in I}$ and $Z=\set{z_j}{j\in J}$; since $|Z|\leq|Y|=\ka$ and $\ka\geq\aleph_0$, we may assume that $J\sub I$ and $|I\sm J|=\ka$.  We fix a bijection $\psi\colon I\to I\sm J$, and define $\th=\partpermII{y_{i\psi}}{y_j}{y_i}{z_j}_{i\in I,\ j\in J}$.
By symmetry, we may assume that $\al$ has a set of upper non-transversals $\set{A_k}{k\in K}$ where $|A_k|\geq2$ for all $k\in K$, and where $\bigcup_{k\in K}A_k$ has size at least~$\xi$.

\pfcase{1}
Suppose first that $|K|\geq\xi$.  For convenience, we may assume that $J\sub K$.  For each $j\in J$, fix distinct $a_j,b_j\in A_j$.  Then with $\th_1=\partXXIV{y_i}{z_j}{}{y_{i\psi}}{a_j}{b_j,y_j}_{i\in I,\ j\in J}$, 
and with $\th$ as defined above, 
we have $(\ep_{Y\cup Z},\ep_Y)=(\th_1\al\th,\th_1\ep_Y\th)\in\si$.

\pfcase{2}
Suppose now that $|K|<\xi$ (and note that this case cannot occur if $\M=\PB_X$).  Choose any subset $W\sub\bigcup_{k\in K}A_k$ of size $\xi$, and write $W=\set{w_j}{j\in J}$.  Define the partition $\th_2=\partpermII{y_i}{z_j}{y_{i\psi}}{y_j,w_j}_{i\in I,\ j\in J}$, and put $\al_1=\th_2\al\th$.  This time $\th_2\ep_Y\th=\ep_{Y\cup Z}$, so we have $(\ep_{Y\cup Z},\al_1)\in\si$.  But $\al_1$ satisfies the conditions of Lemma~\ref{lem:tech08} (note that $\al_1$ has at most $|K|<\xi$ transversals contained in $Z\cup Z'$), so the proof is complete after applying that lemma.
\epf

We now move on to three further lemmas, which give somewhat more general situations, albeit still technical in nature, under which the congruence $\sigma$ must contain a pair of the form~$(\epsilon_{Y\cup Z},\epsilon_Y)$.

\begin{lemma}\label{lem:tech07}
If $\aleph_0\leq\xi\leq\ka$ and $(\al,\be)\in\si$, where $\al\cap\be$ has $\ka$ transversals, and $\al\sd\be$ has at least $\xi$ transversals, then for any disjoint subsets $Y,Z\sub X$ with $|Y|=\ka$ and $|Z|=\xi$, we have ${(\ep_{Y\cup Z},\ep_Y)\in\si}$.
\end{lemma}

\pf
Fix arbitrary disjoint subsets $Y,Z\sub X$ with $|Y|=\ka$ and $|Z|=\xi$, write $Y=\set{y_i}{i\in I}$ and $Z=\set{z_j}{j\in J}$, and let the transversals of $\al\cap\be$ be ${\set{A_i\cup B_i'}{i\in I}}$.  We begin by claiming that either
\begin{enumerate}[label=(\alph*)]
\item \label{it:h07a} $(\ep_{Y\cup Z},\al_1)\in\si$ for some $\al_1\in\M$ containing all the transversals of $\ep_Y$, but fewer than $\xi$ of the transversals of $\ep_Z$, or else 
\item \label{it:h07b}  there exist transversals $\set{C_k\cup D_k'}{k\in K}\sub\be\sm\al$ and $\set{E_k\cup F_k'}{k\in K}\sub\al\sm\be$, where $|K|=\xi$, and such that $C_k\cap E_k$ and $D_k\cap F_k$ are non-empty for each $k\in K$.
\end{enumerate}
Since $\al\sd\be$ has at least $\xi$ transversals, we may assume without loss of generality that $\be\sm\al$ contains $\xi$ transversals, say ${\set{C_j\cup D_j'}{j\in J}}$. 
For each $j\in J$, fix arbitrary $c_j\in C_j$ and $d_j\in D_j$.
Let $\th_1=\partpermII{y_i}{z_j}{A_i}{c_j}_{i\in I,\ j\in J}$ and $\th_2=\partpermII{B_i}{d_j}{y_i}{z_j}_{i\in I,\ j\in J}$, and put $\al_1=\th_1\al\th_2$.  Since $\th_1\be\th_2=\ep_{Y\cup Z}$, we have $(\ep_{Y\cup Z},\al_1)\in\si$.  Now, $\al_1$ contains the transversals of~$\ep_Y$.  If $\al_1$ contains fewer than $\xi$ of the transversals of $\ep_Z$, then \ref{it:h07a} holds, so let us assume that~$\al_1$ contains $\xi$ of the transversals of $\ep_Z$, say $\bigset{\{z_k,z_k'\}}{k\in K}$ where $K\sub J$.  
Lemma~\ref{lem:trans1} (applied to the product $\al_1=\th_1\al\th_2$) says that for any $k\in K$, $\al$ contains a transversal $E_k\cup F_k'$ such that $E_k\cap\codom(\th_1)=\{c_k\}$ and $F_k\cap\dom(\th_2)=\{d_k\}$.  This shows that $C_k\cap E_k$ and $D_k\cap F_k$ are non-empty for each $k\in K$.  
Since $C_k\cup D_k'\in\beta\setminus\alpha$ it follows that
$E_k\cup F_k'\in \al\sm\be$ for all $k\in K$. Hence \ref{it:h07b} holds, and the claim is proved. 

Returning now to the main proof, note that if \ref{it:h07a} holds, then Lemma \ref{lem:tech08} immediately gives $(\ep_{Y\cup Z},\ep_Y)\in\si$. Thus, for the remainder of the proof, we will assume that \ref{it:h07b} holds.  
For any $k\in K$, we have $C_k\cup D_k'\not=E_k\cup F_k'$, so either $C_k\cup D_k'\not\sub E_k\cup F_k'$ or $E_k\cup F_k'\not\sub C_k\cup D_k'$.  Since $|K|=\xi\geq\aleph_0$, we may assume by symmetry that $\xi$ values of $k$ satisfy the latter.  By symmetry again, we may assume that $\xi$ values of $k$ satisfy $E_k\not\sub C_k$; let $L$ be the set of all such $k$.  For each $l\in L$, fix some $d_l\in D_l\cap F_l$ and $e_l\in E_l\sm C_l$.  Since $|J|=\xi=|L|$, we may write $J=\set{j_l}{l\in L}$.  Let $\th_3=\partpermII{y_i}{z_{j_l}}{A_i}{e_l}_{i\in I,\ l\in L}$ and $\th_4=\partpermII{B_i}{d_l}{y_i}{z_{j_l}}_{i\in I,\ l\in L}$, and put $\al_2=\th_3\be\th_4$.  Then $\th_3\al\th_4=\ep_{Y\cup Z}$, so $(\ep_{Y\cup Z},\al_2)\in\si$.  But $\al_2$ satisfies the conditions of Lemma~\ref{lem:tech08} (note that $\{z,z'\}$ is not a block of~$\al_2$ for any $z\in Z$), so that lemma completes the proof.
\epf

\begin{lemma}\label{lem:tech06}
If $\aleph_0\leq\xi\leq\ka$ and $(\al,\be)\in\si$, where $\al\cap\be$ has $\ka$ transversals, and $\al\sd\be$ has fewer than~$\xi$ transversals but at least $\xi$ non-transversals, then for any disjoint subsets $Y,Z\sub X$ with $|Y|=\ka$ and $|Z|=\xi$, we have ${(\ep_{Y\cup Z},\ep_Y)\in\si}$.
\end{lemma}
\pf
Fix arbitrary disjoint subsets $Y,Z\sub X$ with $|Y|=\ka$ and $|Z|=\xi$,  write ${Y=\set{y_i}{i\in I}}$ and $Z=\set{z_j}{j\in J}$, and suppose the transversals of $\al\cap\be$ are $\set{A_i\cup B_i'}{i\in I}$.  Since $|Z|\leq|Y|$, we may assume for convenience that $J\sub I$ and $|I\sm J|=\ka$.  We begin by claiming that either
\begin{enumerate}[label=(\alph*)]
\item \label{it:h06a} $(\ep_Y,\al_1)\in\si$ for some $\al_1\in\M$ containing all the transversals of $\ep_Y$, and such that the union of the non-singleton non-transversals of $\al_1$ has size at least $\xi$, or else 
\item \label{it:h06b} there exist non-transversals $\set{C_l}{l\in L}\sub\be\sm\al$ and $\set{D_l}{l\in L}\sub\al\sm\be$, where $|L|=\xi$, and such that $C_l\cap D_l$ is non-empty for each $l\in L$.
\end{enumerate}
Since $\al\sd\be$ has at least $\xi$ non-transversals, we may assume without loss of generality that $\be\sm\al$ contains $\xi$ upper non-transversals, say $\set{C_j}{j\in J}$.  For each $j\in J$, fix some $c_j\in C_j$.  Define $\th_1=\partpermII{y_i}{z_j}{A_i}{c_j}_{i\in I,\ j\in J}$ and $\th_2=\binom{B_i}{y_i}_{i\in I}$, and put $\al_1=\th_1\al\th_2$.  Since $\th_1\be\th_2=\ep_Y$, we have $(\ep_Y,\al_1)\in\si$.  Note that $\al_1$ contains the transversals of $\ep_Y$.  If the union of the non-singleton non-transversals of $\al_1$ has size at least $\xi$, then \ref{it:h06a} holds, so let us assume that the union of the non-singleton non-transversals of $\al_1$ has size strictly less than $\xi$.  Thus, since $|Z|=\xi$, and since $\dom(\al_1)=Y$, it follows that $\xi$ of the elements of $Z$ belong to singleton blocks of~$\al_1$, say $\bigset{\{z_k\}}{k\in K}$, where~$K\sub J$.  
For any $k\in K$, $c_k$ belongs either to some upper non-transversal of $\al$, or else to the upper part of some transversal from $\al\sm\be$ (by construction, we cannot have $c_k\in A_i$ for any $i$).  Since $\al\sd\be$ contains fewer than $\xi$ transversals, it follows that the set
\[
L=\set{k\in K}{c_k\text{ belongs to an upper non-transversal of }\al}
\]
has size $\xi$.  For each $l\in L$, let $D_l$ be the non-transversal of $\al$ containing $c_l$.  For distinct $l_1,l_2\in L$, we have $D_{l_1}\not=D_{l_2}$ (or else $z_{l_1}$ and $z_{l_2}$ would belong to the same block of $\th_1\al\th_2=\al_1$, a contradiction).  It quickly follows that the non-transversals $\set{C_l}{l\in L}$ and $\set{D_l}{l\in L}$ satisfy the conditions of \ref{it:h06b}.  This completes the proof of the claim.

Returning now to the main proof, note that if \ref{it:h06a} holds, then Lemma \ref{lem:tech02} immediately gives $(\ep_{Y\cup Z},\ep_Y)\in\si$. Thus, for the remainder of the proof, we will assume that \ref{it:h06b} holds.  
Since~${C_l\not=D_l}$ for all $l\in L$, we may assume by symmetry that the set $M=\set{l\in L}{D_l\not\sub C_l}$ has size~$\xi$.  For each $m\in M$, fix some $c_m\in C_m\cap D_m$ and $d_m\in D_m\sm C_m$.  
At this point, let us remember that $M\sub J\sub I$ and $|I\sm J|=\ka$.  It follows that $|I\sm M|=\ka$ as well.  For each $m\in M$, choose some $a_m\in A_m$.  Let $\th_3=\partXXIV{A_i}{c_m}{}{A_i}{c_m}{a_m,d_m}_{i\in I\sm M,\ m\in M}$, and put $(\al_2,\be_2)=(\th_3\al,\th_3\be)\in\si$.  Then $\al_2\cap\be_2$ contains~$\ka$ transversals, namely $A_i\cup B_i'$ for each $i\in I\sm M$, and $\al_2\sm\be_2$ contains $\xi$ transversals, namely $\{c_m\}\cup B_m'$ for each $m\in M$.
Thus, $(\al_2,\be_2)\in\si$ satisfies the conditions of Lemma \ref{lem:tech07}, so applying that lemma now completes the proof.
\epf

\begin{lemma}\label{lem:tech09}
If $(\al,\be)\in\si\restr_{D_\ka}$ where $\ka\geq\aleph_0$, and if $\al\cap\be$ has fewer than $\ka$ transversals, then for any subset $Y\sub X$ with $|Y|=\ka$, we have $(\ep_Y,\emptypart)\in\si$.
\end{lemma}

\pf
Let $Y\sub X$ with $|Y|=\ka$, and write $Y=\set{y_i}{i\in I}$.  Since $\be\in D_\ka$, and since $\be=(\al\cap\be)\cup(\be\sm\al)$, the assumption on transversals of $\al\cap\be$ implies that $\be\sm\al$ contains $\ka$ transversals, say $\set{A_i\cup B_i'}{i\in I}$.  For each $i\in I$, fix some $a_i\in A_i$ and $b_i\in B_i$.  Let $\th_1=\binom{y_i}{a_i}_{i\in I}$ and $\th_2=\binom{b_i}{y_i}_{i\in I}$, and put $\al_1=\th_1\al\th_2$.  Since $\th_1\be\th_2=\ep_Y$, we have $(\ep_Y,\al_1)\in\si$.  If $\al_1$ had fewer than $\ka$ of the transversals of $\ep_Y$, then the desired conclusion would follow from Lemma \ref{lem:tech01}.  Thus, we assume $\al_1$ has $\ka$ such transversals, say $\bigset{\{y_j,y_j'\}}{j\in J}$, where $J\sub I$.  

Lemma \ref{lem:trans1} (applied to the product $\al_1=\th_1\al\th_2$) says that for any $j\in J$, $\al$ contains a transversal $C_j\cup D_j'$ such that $C_j\cap\codom(\th_1)=\{a_j\}$ and $D_j\cap\dom(\th_2)=\{b_j\}$.  As in the last paragraph of the proof of Lemma \ref{lem:tech07}, we may assume that the set $K=\set{j\in J}{C_j\not\sub A_j}$ has size $\ka$.  For each $k\in K$, fix some $c_k\in C_k\sm A_k$.  Also, since $|I|=\ka=|K|$, we may write $I=\set{i_k}{k\in K}$.  Let $\th_3=\binom{y_{i_k}}{c_k}_{k\in K}$ and $\th_4=\binom{b_k}{y_{i_k}}_{k\in K}$, and put $\al_2=\th_3\be\th_4$.  Since $\th_3\al\th_4=\ep_Y$, we have $(\ep_Y,\al_2)\in\si$.  Since $\al_2$ contains no transversals of $\ep_Y$, Lemma \ref{lem:tech01} gives $(\ep_Y,\emptypart)\in\si$.
\epf

Finally, we are now in the position to state and prove the following more general result.

\begin{lemma}\label{lem:tech2}
If $(\al,\be)\in\si\restr_{D_\ka}$ with $\aleph_0\leq|\al\sd\be|=\xi\leq\ka$, then for any disjoint subsets $Y,Z\sub X$ with $|Y|=\ka$ and $|Z|=\xi$, we have $(\ep_{Y\cup Z},\ep_Y)\in\si$.
\end{lemma}

\pf
We identify three possibilities:
\begin{thmenum}
\item \label{it:h2i} $\al\cap\be$ has fewer than $\ka$ transversals, or
\item \label{it:h2ii} $\al\cap\be$ has $\ka$ transversals and $\al\sd\be$ has $\xi$ transversals, or
\item \label{it:h2iii} $\al\cap\be$ has $\ka$ transversals, while $\al\sd\be$ has fewer than $\xi$ transversals but $\xi$ non-transversals.
\end{thmenum}
Lemmas \ref{lem:tech07} and \ref{lem:tech06} deal with cases \ref{it:h2ii} and \ref{it:h2iii}, respectively.  For \ref{it:h2i}, Lemma \ref{lem:tech09} gives $(\ep_{Y\cup Z},\emptypart),(\ep_Y,\emptypart)\in\si$, noting that $|Y\cup Z|=\ka$, so that $(\ep_{Y\cup Z},\ep_Y)\in\si$ by transitivity.
\epf

To prove the main result of this subsection (Lemma \ref{lem:tech3} below), we will also need the next lemma, which provides a modest upper bound for the set $\bigset{|\al\sd\be|}{(\al,\be)\in\si\restr_{D_\ka}}$ where~${\ka\geq\eta=\eta(\si)\geq\aleph_0}$.  
A much stronger bound is exhibited in Lemma \ref{la338}, which is derived as a consequence of Lemma~\ref{lem:tech3}.

\begin{lemma}\label{lem:tech1}
If $(\al,\be)\in\si\restr_{D_\ka}$ where $\kappa\geq\eta=\eta(\si)\geq\aleph_0$, then $|\al\sd\be|<\ka$; consequently, $\al\cap\be$ contains $\ka$ transversals.
\end{lemma}

\pf
It suffices to prove the first assertion; indeed, the second follows from the first, together with the facts that $\al=(\al\cap\be)\cup(\al\sm\be)$, $\al\in D_\ka$ and $\ka\geq\aleph_0$.
To prove the first assertion, suppose to the contrary that $|\al\sd\be|\geq\ka$.  Fix arbitrary disjoint subsets $Y,Z\sub X$ with ${|Y|=|Z|=\ka}$.
One of Lemmas 
\ref{lem:tech07}, \ref{lem:tech06} or \ref{lem:tech09} applies, with $\xi=\kappa$,
and so we have $(\epsilon_{Y\cup Z},\epsilon_Y)\in\sigma$ or $(\epsilon_Z,\epsilon_\emptyset)\in\sigma$.
In fact, the former option implies 
$(\epsilon_Z,\epsilon_\emptyset)=(\epsilon_Z\epsilon_{Y\cup Z},\epsilon_Z\epsilon_Y)\in\sigma$,
so $(\epsilon_Z,\epsilon_\emptyset)\in\sigma$ in all cases.
But this contradicts the definition of $\eta=\eta(\sigma)$ as $(\ep_Z,\emptypart)\in D_\ka\times D_0$ with~$\ka\geq\eta$.
\epf

We may now tie together all the loose ends, and prove the main result of this subsection.  Part~\ref{it:h3ii} of the next lemma is Lemma \ref{lem:YZ}, while part \ref{it:h3i} is a finite analogue that will be of use later on in this section.

\begin{lemma}[cf.~Lemma \ref{lem:YZ}]\label{lem:tech3}
Suppose $(\al,\be)\in\si\restr_{D_\ka}$ where $\ka\geq\eta=\eta(\si)\geq\aleph_0$ and ${|\al\sd\be|\not=0}$.
\begin{thmenum}
\item \label{it:h3i} For any disjoint subsets $Y,Z\sub X$ with $|Y|\leq\ka$ and $|Z|<\aleph_0$, we have $(\ep_{Y\cup Z},\ep_Y)\in\si$.
\item \label{it:h3ii}For any disjoint subsets $Y,Z\sub X$ with $|Y|\leq\ka$ and $|Z|\leq|\al\sd\be|$, we have $(\ep_{Y\cup Z},\ep_Y)\in\si$.
\end{thmenum}
\end{lemma}

\begin{proof}
We first note that it suffices to prove the result 
assuming $|Y|=\kappa$ throughout and $|Z|=|\alpha\sd\beta |$
in part \ref{it:h3ii}.  
Indeed, if $(\ep_{Y\cup Z},\ep_Y)\in\si$ for some disjoint subsets $Y,Z\sub X$, then for any disjoint $U,V\sub X$ with $|U|\leq|Y|$ and $|V|\leq|Z|$, we fix injections $\psi\colon U\to Y$ and $\phi\colon V\to Z$, define the partitions $\th_1=\partpermII uv{u\psi}{v\phi}_{u\in U,\ v\in V}$ and $\th_2=\partpermII {u\psi}{v\phi}uv_{u\in U,\ v\in V}$, and obtain $(\ep_{U\cup V},\ep_U)=(\th_1\ep_{Y\cup Z}\th_2,\th_1\ep_Y\th_2)\in\si$.

\pfitem{\ref{it:h3i}}  Observe that it suffices to assume that $|Z|=1$.  Indeed, if the result is true for $|Z|=1$, and if $W=\{w_1,\ldots,w_n\}\sub X$ is disjoint from $Y$, then writing $Z_i=Y\cup\{w_1,\ldots,w_i\}$ for each $0\leq i\leq n$, we have $(\ep_{Z_0},\ep_{Z_1}),(\ep_{Z_1},\ep_{Z_2}),\ldots,(\ep_{Z_{n-1}},\ep_{Z_n})\in\si$, at which point transitivity gives $(\ep_Y,\ep_{Y\cup W})=(\ep_{Z_0},\ep_{Z_n})\in\si$.

So, now, write $Y=\set{y_i}{i\in I}$ and $Z=\{z\}$.  By Lemma~\ref{lem:tech1}, $\al\cap\be$ contains $\ka$ transversals, say $\set{A_i\cup B_i'}{i\in I}$.  By symmetry, we may assume that either $\al\sm\be$ contains a transversal, or else $\al\sd\be$ contains no transversals but $\al\sm\be$ contains an upper non-transversal.

\pfcase{1}
Suppose first that $\al\sm\be$ contains a transversal, say $A\cup B'$.  If this transversal is not properly contained in any transversal of $\be$, then we choose some $a\in A$ and $b'\in B'$ such that they do not belong to the same block of $\be$. Then with $\th_1=\partpermII{y_i}z{A_i}a_{i\in I}$ and $\th_2=\partpermII{B_i}b{y_i}z_{i\in I}$, we have $(\ep_{Y\cup Z},\ep_Y)=(\th_1\al\th_2,\th_1\be\th_2)\in\si$.  If $A\cup B'$ is properly contained in some transversal $C\cup D'$ of~$\be$, then $C\cup D'$ is not properly contained in any transversal of $\al$, and we can we then repeat the previous argument with the roles of $\al$ and $\be$ reversed.

\pfcase{2}
Now suppose $\al\sd\be$ contains no transversals but $\al\sm\be$ contains an upper non-transversal, say~$C$.  Renaming $\al,\be$ if necessary (if $C$ is a proper subset of a non-transversal of $\be$), we may assume there exists $c,d\in C$ such that $c$ and $d$ belong to distinct blocks of $\be$.  Fix some $i\in I$ and some $a\in A_i$.  Let $\th=\partXXIV{A_j}c{}{A_j}c{a,d}_{j\in I\sm\{i\}}$, and put $(\al_1,\be_1)=(\th\al,\th\be)\in\si\restr_{D_\ka}$.  Then $\al_1\sm\be_1$ contains the transversal $\{c\}\cup B_i'$, so we have reduced to Case 1.

\pfitem{\ref{it:h3ii}}
If $\xi=|\al\sd\be|$ is finite, we apply part \ref{it:h3i}; if $\xi$ is infinite, then we apply Lemma \ref{lem:tech2}, keeping in mind that Lemma \ref{lem:tech1} gives $\xi<\ka$.
\epf

\subsection{The lattice of partitions}\label{subsect:part_lat}

In the remainder of Section \ref{sect:tech}, it will often be convenient to make use of two additional operations on $\P_X$, which we denote by $\wedge$ and $\vee$.  These come from the fact that partitions (of arbitrary sets) have a natural lattice order;
see, for example, \cite[Section V.4]{MR2768581}.  We briefly review the relevant concepts here.

If $\al$ and $\be$ are partitions of some set (such as $X$ or $X\cup X'$), we write $\al\pre\be$ to indicate that~$\al$ \emph{refines} $\be$, meaning that every block of $\al$ is contained in a block of $\be$; this is the same as saying that the equivalence relation corresponding to $\al$ is contained in that corresponding to $\be$.  
Note that if $\alpha\pre\beta$ then $|\alpha\setminus\beta |\geq |\beta\setminus\alpha |$;
in particular, if $|\alpha\sd\beta | $ is infinite then $|\alpha\sd\beta |=|\alpha\setminus\beta |$.
Since inclusion is a lattice ordering on equivalences, we have natural meet and join operations on partitions; 
we write $\al\wedge\be$  for the greatest partition~$\ga$ satisfying $\ga\pre\al,\be$, and 
$\al\vee\be$ for the least $\gamma$ satisfying $\al,\be\pre\ga$.

We will need the following lemma on a number of occasions.

\begin{lemma}\label{lem:joins}
Let $\al$ and $\be$ be partitions of some set, and let $\th$ be either of 
$\al\vee\be$ or $\al\wedge\be$.  Then
\begin{thmenum}
\item \label{it:joinsi} $|\al\sd\th|,|\be\sd\th|\leq|\al\sd\be|\leq|\al\sd\th|+|\be\sd\th|$,
\item \label{it:joinsii} $|\al\sd\be| = {\max}\big(|\al\sd\th|,|\be\sd\th|\big)$ if $|\al\sd\be|\geq\aleph_0$.
\end{thmenum}
\end{lemma}

\begin{proof}
We just prove \ref{it:joinsi}, as \ref{it:joinsii} quickly follows.  First note that 
\[
\alpha\setminus\theta \subseteq \alpha\setminus\beta\subseteq (\alpha\setminus\theta)\cup(\theta\setminus\beta)
\AND
\beta\setminus\theta \subseteq \beta\setminus\alpha\subseteq (\beta\setminus\theta)\cup(\theta\setminus\alpha) ,
\]
so that
\[
|\alpha\setminus\theta| \leq |\alpha\setminus\beta| \leq |\alpha\setminus\theta|+|\theta\setminus\beta|
\AND
|\beta\setminus\theta| \leq |\beta\setminus\alpha| \leq |\beta\setminus\theta|+|\theta\setminus\alpha|.
\]
Furthermore, we have $|\theta\setminus\alpha |\leq |\beta\setminus\alpha |$.
Indeed, if $\theta=\alpha\vee\beta$ then
every block of $\theta\setminus\alpha$ is a union of blocks from $\be\sm\al$,
while if $\theta=\alpha\wedge\beta$ then every block of $\theta\setminus\alpha$ is a subset of some block of~$\be\sm\al$.
Dually, $|\theta\setminus\beta |\leq |\alpha\setminus \beta |$,
and the result follows.
\end{proof}

While the partition monoid $\P_X$ is a lattice under the above operations $\wedge$ and $\vee$, the partial Brauer monoid $\PB_X$ is closed under $\wedge$ but not $\vee$.  
Note that $\emptypart$ is the $\pre$-least element in both~$\P_X$ and $\PB_X$, while $\binom XX$ is the $\pre$-greatest element in $\P_X$.  There is no $\pre$-greatest element in $\PB_X$, though there are many $\pre$-maximal elements; these are precisely the partitions from  $\PB_X$ with at most one singleton block.
The $\wedge$ operation on $\PB_X$ will be used extensively in Subsections~\ref{subsect:PBX1} and \ref{subsect:PBX2}, and the $\vee$ operation on $\P_X$ in Subsections \ref{subsect:PX1} and \ref{subsect:PX2}.

We record here the following obvious fact concerning partitions from $\P_X$ of rank $0$:

\begin{lemma}
\label{lem:Jr0}
For $\alpha,\beta\in\P_X$ of rank $0$, we have 
\[
\epfreseq
\overline{\alpha\wedge\beta}=\overline{\alpha}\wedge\overline{\beta} \COMMA
\underline{\alpha\wedge\beta}=\underline{\alpha}\wedge\underline{\beta} \COMMA
\overline{\alpha\vee\beta}=\overline{\alpha}\vee\overline{\beta} \COMMA
\underline{\alpha\vee\beta}=\underline{\alpha}\vee\underline{\beta}.  
\]
\end{lemma}

We mention in passing the following compatibility result, even though it will not be needed subsequently.
It was proved in \cite[Lemma 6.1]{FitzGerald2013}, and can be seen directly 
using the definition of product as $\al\be = (\al_\downarrow\vee\be^\uparrow)\restr_{X\cup X'}$ (cf.~Subsection \ref{subsect:M}), and the fact that refinement of partitions corresponds to inclusion of equivalences.

\begin{prop}\label{prop:ordered}
The order $\pre$ is compatible with multiplication in the monoid $\M$ (standing for $\P_X$ or $\PB_X$ for any set $X$), meaning that 
\begin{align*}
[\al_1\pre\al_2 \text{ and } \be_1\pre\be_2] &\implies \al_1\be_1\pre\al_2\be_2 &&\text{for all $\al_1,\al_2,\be_1,\be_2\in\M$.}
\epfreseq
\end{align*}
\end{prop}

As a result of Proposition \ref{prop:ordered}, we may think of $\M$ as an ordered monoid: i.e., an algebra of type $(\M,\cdot,\pre)$; cf.~\cite[Chapter 11]{Blyth2005}.  In fact, since $\pre$ is also obviously compatible with the involution (i.e., $\al\pre\be\implies\al^*\pre\be^*$), we may think of $\M$ as an ordered $*$-monoid $(\M,\cdot,{}^*,\pre)$.  These structures could be further enhanced by adding the $\wedge$ and $\vee$ operations in the case $\M=\P_X$, or just the $\wedge$ operation in the case $\M=\PB_X$.  It would be interesting to study these enhanced algebraic structures in their own right (cf.~\cite{Maddux2006,HJM2016,JS2011}), but this is beyond the scope of the current article.  
However, we will make one last observation before moving on.  In light of Proposition~\ref{prop:ordered}, 
we have inequalities such as
\[
\th(\al\wedge\be)\pre(\th\al)\wedge(\th\be) \AND \th(\al\vee\be)\succeq(\th\al)\vee(\th\be),
\]
but these can be strict.  For example in $\P_3$, 
\[
\th(\al\wedge\be)\not=(\th\al)\wedge(\th\be) \qquad\text{for}\qquad
\th = \custpartn{1,2,3}{1,2,3}{\stline21\stline23} \COMMa \al = \custpartn{1,2,3}{1,2,3}{\stline12\stline22} \COMMa \be = \custpartn{1,2,3}{1,2,3}{\stline22\stline32} , 
\]
while
\[
\th(\al\vee\be)\not=(\th\al)\vee(\th\be) \qquad\text{for}\qquad
\th = \custpartn{1,2,3}{1,2,3}{\stline11} \COMMa \al = \custpartn{1,2,3}{1,2,3}{\stline22} \COMMa \be = \custpartn{1,2,3}{1,2,3}{\stline12\stline23} .
\]

\subsection[Lemmas \ref{la328a} and \ref{la328c} for $\PB_X$]{\boldmath Lemmas \ref{la328a} and \ref{la328c} for $\PB_X$}\label{subsect:PBX1}

This subsection and the next exclusively concern the partial Brauer monoid $\PB_X$, where $X$ is infinite.  
Throughout this subsection, $\si$ denotes an arbitrary congruence on $\PB_X$.  By an \emph{upper} or \emph{lower hook} we mean a two-element subset of $X$ or of $X'$, respectively.  Unless otherwise stated, when we refer to a set simply as a \emph{hook}, we mean an \emph{upper} hook.

\begin{lemma}[cf.~Lemma \ref{la328a}]
\label{la-tt5}
Suppose $(\alpha,\beta)\in \sigma\restr_{D_0}$ with $\overline{\alpha}\neq\overline{\beta}$.
\begin{thmenum}
\item \label{it:tt5i}
For any $\gamma,\delta\in D_0$ with $|\overline{\gamma}\sd\overline{\delta}|<\aleph_0$, $\ol\ga\pre\ol\de$ and $\underline{\gamma}=\underline{\delta}$, we have $(\gamma,\delta)\in\sigma$.
\item \label{it:tt5ii}
For any $\gamma,\delta\in D_0$ with $|\overline{\gamma}\sd\overline{\delta}|<\aleph_0$ and $\underline{\gamma}=\underline{\delta}$, we have $(\gamma,\delta)\in\sigma$.
\end{thmenum}
\end{lemma}

\pf
\ref{it:tt5i}  Since $|\ol\ga\sd\ol\de|<\aleph_0$, there is a sequence $\ga=\ga_0\pre\ga_1\pre\cdots\pre\ga_k=\de$ such that each $\ul\ga_i=\ul\ga=\ul\de$, and $\ga_i$ has all but one of the (upper) hooks of $\ga_{i+1}$ for each $i$.  In light of this, it suffices inductively to assume that~$\ga$ has all but one of the hooks of $\de$.  Let the hooks of $\ga$ be $\set{A_i}{i\in I}$, and let the additional hook of $\de$ be $\{x,y\}$.  Since $\ol\al\not=\ol\be$, we may assume without loss of generality that $\be$ has some hook $\{u,v\}$ that is not a hook of $\al$.  Then with $\th=\partXXIV xy{A_i}uv{}_{i\in I}$, we have $(\ga,\de)=(\th\al\ga,\th\be\ga)\in\si$.

\pfitem{\ref{it:tt5ii}}  Let $\th=\ga\wedge\de\in D_0$.  Since $\ol\th=\ol\ga\wedge\ol\de$ by Lemma \ref{lem:Jr0}, we have $\ol\th\pre\ol\ga$.  By Lemma \ref{lem:joins} \ref{it:joinsi}, we also have $|\ol\ga\sd\ol\th| = |\ol\ga\sd(\ol\ga\wedge\ol\de)|\leq|\ol\ga\sd\ol\de|<\aleph_0$.  Again by Lemma \ref{lem:Jr0} we have $\ul\th=\ul\ga\wedge\ul\de=\ul\ga$.  Thus, part \ref{it:tt5i} gives $(\ga,\th)\in\si$.  By symmetry $(\de,\th)\in\si$, and by transitivity $(\ga,\de)\in\si$.
\epf

\newpage

\begin{lemma}[cf.~Lemma \ref{la328c}]
\label{la-tt6}
Suppose $(\alpha,\beta)\in\sigma\restr_{D_0}$ with $|\overline{\alpha}\sd\overline{\beta}|=\xi\geq\aleph_0$. 
\begin{thmenum}
\item \label{it:tt6i}
For any $\gamma,\delta\in D_0$ with $|\overline{\gamma}\sd\overline{\delta}|\leq\xi$, $\ol\ga\pre\ol\de$ and $\underline{\gamma}=\underline{\delta}$, we have $(\gamma,\delta)\in\sigma$.
\item \label{it:tt6ii}
For any $\gamma,\delta\in D_0$ with $|\overline{\gamma}\sd\overline{\delta}|\leq\xi$ and $\underline{\gamma}=\underline{\delta}$, we have $(\gamma,\delta)\in\sigma$.
\end{thmenum}
\end{lemma}

\pf \ref{it:tt6i}
Let the (upper) hooks of $\ga$ be $\set{A_i}{i\in I}$, and the remaining hooks of $\de$ be ${\set{B_j}{j\in J}}$.  We will write $\ka=|J|$.  Since $\ol\ga\pre\ol\de$ we have $3\ka=|\ol\ga\sd\ol\de|\leq\xi$; since $\xi\geq\aleph_0$ it follows that $\ka\leq\xi$.  
It follows from $|\overline{\alpha}\sd\overline{\beta}|=\xi\geq\aleph_0$ that at least one of $\overline{\alpha}\setminus\overline{\beta}$ or $\overline{\beta}\setminus\overline{\alpha}$ contains~$\xi$ hooks.
Without loss of generality, we assume this is the case for $\overline{\be}\setminus\overline{\al}$, and we write~$\HH$ for the set of hooks in $\overline{\be}\setminus\overline{\al}$.  
Let $\Ga$ be the graph with vertex set $\HH$, such that there is an edge between~$H_1,H_2\in\HH$ if there is at least one hook of $\al$ with one vertex from $H_1$ and one from $H_2$.  
Then~$\Ga$ satisfies the conditions of Lemma \ref{la-tt4}, and therefore has an independent set of size $\xi$.
Within this independent set we fix a subset of size $\ka$, say $\bigset{H_j}{ j\in J}$.
Now write~$B_j=\{x_j,y_j\}$ and $H_j=\{u_j,v_j\}$ for each~$j$.  Then with $\th=\partXXIV {x_j}{y_j}{A_i}{u_j}{v_j}{}_{i\in I,\ j\in J}$, we have $(\ga,\de)=(\th\al\ga,\th\be\ga)\in\si$.

\pfitem{\ref{it:tt6ii}}  
This follows from \ref{it:tt6i} in the same way that Lemma \ref{la-tt5} \ref{it:tt5ii} follows from Lemma \ref{la-tt5} \ref{it:tt5i}.
\epf

\subsection[Lemmas \ref{la340b} and \ref{la340} for $\PB_X$]{\boldmath Lemmas \ref{la340b} and \ref{la340} for $\PB_X$}\label{subsect:PBX2}

In the following two lemmas $\si$ denotes an arbitrary congruence on $\PB_X$ with $\eta=\eta(\si)\geq\aleph_0$.

\begin{lemma}[cf.~Lemma \ref{la340b}]
\label{la-sw4}
If $(\alpha,\beta)\in\sigma\restr_{D_\kappa}$ where $\kappa\geq\eta\geq\aleph_0$ and $\alpha\neq\beta$, then for any $\gamma,\delta\in D_\kappa$ with $|\gamma\sd\delta|<\aleph_0$, we have $(\gamma,\delta)\in \sigma$.
\end{lemma}

\begin{proof}
We will show that $(\ga,\ga\wedge\de)\in\si$; by symmetry, it will follow that $(\de,\ga\wedge\de)\in\si$, and then by transitivity that $(\ga,\de)\in\si$.  Suppose the transversals, upper hooks and lower hooks of $\ga\cap\de$ are $\bigset{\{a_i,b_i'\}}{i\in I}$, $\set{A_j}{j\in J}$ and $\set{B_k'}{k\in K}$; by Lemma \ref{lem:tech1}, we have~$|I|=\ka$.  Put $Y_1=\set{a_i}{i\in I}$ and $Y_2=\set{b_i}{i\in I}$, noting that these are both of size $\ka$.  Let~$Z$ be the union of all the blocks from $\ga\sm\de$, noting that $Z$ is finite, and write $Z=Z_1\cup Z_2'$ where $Z_1,Z_2\sub X$.  Finally, let $W_1=Y_1\cup Z_1$ and $W_2=Y_2\cup Z_2$.  Then Lemma~\ref{lem:tech3}~\ref{it:h3i} gives $(\ep_{Y_1},\ep_{W_1}),(\ep_{Y_2},\ep_{W_2})\in\si$.  Then with $\th_1=\partn w{A_j}w{}_{w\in W_1,\, j\in J}$ and $\th_2=\partn w{}w{B_k}_{w\in W_2,\, k\in K}$, we have $(\ga,\ga\wedge\de) = (\th_1\ep_{W_1}\ga\ep_{W_2}\th_2,\th_1\ep_{Y_1}\ga\ep_{Y_2}\th_2)\in\si$, as required.
\end{proof}

\begin{lemma}[cf.~Lemma \ref{la340}]
\label{la-sw5}
If $(\alpha,\beta)\in \sigma\restr_{D_\kappa}$ where $\kappa\geq\eta\geq\aleph_0$, then for any $\gamma,\delta\in D_\kappa$ with $|\gamma\sd\delta|\leq|\alpha\sd\beta|$, we have $(\gamma,\delta)\in \sigma$.
\end{lemma}

\begin{proof}
The case of finite $|\gamma\sd\delta|$ is covered by Lemma \ref{la-sw4}.  The infinite case is proved in exactly the same way as Lemma \ref{la-sw4}; this time rather than $Z$ being finite, we have $|Z|\leq|\gamma\sd\delta|$.
\end{proof}

\subsection[Lemmas \ref{la328a} and \ref{la328c} for $\P_X$]{\boldmath Lemmas \ref{la328a} and \ref{la328c} for $\P_X$}\label{subsect:PX1}

This subsection and the next exclusively concern the partition monoid $\P_X$, where $X$ is infinite.  Throughout this subsection, $\si$ denotes an arbitrary congruence on $\P_X$.

For any non-empty subset $Y\subseteq X$ we will write~$\upsilon_Y=\partXX Y{}\in\P_X$; so $\ups_Y$ has rank $0$, has $Y$ as a block, with all other blocks being singletons (note that $\ups_Y=\emptypart$ if~$|Y|=1$).  By a \emph{disjoint family of subsets} of $X$, we mean a collection $\Y=\set{Y_i}{i\in I}$, where the $Y_i$ are pairwise disjoint non-empty subsets of $X$; for such a family $\Y$, we write $\ups_\Y=\partXX {Y_i}{}_{i\in I}$.

\newpage

\begin{lemma}[cf.~Lemma \ref{la328a}]
\label{la-tt1}
Suppose $(\alpha,\beta)\in \sigma\restr_{D_0}$ with $\overline{\alpha}\neq\overline{\beta}$.
\begin{thmenum}
\item \label{it:tt1i}
 For any two-element subset $Y\subseteq X$, we have $(\upsilon_Y,\emptypart)\in\sigma$.
\item \label{it:tt1ii}
 For any finite subset $Y\subseteq X$, we have $(\upsilon_Y,\emptypart)\in\sigma$.
\item \label{it:tt1iii}
 For any finite disjoint family $\Y$ of finite subsets of $X$, we have $(\ups_\Y,\emptypart)\in\si$.
\item \label{it:tt1iv}
For any $\gamma,\delta\in D_0$ with $|\overline{\gamma}\sd\overline{\delta}|<\aleph_0$, $\ol\ga\pre\ol\de$ and $\underline{\gamma}=\underline{\delta}$, we have $(\gamma,\delta)\in\sigma$.
\item \label{it:tt1v}
For any $\gamma,\delta\in D_0$ with $|\overline{\gamma}\sd\overline{\delta}|<\aleph_0$ and $\underline{\gamma}=\underline{\delta}$, we have $(\gamma,\delta)\in\sigma$.
\end{thmenum}
\end{lemma}

\begin{proof}
\ref{it:tt1i}
Write $Y=\{ y_1,y_2\}$.
Without loss of generality, we may assume that there exist distinct $x_1,x_2\in X$ that belong to the same block of $\overline{\alpha}$ but to different blocks of $\overline{\beta}$.
Then with $\th=\partpermII{y_1}{y_2}{x_1}{x_2}$,
we have $(\upsilon_Y,\emptypart)=(\theta \alpha\emptypart,\theta\beta\emptypart)\in\sigma$.

\pfitem{\ref{it:tt1ii}}
We use induction on $n=|Y|$.
For $n=1$ there is nothing to prove, and $n=2$ is part~\ref{it:tt1i}.
So suppose $n\geq 3$ and that the assertion holds for all subsets of $X$ of size less than~$n$.
Let $Y=\{y_1,\dots,y_n\}$ be an arbitrary subset of size $n$.
Further, let $Y_1=\{y_1,\dots,y_{n-1}\}$ and ${Y_2=\{y_{n-1},y_n\}}$.
By induction, we have $(\upsilon_{Y_1},\emptypart),(\upsilon_{Y_2},\emptypart)\in\sigma$; transitivity then gives ${(\ups_{Y_1},\ups_{Y_2})\in\si}$.
Then with $\th=\partpermII{Y_1}{y_n}{Y_1}{y_n}$,
we have
$(\upsilon_{Y_1},\upsilon_Y)=(\theta\upsilon_{Y_1},\theta\upsilon_{Y_2})\in\sigma$.  Another appeal to transitivity gives $(\upsilon_Y,\emptypart)\in\sigma$.

\pfitem{\ref{it:tt1iii}}
Write $\Y=\set{Y_i}{i\in I}$ and put $Y=\bigcup_{i\in I} Y_i$; since $|Y|<\aleph_0$, part \ref{it:tt1ii} gives $(\upsilon_Y,\emptypart)\in\sigma$.  Then with  $\theta=\binom{Y_i}{Y_i}_{i\in I}$, we have $( \upsilon_Y,\upsilon_\Y)=( \theta\upsilon_Y,\theta\emptypart)\in \sigma$, and hence $( \upsilon_\Y,\emptypart)\in\sigma$ by transitivity.

\pfitem{\ref{it:tt1iv}}
Write $\ol\ga\cap\ol\de=\set{A_i}{i\in I}$ and $\ol\de\sm\ol\ga=\set{B_j}{j\in J}$, noting that $J$ is finite.  For each $j\in J$, let the blocks of $\ol\ga$ contained in~$B_j$ be $\set{C_{jk}}{k\in K_j}$, again noting that each $K_j$ is finite.  Let $\Y=\set{Y_j}{j\in J}$, where the sets~$Y_j=\set{y_{jk}}{k\in K_j}\sub X$ are pairwise disjoint.  By \ref{it:tt1iii}, we have $(\ups_\Y,\emptypart)\in\si$.  Then with $\th=\partn{C_{jk}}{A_i}{y_{jk}}{}_{i\in I,\ j\in J,\ k\in K_j}$, we have $(\gamma,\delta)=(\theta\emptypart\ga,\theta\upsilon_\Y\ga)\in\sigma$.

\pfitem{\ref{it:tt1v}}
Let $\th=\ga\vee\de\in D_0$.  Since $\ol\th=\ol\ga\vee\ol\de$ by Lemma \ref{lem:Jr0}, we have $\ol\ga\pre\ol\th$.  By Lemma~\ref{lem:joins}~\ref{it:joinsi}, we also have $|\ol\ga\sd\ol\th| = |\ol\ga\sd(\ol\ga\vee\ol\de)|\leq|\ol\ga\sd\ol\de|<\aleph_0$.  Again by Lemma \ref{lem:Jr0} we have $\ul\th=\ul\ga\vee\ul\de=\ul\ga$.  Thus, part \ref{it:tt1iv} gives $(\ga,\th)\in\si$.  By symmetry $(\de,\th)\in\si$, and by transitivity $(\ga,\de)\in\si$.
\end{proof}

\begin{lemma}[cf.~Lemma \ref{la328c}]
\label{la-tt2}
Suppose $(\alpha,\beta)\in\sigma\restr_{D_0}$ with $|\overline{\alpha}\sd\overline{\beta}|=\xi\geq\aleph_0$. 
\begin{thmenum}
\item \label{it:tt2i} There exists $(\al_1,\beta_1)\in\si\restr_{D_0}$ such that $|\overline{\alpha}_1\sd\overline{\beta}_1|=\xi$ and $\overline{\alpha}_1\pre\overline{\beta}_1$.
\item \label{it:tt2ii} For any subset $Y\subseteq X$ of size $\xi$, we have $(\upsilon_Y,\emptypart)\in\sigma$.
\item \label{it:tt2iii} For any disjoint family $\Y$ of $\xi$ subsets of $X$ each of size $\xi$, we have $(\ups_\Y,\emptypart)\in\si$.
\item \label{it:tt2iv} For any  $\gamma,\delta\in D_0$ with $|\overline{\gamma}\sd\overline{\delta}|\leq\xi$, $\ol\ga\pre\ol\de$ and $\underline{\gamma}=\underline{\delta}$, we have $(\gamma,\delta)\in\sigma$.
\item  \label{it:tt2v} For any  $\gamma,\delta\in D_0$ with $|\overline{\gamma}\sd\overline{\delta}|\leq\xi$ and $\underline{\gamma}=\underline{\delta}$, we have $(\gamma,\delta)\in\sigma$.
\end{thmenum}
\end{lemma}

\begin{proof}
\ref{it:tt2i}
Let $\th=\al\vee\be$, noting that $\ol\al,\ol\be\pre\ol\al\vee\ol\be=\ol\th$ by Lemma \ref{lem:Jr0}.  By Lemma \ref{lem:joins} \ref{it:joinsii}, and since $|\ol\al\sd\ol\be|=\xi\geq\aleph_0$, we may assume without loss of generality that $|\ol\al\sd\ol\th|=\xi$.  
Write $\ol\al=\set{A_i}{i\in I}$, and with $\th_1=\binom{A_i}{A_i}_{i\in I}$ define $(\al_1,\be_1)=(\th_1\al,\th_1\be)\in\si$.  Then $\al_1,\be_1\in D_0$, and also $\ol\al_1=\ol\al$ and $\ol\be_1=\ol\th$, so that $|\ol\al_1\sd\ol\be_1|=\xi$ and $\ol\al_1\pre\ol\be_1$.

\pfitem{\ref{it:tt2ii}}
By \ref{it:tt2i}, we may assume that $\ol\al\pre\ol\be$, and we note that $\xi=|\ol\al\sd\ol\be|=|\ol\al\sm\ol\be|$.  Let $Z$ be a set that contains precisely one element of each block of $\ol\al\sm\ol\be$.  Since $\ol\al\pre\ol\be$, we may fix a subset $Z_1\sub Z$ containing precisely one element of each block of $\ol\be\sm\ol\al$.  We also put $Z_2=Z\sm Z_1$.  Since every block of $\overline{\beta}\setminus\overline{\alpha}$ contains at least two blocks of $\overline{\alpha}\setminus\overline{\beta}$, it follows that $|Z_2|\geq|Z_1|$, and so $\xi=|\ol\al\sm\ol\be|=|Z|=|Z_1|+|Z_2|={\max}\big(|Z_1|,|Z_2|\big)=|Z_2|$.  We may therefore fix some bijection $\psi\colon Y\to Z_2$.  Then with $\th=\partn y{}{y\psi}{Z_1}_{y\in Y}$, we have $(\upsilon_Y,\emptypart)=(\theta\beta\emptypart,\theta\alpha\emptypart)\in\sigma$.

\pfitem{\ref{it:tt2iii}}
This is essentially identical to Lemma \ref{la-tt1} \ref{it:tt1iii}.

\pfitem{\ref{it:tt2iv}}
Let $\Y=\set{Y_i}{i\in I}$ where $|I|=\xi=|Y_i|$ for all $i\in I$, so that $(\ups_\Y,\emptypart)\in\sigma$ by part \ref{it:tt2iii}.  Since $|\ol\ga\sd\ol\de|\leq\xi=|I|$, we may write $\ol\de\sm\ol\ga=\set{D_j}{j\in J}$, where $J\sub I$.  For $j\in J$, let the blocks of $\ol\ga$ contained in $D_j$ be $\set{C_{jk}}{k\in K_j}$; since $|K_j|\leq|\ol\ga\sd\ol\de|\leq\xi$, we may fix an injective map $\psi_j\colon K_j\to Y_j$ for each $j$ (recall that $J\sub I$).  Further, let $\ol\ga\cap\ol\de=\set{E_l}{l\in L}$.  
Then with
$
\th=\partn{C_{jk}}{E_l}{k\psi_j}{}_{j\in J,\ k\in K_j,\ l\in L},
$
we have $(\ga,\de)=(\th\emptypart\ga,\th\ups_\Y\ga)\in\si$, as required.

\pfitem{\ref{it:tt2v}}
This is essentially identical to Lemma \ref{la-tt1} \ref{it:tt1v}.
\end{proof}

\subsection[Lemmas \ref{la340b} and \ref{la340} for $\P_X$]{\boldmath Lemmas \ref{la340b} and \ref{la340} for $\P_X$}\label{subsect:PX2}

In the following two lemmas, $\si$ denotes an arbitrary congruence on $\P_X$ with $\eta=\eta(\si)\geq\aleph_0$.

For two disjoint sets $Y,Z\subseteq X$, let $\epsilon_{Y,Z}=\epsilon_{Y\cup Z}$, and
let $\omega_{Y,Z}=\partpermII yZyZ_{y\in Y}$.
More generally, for a disjoint family $\Z=\set{ Z_i}{i\in I}$ of subsets of $X$, all of whose members are also disjoint from $Y$,  let
$\epsilon_{Y,\Z}=\epsilon_{Y,Z}$ with $Z=\bigcup_{i\in I}Z_i$, and let
$\omega_{Y,\Z}=\partpermII y{Z_i}y{Z_i}_{y\in Y,\ i\in I}$.

\begin{lemma}[cf.~Lemma \ref{la340b}]
\label{la-rg1}
Suppose $(\alpha,\beta)\in\sigma\restr_{D_\kappa}$ where $\kappa\geq\eta\geq\aleph_0$ and $\alpha\neq\beta$. 
\begin{thmenum}
\item \label{it:rg1i}
For any disjoint $Y,Z\subseteq X$ with $|Y|=\kappa$ and $|Z|<\aleph_0$, 
we have $(\epsilon_{Y,Z},\omega_{Y,Z})\in\sigma$.
\item \label{it:rg1ii}
For any $Y\subseteq X$ of size $\kappa$, and any finite disjoint family $\Z$ of finite subsets of $X$, all of them disjoint from $Y$, we have
$(\epsilon_{Y,\Z},\omega_{Y,\Z})\in\sigma$.
\item \label{it:rg1iii}
For any $\gamma,\delta\in D_\kappa$ with $|\gamma\sd\delta|<\aleph_0$ and $\ga\pre\de$, we have $(\gamma,\delta)\in\sigma$.
\item \label{it:rg1iv}
For any $\gamma,\delta\in D_\kappa$ with $|\gamma\sd\delta|<\aleph_0$, we have $(\gamma,\delta)\in\sigma$.
\end{thmenum}
\end{lemma}

\begin{proof}
\ref{it:rg1i}
By Lemma \ref{lem:tech3} \ref{it:h3i}, we may assume that $(\al,\be)=(\ep_Y,\ep_{Y,Z})$.  We must show that $(\be,\om)\in\si$, where for brevity we write $\om=\om_{Y,Z}$.  But $(\al,\om)=(\al\om\al,\be\om\be)\in\si$, and so transitivity gives $(\be,\om)\in\si$.

\pfitem{\ref{it:rg1ii}}
Write $\Z=\set{Z_i}{i\in I}$, $Z=\bigcup_{i\in I}Z_i$ and $\om=\om_{Y,\Z}$.  By part \ref{it:rg1i}, we may assume that $(\al,\be)=(\ep_{Y,\Z},\om_{Y,Z})$.  This time, we have $(\om,\be)=(\al\om\al,\be\om\be)\in\si$.  Transitivity gives $(\al,\om)\in\si$.

\pfitem{\ref{it:rg1iii}}
Let us first index the various blocks of $\gamma$ and $\delta$; in what follows, all indexing sets ($I^T$, $I^U$, and so on) are assumed to be pairwise disjoint.  First, for the blocks of $\delta\setminus\gamma$, we assume
\begin{itemize}
\item
the transversals in $\delta\setminus\gamma$ are $\set{A_i\cup B_i'}{i\in I^T}$,
\item
the upper non-transversals in $\delta\setminus\gamma$ are $\set{A_i}{i\in I^U}$,
\item
the lower non-transversals in $\delta\setminus\gamma$ are $\set{B_i'}{i\in I^L}$.
\end{itemize}
Recall that the blocks of $\gamma$ are contained in blocks of $\delta$. So, for each $i\in I^T$, we assume
\begin{itemize}
\item
the transversals of $\gamma$ contained in $A_i\cup B_i'$ are $\set{A_{ij}\cup B_{ij}'}{j\in J^{TT}_i}$,
\item
the upper non-transversals of $\gamma$ contained in $A_i\cup B_i'$ are $\set{A_{ij}}{j\in J^{TU}_i}$,
\item
the lower non-transversals of $\gamma$ contained in $A_i\cup B_i'$ are $\set{B_{ij}'}{j\in J^{TL}_i}$.
\end{itemize}
Likewise, for $i\in I^U$, we assume
\begin{itemize}
\item
the upper non-transversals of $\gamma$ contained in $A_i$ are $\set{A_{ij}}{j\in J^{UU}_i}$,
\end{itemize}
while for $i\in I^L$, we assume
\begin{itemize}
\item
the lower non-transversals of $\gamma$ contained in $B_i'$ are $\set{B_{ij}'}{j\in J^{LL}_i}$.
\end{itemize}
Finally, we assume
\begin{itemize}
\item
the transversals of $\ga\cap\de$ are $\set{C_k\cup D_k'}{k\in K^T}$,
\item
the upper non-transversals in $\ga\cap\de$ are $\set{C_k}{k\in K^U}$,
\item
the lower non-transversals in $\ga\cap\de$ are $\set{D_k'}{k\in K^L}$.
\end{itemize}
Note that all $I$- and $J$-type index sets are finite since $|\gamma\sd\delta|<\aleph_0$, and therefore $|K^T|=\kappa$, because $\gamma$ and $\delta$ have rank $\kappa\geq\aleph_0$.

Let $P^U$ denote the set of all pairs $(i,j)$ for which there is a block $A_{ij}$: i.e.,
\[
P^U=\bigset{ (i,j) } { [i\in I^T \text{ and } j\in J^{TT}_i\cup J^{TU}_i]\mbox{ or }
[i\in I^U \text{ and } j\in J^{UU}_i]}.
\]
Define $P^L$ analogously with respect to $B_{ij}$ blocks, noting that
\[
P^U\cap P^L=\bigset{(i,j)}{i\in I^T,\ j\in J_i^{TT}}.
\]
Let $P=P^U\cup P^L$, again noting that $P$ is finite.
Put $I=I^T\cup I^U\cup I^L$, and fix an arbitrary disjoint family~$\Z=\set{Z_i}{i\in I}$, where for each $i\in I$, $Z_i=\set{z_{ij}}{(i,j)\in P}$, and let $Z=\bigcup_{i\in I}Z_i$.  
Also let~$Y=\set{ y_k}{ k\in K^T}$ be an arbitrary subset of $X$ of size $\kappa$ disjoint from $Z$.  Then $(\epsilon_{Y,\Z},\omega_{Y,\Z})\in\sigma$ by part \ref{it:rg1ii}.
Then with
\[
\theta_1= \partXXIV{C_k}{A_{ij}}{C_l}{y_k}{z_{ij}}{}_{k\in K^T,\ (i,j)\in P^U,\ l\in K^U}
\AND
\theta_2= \partXXIV{y_k}{z_{ij}}{}{D_k}{B_{ij}}{D_l}_{k\in K^T,\ (i,j)\in P^L,\ l\in K^L},
\]
we have $(\gamma,\delta)=(\theta_1\epsilon_{Y,\Z}\theta_2,\theta_1\omega_{Y,\Z}\theta_2)\in\sigma$ as well.

\pfitem{\ref{it:rg1iv}}
Let $\th=\gamma\vee\delta$, noting that $\ga,\de\pre\th$ and that $\rank(\th)=\ka$ since $\ga\cap\de$ contains $\ka$ transversals (which itself follows from $|\ga\sd\de|<\aleph_0$). From $|\gamma\sd\delta|<\aleph_0$ and Lemma \ref{lem:joins} \ref{it:joinsi}, it follows that $|\gamma\sd\th|, |\delta\sd\th|<\aleph_0$ as well.  
From \ref{it:rg1iii}, we now have $(\gamma,\th),(\delta,\th)\in\sigma$, and hence $(\gamma,\delta)\in\sigma$ by transitivity.
\end{proof}

\begin{lemma}[cf.~Lemma \ref{la340}]
\label{la-hd5}
Suppose $(\alpha,\beta)\in \sigma\restr_{D_\kappa}$ where $\kappa\geq\eta\geq\aleph_0$
and $|\alpha\sd\beta|=\xi$.
\begin{thmenum}
\item \label{it:hd5i} For any disjoint $Y,Z\subseteq X$ with $|Y|=\kappa$ and $|Z|=\xi$, we have $(\epsilon_{Y,Z},\omega_{Y,Z})\in\sigma$.
\item \label{it:hd5ii} For any $Y\subseteq X$ of size $\kappa$, and any disjoint family $\Z$ of $\xi$ subsets of $X$, all of them of size~$\xi$, and all of them disjoint from $Y$, we have $(\epsilon_{Y,\Z},\omega_{Y,\Z})\in\sigma$.
\item \label{it:hd5iii} For any $\gamma,\delta\in D_\kappa$ with $|\gamma\sd\delta|\leq\xi$ and $\ga\pre\de$, we have $(\gamma,\delta)\in\sigma$. 
\item \label{it:hd5iv} For any $\gamma,\delta\in D_\kappa$ with $|\gamma\sd\delta|\leq\xi$, we have $(\gamma,\delta)\in\sigma$. 
\end{thmenum}
\end{lemma}

\pf
The proofs of all four parts are essentially identical to those of the corresponding parts of Lemma~\ref{la-rg1}.  In part \ref{it:hd5i}, we apply the second part of Lemma \ref{lem:tech3} instead of the first.  In part~\ref{it:hd5iii}, the $I$-, $J$- and $P$-type index sets are of size at most $\xi$, rather than being finite; also we have $|\ga\sd\de|\leq\xi=|\al\sd\be|<\ka$ (the latter from Lemma \ref{lem:tech1}), and since $\ga,\de\in D_\ka$ it follows that $\ga\cap\de$ has $\ka$ transversals, so $|K^T|=\ka$.  In part~\ref{it:hd5iv}, after defining $\th=\ga\vee\de$, $|\ga\sd\th|$ and $|\de\sd\th|$ are at most $\xi$, rather than being finite.
\epf

\part{The lattice of congruences}\label{part:II}

In the first part of the paper we classified all of the congruences on the partition monoid $\P_X$ and partial Brauer monoid $\PB_X$ over an arbitrary infinite set $X$.  
This second part constitutes a detailed analysis of the congruence lattices $\Cong(\P_X)$ and $\Cong(\PB_X)$.  

In Section \ref{sect:reversals} we investigate the order-reversing mappings $\Psi(\si)$ associated to congruences of type \ref{CT2}, which will play a crucial role in many of the subsequent sections.
In Section \ref{sect:order} we characterise the order relation in the lattices, and give formulae for meets and joins.
In Section \ref{sect:Hasse} we discuss Hasse diagrams of the lattices.
%
Section \ref{sect:global} concerns ``global'' properties of the 
lattices: we show they are distributive and well quasi-ordered, and we also describe the $*$-congruences (congruences that also preserve the involution $\al\mt\al^*$) and the lattice formed by them.  
In Section \ref{sect:gen}, we describe the principal congruences, and then for each congruence calculate the minimal size of a set of generating pairs.  
Finally, in Section \ref{sect:OM} we compare and contrast the results of this paper with existing results on finite diagram monoids and (finite and infinite) transformation monoids, before discussing directions for future research.

In all that follows, we continue to use $\M$ to stand for either $\P_X$ or $\PB_X$.  We generally use Theorem \ref{thm-main} without explicit reference, and we regard the parameters appearing in the theorem (and also defined in Section \ref{sect:stage2}) as functions having the congruence itself as their argument; thus, if $\si\in\Cong(\M)$, we will refer to $n(\si)$ and/or $\eta(\si)$, to $\ze_1(\si)$ and $\ze_2(\si)$, and so on.

\section{Reversals}\label{sect:reversals}

Throughout the rest of the paper, it will often be convenient to consider the mapping ${\Psi=\Psi(\sigma)}$ as a parameter of a congruence $\sigma$ of type \ref{CT2}, alternative to the parameters $k(\sigma)$, $\xi_i(\sigma)$ and~$\eta_i(\sigma)$, as explained in Remark \ref{rem:PsiPar}.  
Recall that $\Psi$ is an order-reversing mapping 
${[\eta,|X|]\rightarrow \{1\} \cup [\aleph_0,\eta]}$, where $\eta=\eta(\si)\in [\aleph_0,|X|^+]$. We will refer to any such mapping as a \emph{reversal}, and we write~${\cR=\cR_{|X|}}$ for the set of all reversals. 
Note that the empty mapping $\emptyset$ is a reversal with~${\eta=|X|^+}$.
Here we gather properties of reversals that will be used in subsequent subsections, the key fact being that $\cR$ is lattice under an order $\pre$ defined below; we prove that~$(\cR,\pre)$ is distributive in Subsection \ref{subsect:dist_R} (see Proposition \ref{prop:Psi_lattice}) and well quasi-ordered in Subsection~\ref{subsect:wqo_R} (see Corollary \ref{cor:Rwqo}).
Throughout, we use standard abbreviations: poset (partially ordered set), qoset (quasi-ordered set) and wqo (well quasi-ordered).


\subsection{Distributivity of reversals}\label{subsect:dist_R}

We begin with some basic facts about posets; for more background, see for example \cite{Blyth2005,DPbook}.
Let $(P,\leq)$ be a poset, and $I$ an arbitrary set.  The set $P^I$ of all functions~${I\to P}$ (equivalently, all $I$-tuples over $P$) is partially ordered under the component-wise order~$\leqC$ defined as follows: if $f,g\in P^I$, then
\[
f\leqC g \quad\iff\quad f(i)\leq g(i) \quad\text{for all $i\in I$}.
\]
It is routine to check that if $P$ is a (distributive) lattice then so too is $(P^I,\leqC)$.
If $f,g\in P^I$ where $P$ is a lattice and $I$ a set, we denote the meet and join of $f$ and $g$ in $P^I$ by $f\wedgeC g$ and~$f\veeC g$, respectively; for example, we have $(f\wedgeC g)(i)=f(i)\wedge g(i)$ for all $i\in I$.

Recall that a map $f\colon P\to Q$ between posets is \emph{order-reversing} if $p\leq q\implies f(p)\geq f(q)$ for all $p,q\in P$.  We write $\Rev(P,Q)\sub Q^P$ for the set of all such mappings.  If $Q$ is a lattice and $P$ an arbitrary poset, and if $f,g\in\Rev(P,Q)$, then it is easy to check that $f\wedgeC g$ and $f\veeC g$ both belong to~$\Rev(P,Q)$; i.e., $\Rev(P,Q)$ is a sublattice of $(Q^P,\leqC)$:

\begin{lemma}\label{lem:PQ}
If $P$ is a poset and $Q$ a (distributive) lattice, then $\Rev(P,Q)$ is a (distributive) lattice under $\leqC$. \epfres
\end{lemma}

We now return our attention to reversals.
If $\Psi\colon [\eta,|X|]\rightarrow \{1\} \cup [\aleph_0,\eta]$ is a reversal, then we define its \emph{extension} $\ext{\Psi}\colon  [0,|X|]\rightarrow [0,|X|^+]$ by
\[
\ext{\Psi}(\kappa) = \begin{cases}
|X|^+ & \text{if }\ka\in[0,\eta)\\
\Psi(\kappa) & \text{if } \ka\in[\eta,|X|].
\end{cases}
\]
It is clear that $\ext\Psi$ is order-reversing, and uniquely determined by $\Psi$, so we have an injective map
\[
\cR\to\RevX\colon \Psi\mapsto\ext\Psi.
\]

We now define an order $\pre$ on $\cR$.  To do so, consider two reversals
\begin{equation}
\label{eq:Psi}
\Psi_1\colon [\eta,|X|]\to\{1\}\cup[\aleph_0,\eta]  \AND \Psi_2\colon [\eta',|X|]\to\{1\}\cup[\aleph_0,\eta'].
\end{equation}
We write
\begin{equation}
\label{eq:prePsi}
\Psi_1\pre\Psi_2 \quad\iff\quad \text{$\eta\leq\eta'$ \ and \ $\Psi_1(\ka)\leq\Psi_2(\ka)$ \ for all \ $\ka\in[\eta',|X|]$.}
\end{equation}
In fact, it is clear that
\begin{equation}
\label{eq:prePsi2}
\Psi_1\pre\Psi_2 \quad\iff\quad  \ext{\Psi_1}\leqC\ext{\Psi_2},
\end{equation}
where $\leqC$ denotes the component-wise order on $\RevX$.  Together with the fact that the map $\Psi\mt\ext\Psi$ is injective, it follows that $\pre$ is a partial order on $\cR$, and that $\Psi\mt\ext\Psi$ is an order-embedding of $\cR$ in $\RevX$.

Note that since $[0,|X|^+]$ is totally ordered, it is a distributive lattice, with the meet and join of two cardinals being their minimum and maximum, respectively.  It follows from Lemma~\ref{lem:PQ} that $\RevX$ is itself a distributive lattice under $\leqC$.

If $\Psi_1$ and $\Psi_2$ are two reversals as in~\eqref{eq:Psi}, then by their form, $\ext{\Psi_1}\wedgeC\ext{\Psi_2}$ and $\ext{\Psi_1}\veeC\ext{\Psi_2}$ are both in the image of the $\Psi\mt\ext\Psi$ map, so we may define $\Psi_1\wedge\Psi_2$ and~$\Psi_1\vee\Psi_2$ to be the unique reversals satisfying 
\[
\ext{(\Psi_1\wedge\Psi_2)}=\ext{\Psi_1}\wedgeC\ext{\Psi_2} \AND \ext{(\Psi_1\vee\Psi_2)}=\ext{\Psi_1}\veeC\ext{\Psi_2}.
\]
Explicitly, if $\eta\leq\eta'$, then 
\[
\Psi_1\wedge\Psi_2=\Psi_2\wedge\Psi_1\colon [\eta,|X|]\to\{1\}\cup[\aleph_0,\eta]  \ANd \Psi_1\vee\Psi_2=\Psi_2\vee\Psi_1\colon [\eta',|X|]\to\{1\}\cup[\aleph_0,\eta']
\]
are given by
\begin{align}
\label{eq:meet}
(\Psi_1\wedge\Psi_2)(\ka) &= \begin{cases}
\Psi_1(\ka) &\hspace{1.5mm}\text{for $\ka\in[\eta,\eta')$}\\
\min(\Psi_1(\ka),\Psi_2(\ka)) &\hspace{1.5mm}\text{for $\ka\in[\eta',|X|]$,}
\end{cases}\\
\intertext{and}
\label{eq:join}
(\Psi_1\vee\Psi_2)(\ka) &= \max(\Psi_1(\ka),\Psi_2(\ka)) \qquad\text{for $\ka\in[\eta',|X|]$.}
\end{align}
It follows quickly from \eqref{eq:prePsi2} that $\Psi_1\wedge\Psi_2$ and $\Psi_1\vee\Psi_2$ are the meet and join of $\Psi_1$ and $\Psi_2$ in~$\cR$, justifying the suggestive notation, and so the map 
$\Psi\mapsto\ext{\Psi}$ is in fact a lattice embedding of~$\cR$ in $\RevX$.  Since the latter is distributive, as observed above, we immediately deduce the following:

\begin{prop}\label{prop:Psi_lattice}
The set $\cR=\cR_{|X|}$ of all reversals is a distributive lattice under the order $\pre$ given by \eqref{eq:prePsi}, and with meet and join operations given by \eqref{eq:meet} and \eqref{eq:join}.  \epfres
\end{prop}

\subsection{Well quasi-orderedness of reversals}\label{subsect:wqo_R}

Recall that a qoset is \emph{well quasi-ordered} (wqo) if it contains no infinite strictly descending chains and no infinite antichains.  Here we prove that the lattice $(\cR,\pre)$ of all reversals is wqo.  

In order to prove this we first need to gather some fundamental facts about qosets.  Unless specified otherwise, we use $\leq$ to denote the quasi-order in any qoset.  Clearly any subset of a wqo qoset is itself wqo under the induced quasi-order.  

The next lemma is part of \cite[Theorem 2.1]{higman52}:

\begin{lemma}\label{lem:alt_wqo}
A qoset $Q$ is wqo if and only if the following condition is satisfied:
\bit
\item[] for any infinite sequence $q_1,q_2,q_3,\ldots$ in $Q$, there exists $i<j$ such that $q_i\leq q_j$. \epfres
\eit
\end{lemma}

For a qoset $Q$ denote by $Q^\ast$ the set of all finite sequences of elements of $Q$.
This set can be equipped with the so-called \emph{domination} quasi-order $\leqdom$: 
\begin{align*}
(q_1,\dots,q_m) &\leqdom (q_1',\dots,q_n') \\
& \iff 
\text{there exist $1\leq j_1<j_2<\dots<j_m\leq n$ such that $q_i\leq q_{j_i}'$ for all $i$.}
\end{align*}

\begin{lemma}[{Higman's Lemma, \cite[Theorem 4.3]{higman52}}]
\label{lem:higman}
If a qoset $Q$ is wqo then so is $(Q^\ast,\leqdom)$.~\epfres
\end{lemma}

An immediate consequence is the following (viewing the direct product $Q_1\times\dots\times Q_k$ as a subqoset of ${(Q_1\cup\cdots\cup Q_k)^*}$):

\begin{lemma}[{Dixon's Lemma, see also \cite[Theorem 2.3]{higman52}}]
\label{lem:dixon}
If the qosets $Q_1,\ldots,Q_k$ are wqo, then so is their direct product
$Q_1\times\dots\times Q_k$ under the component-wise quasi-order. \epfres
\end{lemma}

As in Subsection \ref{subsect:dist_R}, for posets $P$ and $Q$, we write $\Rev(P,Q)$ for the set of all order-reversing functions $P\to Q$.  So $\Rev(P,Q)$ is a poset under the component-wise order $\leqC$.  The proof of the next result uses ideas similar to those introduced in Subsection \ref{subsect:infinite_eta}.

\begin{prop}\label{prop:RevPQ}
If $P$ and $Q$ are well-ordered chains, then $\Rev(P,Q)$ is well quasi-ordered under $\leqC$.
\end{prop}

\pf
We denote the orders on both $P$ and $Q$ by $\leq$.
Let $\top$ be a symbol belonging to neither~$P$ nor $Q$, and denote by $\Ptop$ and $\Qtop$ the well-ordered chains obtained by adjoining $\top$ as a new top element to $P$ and $Q$.

Consider some $f\in\Rev(P,Q)$.  Since the image of $f$ is a descending chain in the well-ordered set $Q$, it must be finite, say $\{q_1,\ldots,q_m\}$ where $q_1>\cdots>q_m$.  Since~$P$ is well-ordered, we may define $p_{i-1}=\min\set{p\in P}{f(p)=q_i}$ for each $i\in\{1,\ldots,m\}$.  We also define $q_0=\top=p_m$.  Note that $p_0<\cdots<p_{m-1}<p_m=\top$ and ${\top=q_0>q_1>\cdots>q_m}$.  We then define 
\[
\tb f = \big( (p_0,q_0),(p_1,q_1),\ldots,(p_m,q_m)\big)\in (\Ptop\times\Qtop)^*.
\]
We claim that 
\begin{equation}\label{eq:fg}
\tb f \leqdom \tb g \quad \implies \quad f\leqC g  \qquad\text{for all $f,g\in\Rev(P,Q)$,}
\end{equation}
where here $\leqdom$ is the domination order on $(\Ptop\times\Qtop)^*$.  

To prove the claim, suppose $f,g\in\Rev(P,Q)$ are such that $\tb f\leqdom\tb g$, and write
\[
\tb f = \big( (p_0,q_0),(p_1,q_1),\ldots,(p_m,q_m)\big) \AND \tb g = \big( (p_0',q_0'),(p_1',q_1'),\ldots,(p_n',q_n')\big).
\]
Note then that
\begin{align*}
p_0<\cdots<p_m=\top \COMMA 
p_0'<\cdots<p_n'=\top \COMMA 
\top=q_0>\cdots>q_m \COMMA 
\top=q_0'>\cdots>q_n' ,
\end{align*}
and that $f(p_{i-1})=q_i$ for all $1\leq i\leq m$, and $g(p_{i-1}')=q_i'$ for all $1\leq i\leq n$.
By assumption, there exist $0\leq j_0<j_1<\cdots<j_m\leq n$ such that $(p_i,q_i)\leq(p_{j_i}',q_{j_i}')$ in $\Ptop\times\Qtop$ for each $0\leq i\leq m$.  

Now let $p\in P$ be arbitrary.  The claim will be proved if we can show that $f(p)\leq g(p)$.  Let us write $f(p)=q_k$ and $g(p)=q_l'$, where $1\leq k\leq m$ and $1\leq l\leq n$.  Note that since $f(p)=q_k=f(p_{k-1})$, we have $p_{k-1}\leq p<p_k$.  

Suppose first that $k=m$.  Then $\top=p_m\leq p_{j_m}'$, so that $p_{j_m}'=\top$, which forces $j_m=n$.  But then 
$f(p) = q_m \leq q_{j_m}' = q_n' \leq q_l' = g(p)$.

Now suppose $1\leq k<m$.
Note then that this forces  $p_k<p_m=\top$, and also $0<j_k<n$.
Now $p<p_k\leq p_{j_k}'$, so from minimality of $p_{j_k}'$ and the fact that
$g$ is order reversing it follows that
$q_l'=g(p)>g(p_{j_k}')=q_{j_k+1}'$.
This means that $l<j_k+1$: i.e., that $l\leq j_k$.
But then ${f(p)=q_k\leq q_{j_k}'\leq q_l'=g(p)}$.
Thus \eqref{eq:fg} is proved.

 Returning to the main proof now, consider an infinite sequence $f_1,f_2,f_3,\ldots$ of elements from $\Rev(P,Q)$.  By Lemmas~\ref{lem:higman} and \ref{lem:dixon}, $(\Ptop\times\Qtop)^*$ is wqo under $\leqdom$, so by Lemma \ref{lem:alt_wqo} it follows that $\tb{f_i}\leqdom\tb{f_j}$ for some $i<j$.  But then by \eqref{eq:fg} we have $f_i\leqC f_j$; the proof concludes by again appealing to Lemma~\ref{lem:alt_wqo}.
\epf

\begin{cor}\label{cor:Rwqo}
The lattice $(\cR,\pre)$ is well quasi-ordered.  
\end{cor}

\pf
We noted at the end of Subsection \ref{subsect:dist_R} that the map $\Psi\to\ext\Psi$ is a lattice embedding of~$\cR$ in ${\Rev}\big([0,|X|],[0,|X|^+]\big)$.  Since the latter is wqo by Proposition \ref{prop:RevPQ}, the claim follows.
\epf

\section{The lattice order and operations}\label{sect:order}

In this section we describe the fundamental properties of the congruence lattice $\Cong(\M)$, where as usual $\M$ denotes either $\P_X$ or $\PB_X$ with $X$ a fixed infinite set.  Specifically, we characterise the inclusion order in Subsection \ref{subsect:inclusion} (see Theorem \ref{thm:comparisons}), and give formulae for the meet and join of arbitrary pairs of congruences in Subsection \ref{subsect:mj} (see Theorem \ref{thm:mj}).  We also record in Corollary \ref{cor:iso} the isomorphism between the lattices $\Cong(\P_X)$ and $\Cong(\PB_X)$.

\subsection{The inclusion order}\label{subsect:inclusion}

The purpose of this subsection is to prove the following result, which characterises the partial order by inclusion on congruences of $\M$, and which we derive as a consequence of Theorem \ref{thm-main} and certain ideas developed during its proof.  For the statement, recall that we consider~$\nabla_{\M}$ to be a congruence of type \ref{CT2} with $k=1$, $\eta=\zeta_1=\zeta_2=\eta_1=|X|^+$ and $\xi_1=1$.

\begin{thm}
\label{thm:comparisons}
Let $X$ be an infinite set, let $\M$ stand for either $\P_X$ or $\PB_X$, and let $\sigma$ and~$\tau$ be congruences on $\M$.
Then $\sigma\subseteq \tau$ if and only if one of the following is true:
\begin{thmenum}
\item 
\label{it:comp1}
$\sigma$ and $\tau$ are both of type \ref{CT1},
$\zeta_1(\sigma)\leq\zeta_1(\tau)$, $\zeta_2(\sigma)\leq\zeta_2(\tau)$, and one of the following holds:
\bit
\item
$n(\sigma)<n(\tau)$, or
\item
$n(\sigma)=n(\tau)$ and $N(\sigma)\leq N(\tau)$,
\eit
\item
\label{it:comp2}
$\sigma$ is of type \ref{CT1}, $\tau$ is of type \ref{CT2}, and
$\zeta_1(\sigma)\leq\zeta_1(\tau)$ and $\zeta_2(\sigma)\leq\zeta_2(\tau)$,
\item
\label{it:comp3}
$\sigma$ and $\tau$ are both of type \ref{CT2}, $\eta(\sigma)\leq\eta(\tau)$,
$\zeta_1(\sigma)\leq\zeta_1(\tau)$, $\zeta_2(\sigma)\leq\zeta_2(\tau)$, and
there exist $0\leq j_1\leq j_2\leq \dots\leq j_{k(\sigma)}\leq k(\tau)$ such that
\[
\xi_i(\sigma)\leq\xi_{j_i}(\tau) \ANd \eta_i(\sigma)\leq\eta_{j_i}(\tau)
\qquad\text{for all } i=1,\dots, k(\sigma),
\]
with the convention that $\xi_0(\tau)=\eta_0(\tau)=\eta(\tau)$.
\end{thmenum}
\end{thm}

\begin{proof}
Throughout the proof we will write the parameters associated with~$\sigma$ with a single dash, and those of $\tau$ with two; for instance, $\eta(\sigma)=\eta'$ and $\eta(\tau)=\eta''$.

\pfitem{($\Leftarrow$)}
We first show that if one of \ref{it:comp1}, \ref{it:comp2} or \ref{it:comp3} is satisfied then $\sigma\subseteq\tau$.

Suppose first that \ref{it:comp1} holds.
From $\zeta_1'\leq\zeta_1''$ and $\zeta_2'\leq\zeta_2''$ it follows that
$\lambda_{\zeta_1'}\subseteq\lambda_{\zeta_1''}$ and
$\rho_{\zeta_2'}\subseteq\rho_{\zeta_2''}$, and from $n'\leq n''$ we have $R_{n'}\subseteq R_{n''}$.
Since $\nu_{N'}\subseteq (D_{n'}\times D_{n'})\cap{\H}$, and since for $(\alpha,\beta)\in{\H}$ we have $|\overline{\alpha}\sd\overline{\beta}|=|\underline{\alpha}\sd\underline{\beta}|=0$,
it follows that if $n'<n''$ we have
 $\nu_{N'}\subseteq\lambda_{\zeta_1''}\cap\rho_{\zeta_2''}\cap R_{n''}$.
On the other hand, if $n'=n''$ and $N'\leq N''$ then we have
$\nu_{N'}\subseteq\nu_{N''}$ straight from the definition of
these relations in Subsection \ref{subsect:statement}.
In either case we have
\[
\sigma=(\lambda_{\zeta_1'}\cap\rho_{\zeta_2'}\cap R_{n'})\cup\nu_{N'}
\subseteq
(\lambda_{\zeta_1''}\cap\rho_{\zeta_2''}\cap R_{n''})\cup\nu_{N''}=\tau.
\]

Next suppose \ref{it:comp2} holds.
Again, we have
$\lambda_{\zeta_1'}\subseteq\lambda_{\zeta_1''}$ and
$\rho_{\zeta_2'}\subseteq\rho_{\zeta_2''}$.
Since $n'$ is finite and~$\eta''$ infinite, we have
$R_{n'}\subseteq R_{\eta''}$ and (as in the previous case)
$\nu_{N'}\subseteq \lambda_{\zeta_1''}\cap\rho_{\zeta_2''}\cap R_{\eta''}$. 
Combining, we have
\[
\sigma=(\lambda_{\zeta_1'}\cap\rho_{\zeta_2'}\cap R_{n'})\cup \nu_{N'}
\subseteq (\lambda_{\zeta_1''}\cap \rho_{\zeta_2''}\cap R_{\eta''})
\cup\mu_{\xi_1''}^{\eta_1''}\cup\dots\cup\mu_{\xi_{k''}''}^{\eta_{k''}''}=\tau.
\]

Finally suppose \ref{it:comp3} holds. As above, we have
\begin{equation}
\label{eq:cmp11}
\lambda_{\zeta_1'}^{\eta'}\cap \rho_{\zeta_2'}^{\eta'}\subseteq
\lambda_{\zeta_1''}^{\eta''}\cap \rho_{\zeta_2''}^{\eta''}\subseteq \tau.
\end{equation}
We also claim that
\begin{equation}
\label{eq:cmp12}
\mu_{\xi_i'}^{\eta_i'}\subseteq \tau \qquad\text{for all } i=1,\dots,k'.
\end{equation}
Indeed, if $j_i=0$, then $\eta_i'\leq\eta_0''=\eta''$ and also $\xi_i'\leq\xi_0''=\eta''\leq\min(\zeta_1'',\zeta_2'')$; together with Lemma~\ref{la32}~\ref{it:32vii} and~\ref{it:32viii}, it quickly follows that $\mu_{\xi_i'}^{\eta_i'}\subseteq \lambda_{\zeta_1''}^{\eta''}\cap \rho_{\zeta_2''}^{\eta''}\subseteq \tau$.
Suppose now that $j_i\geq 1$.
Then from $\xi_i'\leq\xi_{j_i}''$ and $\eta_i'\leq\eta_{j_i}''$ 
we have $\mu_{\xi_i'}\subseteq\mu_{\xi_{j_i}''}$ and $R_{\eta_i'}\subseteq R_{\eta_{j_i}''}$,
and hence $\mu_{\xi_i'}^{\eta_i'}\subseteq \mu_{\xi_{j_i}''}^{\eta_{j_i}''}\subseteq\tau$,
completing the proof of \eqref{eq:cmp12}.
Combining \eqref{eq:cmp11} and \eqref{eq:cmp12} yields $\sigma\subseteq\tau$, as desired.

\pfitem{($\Rightarrow$)}
Suppose now that $\sigma\subseteq\tau$; we must show that one of \ref{it:comp1}--\ref{it:comp3} holds.
First, from $\si\sub\tau$ and the definitions of the parameters $\eta,\ze_1,\ze_2$ we immediately have
\begin{equation}
\label{eq:cmp1}
\eta'\leq\eta''\COMMA \zeta_1'\leq\zeta_1''\COMMA \zeta_2'\leq\zeta_2''.
\end{equation}
We now split our considerations into cases, depending on the types of $\sigma$ and $\tau$.
Note that $\eta'\leq\eta''$ immediately implies that it is impossible for $\si$ to be of type \ref{CT2} and $\tau$ of type~\ref{CT1}.

\pfcase1  Suppose first that $\sigma$ and $\tau$ are both of type \ref{CT1}.
We first obtain $n'\leq n''$ from~\eqref{eq:cmp1}.
It remains to show that if $n'=n''$ then $N'\leq N''$.
Indeed, if $n'=n''$ but $N'\nleq N''$, say with $\pi\in N'\setminus N''$,
then $(\permdec{\pi},\permdec{\id_{n'}})\in \nu_{N'}\setminus\nu_{N''}$, which would contradict $\sigma\subseteq \tau$ because
$\nu_{N'}=\sigma\restr_{D_{n'}}$ and $\nu_{N''}=\tau\restr_{D_{n'}}$.
Thus, in this case, condition \ref{it:comp1} is satisfied.

\pfcase2  If $\sigma$ is of type \ref{CT1} and $\tau$ is of type \ref{CT2}, then \eqref{eq:cmp1} implies that condition~\ref{it:comp2} holds.

\pfcase3  Finally, suppose $\sigma$ and $\tau$ are both of type \ref{CT2}.
For each $i=1,\dots,k'$ let
\[
j_i=\min\set{ j}{ 0\leq j\leq k'',\ \eta_i'\leq \eta_j''}.
\]
Notice that $j_i$ is well defined because $\eta_i'\leq |X|^+=\eta_{k''}''$.
From $\eta_1'<\eta_2'<\dots<\eta_{k'}'$ it follows that
\[
0\leq j_1\leq j_2\leq\dots\leq j_{k'}\leq k''.
\]
Thus, since $\eta_i'\leq \eta_{j_i}''$ for all $i$ by definition, the proof will be complete if we can show that
\begin{equation}
\label{eq:cmp3}
\xi_i'\leq \xi_{j_i}'' \qquad\text{ for all } i=1,\dots,k'.
\end{equation}
Clearly $\xi_i'\leq \xi_{j_i}''$ if $\xi_i'=1$ or if $j_i=0$, since in the latter case we have $\xi_i'\leq\eta'\leq\eta''=\xi_0''$.
To deal with the remaining cases, suppose, aiming for a contradiction, that $\xi_i'>\xi_{j_i}''$ for some $i$
with~$\xi_i'\geq\aleph_0$ and $j_i>0$.
Let $\alpha,\beta\in\M$ be two partitions of rank $\eta_{j_i-1}''$ satisfying
$|\alpha\sd\beta|=\xi_{j_i}''$.
From $\eta_{j_i-1}''<\eta_i'$ and $|\alpha\sd\beta|=\xi_{j_i}''<\xi_i'$ we have
$(\alpha,\beta)\in \mu_{\xi_i'}^{\eta_i'}\subseteq\sigma$.
We claim that
\begin{equation}
\label{eq:cmp2}
(\alpha,\beta)\not\in\tau=
(\lambda_{\zeta_1''}^{\eta''}\cap \rho_{\zeta_2''}^{\eta''})
\cup\mu_{\xi_1''}^{\eta_1''}\cup\dots\cup\mu_{\xi_{k''}''}^{\eta_{k''}''}.
\end{equation}
From $j_i\neq 0$ (and $\al\not=\be$) we have that $\eta_{j_i-1}''\geq \eta''$ and so
$(\alpha,\beta)\not\in\lambda_{\zeta_1''}^{\eta''}\cap\rho_{\zeta_2''}^{\eta''}$.
Now consider an arbitrary $j\in\{1,\dots,k''\}$.
If $j<j_i$ then 
$\rank(\alpha)=\rank(\beta)=\eta_{j_i-1}''\geq\eta_j''$, and hence
$(\alpha,\beta)\not\in\mu_{\xi_j''}^{\eta_j''}$.
If, on the other hand, $j\geq j_i$ then
from $|\alpha\sd\beta|=\xi_{j_i}''\geq \xi_j''$ it again follows
that $(\alpha,\beta)\not\in\mu_{\xi_j''}^{\eta_j''}$.
Thus $(\alpha,\beta)$ does not belong to any of the relations the union of which makes up $\tau$, and hence~\eqref{eq:cmp2} is proved.
But we now have $(\alpha,\beta)\in\sigma\setminus\tau$, which contradicts $\sigma\subseteq\tau$. This means that \eqref{eq:cmp3} holds.
As noted above, this completes the proof of the theorem.
\end{proof}

Of course every congruence is contained in $\nabla_{\M}$.  Note that when $\tau=\nabla_{\M}$ in part~\ref{it:comp3} of the above theorem, we take $j_i=0$ for all $i=1,\ldots,k(\si)$.

The next result follows immediately from Theorems \ref{thm-main} and \ref{thm:comparisons} for infinite $X$, and from \cite[Theorems 5.4 and~6.1]{EMRT2018} for finite $X$.

\begin{cor}\label{cor:iso}
For any set $X$, the lattices $\Cong(\P_X)$ and $\Cong(\PB_X)$ are isomorphic.  An explicit isomorphism $\Cong(\P_X)\to\Cong(\PB_X)$ is given by the mapping  $\si\mt\si\restr_{\PB_X}$.  \epfres
\end{cor}

\subsection{Meets and joins}\label{subsect:mj}

We now use Theorems \ref{thm-main} and \ref{thm:comparisons}, and ideas from Section \ref{sect:reversals}, to describe the meet $\si\wedge\tau$ and join $\si\vee\tau$ of an arbitrary pair of congruences $\si,\tau$ on $\M$.
In the statement and proof, it will also be convenient to make use of a total order $\pre$ defined on the set
\[
\NN=\set{N}{N\normal\S_n \text{ for some }n\in[1,\aleph_0)}
\]
of all normal subgroups of all finite $\S_n$ as follows.  If $N\normal\S_n$ and $N'\normal\S_{n'}$, then
\begin{align}
\label{eq:N_pre}
N\pre N' \quad &\iff\quad n<n' \qquad\text{or}\qquad n=n' \text{ \ and \ } N\leq N'.
\intertext{Since the normal subgroups of $\S_n$ form a chain for every $n$, this is a total order on $\NN$, and so
we may speak of the maximum and minimum of any pair $N,N'\in\NN$, which we will denote by $\max(N,N')$ and $\min(N,N')$.
Using the order $\pre$, Theorem~\ref{thm:comparisons}~\ref{it:comp1}
concerning congruences $\sigma,\tau$ of type \ref{CT1} can be re-stated as follows:}
\label{eq:CT1_cont}
\si\sub\tau \quad &\iff \quad \text{$\ze_1(\si)\leq\ze_1(\tau)$, \ $\ze_2(\si)\leq\ze_2(\tau)$ \ and \ $N(\si)\pre N(\tau)$.}
\end{align}

\begin{thm}
\label{thm:mj}
Let $X$ be an infinite set, let $\M$ stand for either $\P_X$ or $\PB_X$, and let $\sigma$ and $\tau$ be congruences on $\M$.  Then
\begin{align*}
\eta(\si\wedge\tau)&=\min(\eta(\si),\eta(\tau)), & \ze_1(\si\wedge\tau)&=\min(\ze_1(\si),\ze_1(\tau)), & \ze_2(\si\wedge\tau)&=\min(\ze_2(\si),\ze_2(\tau)), \\
\eta(\si\vee\tau)&=\max(\eta(\si),\eta(\tau)), & \ze_1(\si\vee\tau)&=\max(\ze_1(\si),\ze_1(\tau)), & \ze_2(\si\vee\tau)&=\max(\ze_2(\si),\ze_2(\tau)), 
\end{align*}
and additionally:
\begin{thmenum}
\item 
\label{it:mj1}
if $\sigma$ and $\tau$ are both of type \ref{CT1} then so too are $\sigma\wedge\tau$ and $\sigma\vee\tau$, with
\[
N(\si\wedge\tau)=\min(N(\si),N(\tau))
\AND
N(\si\vee\tau)=\max(N(\si),N(\tau)),
\]

\item
\label{it:mj2}
if $\sigma$ and $\tau$ have different types, then
$\sigma\wedge\tau$ is of type \ref{CT1} and $\sigma\vee\tau$  is of type~\ref{CT2}, with
\[
N(\si\wedge\tau)=
\begin{cases} 
N(\si) & \text{if } \eta(\sigma)<\aleph_0\leq\eta(\tau)\\
N(\tau) & \text{if } \eta(\tau)<\aleph_0\leq\eta(\sigma)
\end{cases}
\ \ANd \ 
\Psi(\si\vee\tau)=
\begin{cases}
\Psi(\tau)& \text{if } \eta(\sigma)<\aleph_0\leq\eta(\tau)\\
\Psi(\sigma) & \text{if } \eta(\tau)<\aleph_0\leq\eta(\sigma),
\end{cases}
\]

\item
\label{it:mj3}
if $\sigma$ and $\tau$ are both of type \ref{CT2}, then so too are $\sigma\wedge\tau$ and $\sigma\vee\tau$, with
\[
\Psi(\si\wedge\tau)=\Psi(\si)\wedge\Psi(\tau)
\AND
\Psi(\si\vee\tau)=\Psi(\si)\vee\Psi(\tau).
\]

\end{thmenum}
\end{thm}

\begin{proof}
All the statements can be proved by following the same method:
\begin{itemize}
\item
prove that the stated parameters form a permissible combination, and hence define a congruence $\vs$ by Theorem \ref{thm-main},
\item
prove that $\vs\subseteq \sigma,\tau$ (or $\sigma,\tau\subseteq \vs$) in the case of $\sigma\wedge\tau$ (or $\sigma\vee\tau$), respectively,
\item
prove that $\vs$ is the greatest (or least) congruence with the above property in the case of~$\sigma\wedge\tau$ (or $\sigma\vee\tau$), respectively.
\end{itemize}
As a sample proof we go through these steps for $\sigma\wedge\tau$ in part \ref{it:mj1}.
Since $\si\wedge\tau=\tau\wedge\si$, we may assume that $n(\si)\leq n(\tau)$.
So let
\[
n = n(\si) = \min(n(\si),n(\tau)) \COMMA
N = \min(N(\si),N(\tau)) \COMMA
\zeta_i=\min(\zeta_i(\sigma),\zeta_i(\tau)),\ i=1,2.
\]
Certainly $n\in[1,\aleph_0)$, and also $N\normal\S_n$
using \eqref{eq:N_pre}.
Next, we clearly have ${\zeta_i\in \{1\}\cup [\aleph_0,|X|^+]}$ for $i=1,2$. Moreover $\zeta_i\neq 1$ if $n=n(\si)\geq 3$; indeed, if $n\geq3$, then $\ze_i(\si)\geq\aleph_0$, and since also $n(\tau)\geq n\geq3$, we have $\ze_i(\tau)\geq\aleph_0$ as well.  
It now follows from Theorem \ref{thm-main} that there exists a congruence $\vs=\lam_{\ze_1}^N\cap\rho_{\ze_2}^N$ of type \ref{CT1}.
We have $\vs\sub\si,\tau$ by construction; cf.~\eqref{eq:CT1_cont}.

Conversely, let $\vs'$ be any congruence of $\M$ satisfying $\vs'\subseteq\sigma,\tau$; we must show that $\vs'\sub\vs$.
By Theorem \ref{thm:comparisons}, $\vs'$ must be of type \ref{CT1}, and 
then \eqref{eq:CT1_cont} yields
\begin{align*}
n(\vs')&\leq \min(n(\sigma),n(\tau))=n=n(\vs), & \zeta_1(\vs')&\leq \min(\zeta_1(\sigma),\zeta_1(\tau))=\zeta_1=\zeta_1(\vs), \\
N(\vs')&\pre \min(N(\sigma),N(\tau))=N=N(\vs), & \zeta_2(\vs')&\leq \min(\zeta_2(\sigma),\zeta_2(\tau))=\zeta_2=\zeta_2(\vs).
\end{align*}
Again appealing to \eqref{eq:CT1_cont}, it follows that $\vs'\sub\vs$, as required.
\end{proof}

Before moving on, it will be convenient to deduce an alternative characterisation of the containment order on congruences of type \ref{CT2}, analogous to \eqref{eq:CT1_cont} for \ref{CT1} congruences.  By Proposition \ref{prop:Psi_lattice}, the set of all reversals is a lattice under the ordering $\pre$ defined in \eqref{eq:prePsi}.  It follows that for any reversals $\Psi_1,\Psi_2$ we have $\Psi_1\pre\Psi_2 \iff \Psi_1=\Psi_1\wedge\Psi_2 \iff \Psi_2=\Psi_1\vee\Psi_2$.  Using the latter observation, the next result follows quickly from Theorem~\ref{thm:mj}~\ref{it:mj3} and the fact that $\si\sub\tau \iff \si=\si\wedge\tau \iff \tau=\si\vee\tau$.

\begin{cor}\label{cor:CT2comp}
If $\si$ and $\tau$ are two congruences on $\M$ of type \ref{CT2}, then 
\[
\si\sub\tau \quad \iff \quad \text{$\ze_1(\si)\leq\ze_1(\tau)$, \ $\ze_2(\si)\leq\ze_2(\tau)$ \ and \ $\Psi(\si)\pre \Psi(\tau)$.}
\epfreseq
\]
\end{cor}

\section{Hasse diagrams}\label{sect:Hasse}

Theorems \ref{thm-main} and \ref{thm:comparisons} completely describe the structure of the congruence lattice $\Cong(\M)$; here as usual $\M$ stands for either $\P_X$ or $\PB_X$ for an infinite set~$X$.  From these theorems, it is possible to obtain a visual/geometric understanding of the lattice, or at least of certain sections of it; thus, in this subsection, we discuss Hasse diagrams.  These diagrams also give a visual interpretation of Theorem~\ref{thm:mj}, which describes meets and joins of arbitrary pairs of congruences.

It is unfeasible to draw Hasse diagrams for the entire lattice $\Cong(\M)$ in general, due to both the number of congruences, and the complicated nature of their comparisons, particularly those between congruences of type \ref{CT2}; cf.~Theorem \ref{thm:comparisons} \ref{it:comp3} and Corollary \ref{cor:CT2comp}. Nonetheless, it is possible to visualise fairly accurately various sections of the lattice, and to piece these together into pictures of the whole lattice for ``small'' $X$. The key concept for doing this is that of a \emph{layer}, which consists of all congruences of a certain type where the parameters $\zeta_1$ and $\zeta_2$ are allowed to range over all permissible values, and all the other parameters are fixed. 
We will denote by $\Lay_1(N)$ a typical layer consisting of congruences of type~\ref{CT1},
and by $\Lay_2(\eta,\xi_1,\dots,\xi_k,\eta_1,\dots,\eta_k)$ a layer consisting of congruences of type \ref{CT2}; as we have seen, we could equally well speak of layers $\Lay_2(\Psi)$ of type \ref{CT2} congruences,
where~${\Psi\colon [\eta,|X|]\rightarrow\{1\}\cup [\aleph_0,\eta)}$ is a reversal.
Now, Theorems \ref{thm:comparisons} and \ref{thm:mj} imply that every layer is a sublattice of $\Cong(\M)$, and is isomorphic to the direct product of two copies of the chain of permissible values for $\zeta_1,\zeta_2$ under the component-wise ordering. Specifically, 
\bit
\item $\Lay_1(N)$ is isomorphic to $\big(\{1\}\cup [\aleph_0,|X|^+]\big)\times \big(\{1\}\cup [\aleph_0,|X|^+]\big)$ if $N\normal\S_n$ with $n\leq2$, 
\item $\Lay_1(N)$ is isomorphic to $ [\aleph_0,|X|^+]\times [\aleph_0,|X|^+]$ if $N\normal\S_n$ with $n\geq3$, and 
\item $\Lay_2(\eta,\xi_1,\dots,\xi_k,\eta_1,\dots,\eta_k)$ is isomorphic to $ [\eta,|X|^+]\times [\eta,|X|^+]$.
\eit
The poset corresponding to $\Lay_1(\S_1)$ is shown in Figure \ref{fig:1layer}.  The layers of type \ref{CT1} with $n=2$ have exactly the same Hasse diagram, while for the layers of type \ref{CT1} with $n\geq3$, and those of type \ref{CT2}, only the indexing sets change, as just discussed.

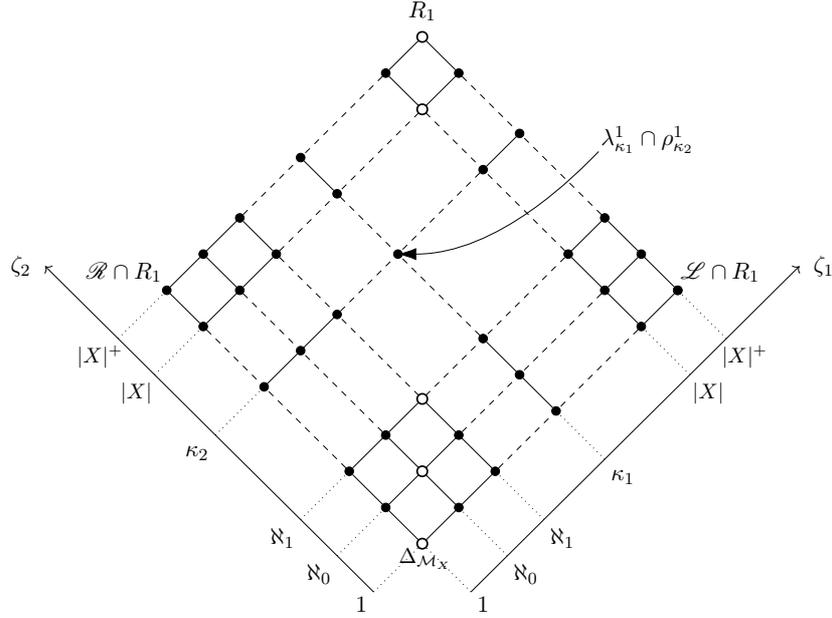
\begin{figure}[ht]
\begin{center}
\scalebox{0.8}{
\begin{tikzpicture}[x={(1mm,1mm)},y={(-1mm,1mm)},inner sep=0.5mm,distance=0.1mm]
\tikzstyle{every label}+=[distance=-3mm]
\begin{scope}[every node/.style={circle,fill=black,draw}]

\node (n11) at (0,0) {};
\node at (-1,-1.5) [rectangle,draw=none,fill=none] {\small$\Delta_{\M}$};
\node (n21) at (6,0) {};
\node (n31) at (12,0)  {};
\node (n41) at (22,0)  {};
\node (n51) at (36,0)  {};
\node (n61) at (42,0)  {};
\node [fill=none,draw=none] at (47,-2) {${\L}\cap R_1$};

\node (n12) at (0,6)  {};
\node (n13) at (0,12)  {};
\node (n14) at (0,26)  {};
\node (n15) at (0,36) {};
\node (n16) at (0,42) {};
\node [fill=none,draw=none] at (-2,47) {${\R}\cap R_1$};

\node (n22) at (6,6)  {};
\node (n23) at (6,12) {};
\node (n24) at (6,26) {};
\node (n25) at (6,36) {};
\node (n26) at (6,42) {};

\node (n32) at (12,6)  {};
\node (n33) at (12,12) {};
\node (n34) at (12,26) {};
\node (n35) at (12,36) {};
\node (n36) at (12,42) {};

\node (n42) at (22,6)  {};
\node (n43) at (22,12) {};
\node (n44) at (22,26) {};   
\node  at (52,15) [draw=none, fill=none] {$\lambda_{\ka_1}^1\cap\rho_{\ka_2}^1$};
\draw [arrows = {-Latex[length=3mm,width=2mm]}] (47,18) parabola [bend at end] (n44);
\node (n45) at (22,36) {};
\node (n46) at (22,42) {};

\node (n52) at (36,6)  {};
\node (n53) at (36,12) {};

\node (n54) at (36,26) {};
\node (n55) at (36,36) {};
\node (n56) at (36,42) {};

\node (n61) at (42,0)  {};
\node (n62) at (42,6)  {};
\node (n63) at (42,12)  {};
\node (n64) at (42,26)  {};
\node (n65) at (42,36)  {};
\node (n66) at (42,42) [label=above:$R_1$] {};
\end{scope}

\draw (n11) -- (n31);
\draw (n12) -- (n32);
\draw (n13) -- (n33);

\draw (n41)--(n43);
\draw (n51)--(n53);

\draw (n11) -- (n13);
\draw (n21) -- (n23);
\draw (n31) -- (n33);

\draw (n14)--(n34);
\draw (n15)--(n35);

\draw(n15)--(n16);
\draw(n25)--(n26);
\draw(n35)--(n36);
\draw(n45)--(n46);
\draw(n55)--(n56);
\draw(n65)--(n66);

\draw(n51)--(n61);
\draw(n52)--(n62);
\draw(n53)--(n63);
\draw(n54)--(n64);
\draw(n55)--(n65);
\draw(n56)--(n66);

\draw(n61)--(n63);
\draw(n16)--(n36);

\nc\dashedline[2]{\draw[dashed] (#1)--(#2);}

\dashedline{n31}{n41}
\dashedline{n32}{n42}
\dashedline{n33}{n43}
\dashedline{n34}{n44}
\dashedline{n35}{n45}
\dashedline{n36}{n46}

\dashedline{n51}{n41}
\dashedline{n52}{n42}
\dashedline{n53}{n43}
\dashedline{n54}{n44}
\dashedline{n55}{n45}
\dashedline{n56}{n46}

\dashedline{n13}{n14}
\dashedline{n23}{n24}
\dashedline{n33}{n34}
\dashedline{n43}{n44}
\dashedline{n53}{n54}
\dashedline{n63}{n64}

\dashedline{n15}{n14}
\dashedline{n25}{n24}
\dashedline{n35}{n34}
\dashedline{n45}{n44}
\dashedline{n55}{n54}
\dashedline{n65}{n64}

\draw [->] (0,-8)--(54,-8);
\node at (56,-10) {$\zeta_1$};

\draw [->] (-8,0)--(-8,54);
\node at (-10,56) {$\zeta_2$};

\draw [dotted] (0,-8)--(n11);
\node at (0,-10) {$1$};
\draw [dotted] (6,-8)--(n21);
\node at (6,-11) {$\aleph_0$};
\draw [dotted] (12,-8)--(n31);
\node at (12,-11) {$\aleph_1$};
\draw [dotted] (22,-8)--(n41);
\node at (22,-11) {$\ka_1$};
\draw [dotted] (36,-8)--(n51);
\draw [dotted] (42,-8)--(n61);
\node at (36,-11) {\small$|X|$};
\node at (42,-11) {\small$|X|^+$};

\draw [dotted] (-8,0)--(n11);
\node at (-10,0) {$1$};
\draw [dotted] (-8,6)--(n12);
\node at (-11,6) {$\aleph_0$};
\draw [dotted] (-8,12)--(n13);
\node at (-11,12) {$\aleph_1$};
\draw [dotted] (-8,26)--(n14);
\node at (-11,26) {$\ka_2$};
\draw [dotted] (-8,36)--(n15);
\draw [dotted] (-8,42)--(n16);
\node at (-11,36) {\small$|X|$};
\node at (-11,42) {\small$|X|^+$};

\foreach \x in {0,6,12,36,42} {\draw[thick, fill=white] (\x,\x) circle (.6);}

\end{tikzpicture}
}
\caption{Hasse diagram of the layer $\Lay_1(\S_1)$. The $*$-congruences are indicated by white vertices; cf.~Subsection \ref{subsect:*}.}
\label{fig:1layer}

\end{center}
\end{figure}

If $L_1$ and $L_2$ are two layers of congruences (as above), we write $L_1\leq L_2$ if there exists $\si\in L_1$ and $\tau\in L_2$ with $\si\sub\tau$.
Again by Theorem \ref{thm:comparisons}, we have $L_1\leq L_2$ precisely in the following situations:
\bit
\item
$L_1=\Lay_1(N)$ and $L_2=\Lay_1(N')$, with $N\pre N'$;
\item
$L_1=\Lay_1(N)$ and $L_2=\Lay_2(\eta,\xi_1,\dots,\xi_k,\eta_1,\dots,\eta_k)$;
\item
$L_1=\Lay_2(\eta,\xi_1,\dots,\xi_k,\eta_1,\dots,\eta_k)$ and
$L_2=\Lay_2(\eta',\xi_1',\dots,\xi_{k'}',\eta_1',\dots,\eta_{k'}')$, where $\eta\leq\eta'$, and
there exist $0\leq j_1\leq \dots\leq j_{k}\leq k'$ such that
$\xi_i\leq \xi_{j_i}'$ and $\eta_i\leq\eta_{j_i}'$ for each $i$.
\item
In the alternative viewpoint, $\Lay_2(\Psi)\leq\Lay_2(\Psi')$ if and only if $\Psi\pre\Psi'$.
\eit
If $L_1\leq L_2$, then the indexing set for $L_1$ contains that for $L_2$.
In each of the above cases the comparisons are as depicted in Figure \ref{fig:2layers}.

\begin{figure}[ht]

\begin{center}

\begin{tikzpicture}[x={(1.2mm,0.2mm)},y={(-0.9mm,0.3mm)},z={(0mm,1.2mm)},inner sep=0.7mm]

\coordinate  (n11) at (0,0);
\coordinate  (n21) at (6,0);
\coordinate  (n31) at (12,0);
\coordinate (n41) at (18,0);
\coordinate  (n51) at (24,0);
\coordinate (n61) at (30,0);
\coordinate  (n71) at (36,0);

\coordinate  (n12) at (0,6);
\coordinate (n22) at (6,6);
\coordinate (n32) at (12,6);
\coordinate (n42) at (18,6);
\coordinate (n52) at (24,6);
\coordinate (n62) at (30,6);
\coordinate (n72) at (36,6);

\coordinate (n13) at (0,12);
\coordinate (n23) at (6,12);
\coordinate (n33) at (12,12);
\coordinate (n43) at (18,12);
\coordinate (n53) at (24,12);
\coordinate (n63) at (30,12);
\coordinate (n73) at (36,12);

\coordinate  (n14) at (0,18);
\coordinate (n24) at (6,18);
\coordinate (n34) at (12,18);
\coordinate (n44) at (18,18);
\coordinate (n54) at (24,18);
\coordinate (n64) at (30,18);
\coordinate (n74) at (36,18);

\coordinate (n15) at (0,24);
\coordinate (n25) at (6,24);
\coordinate (n35) at (12,24);
\coordinate (n45) at (18,24);
\coordinate (n55) at (24,24);
\coordinate (n65) at (30,24);
\coordinate (n75) at (36,24);

\coordinate  (n16) at (0,30);
\coordinate (n26) at (6,30);
\coordinate (n36) at (12,30);
\coordinate (n46) at (18,30);
\coordinate (n56) at (24,30);
\coordinate (n66) at (30,30);
\coordinate (n76) at (36,30);

\coordinate  (n17) at (0,36);
\coordinate (n27) at (6,36);
\coordinate (n37) at (12,36);
\coordinate (n47) at (18,36);
\coordinate (n57) at (24,36);
\coordinate (n67) at (30,36);
\coordinate (n77) at (36,36);

\fill [black!15,draw=black] (n11)--(n71)--(n77)--(n17)--(n11);
\draw (n12)--(n72);
\draw (n13)--(n73);
\draw (n14)--(n74);
\draw (n15)--(n75);
\draw (n16)--(n76);
\draw (n17)--(n77);

\draw (n21)--(n27);
\draw (n31)--(n37);
\draw (n41)--(n47);
\draw (n51)--(n57);
\draw (n61)--(n67);
\draw (n71)--(n77);

\coordinate  (m11) at (18,18,10);
\coordinate  (m21) at (24,18,10);
\coordinate  (m31)  at (30,18,10);
\coordinate  (m41) at (36,18,10);

\coordinate  (m12) at (18,24,10);
\coordinate (m22) at (24,24,10);
\coordinate (m32) at (30,24,10);
\coordinate (m42) at (36,24,10);

\coordinate   (m13) at (18,30,10);
\coordinate (m23) at (24,30,10);
\coordinate (m33) at (30,30,10);

\coordinate (m43) at (36,30,10);

\coordinate  (m14) at (18,36,10);
\coordinate (m24) at (24,36,10);
\coordinate (m34) at (30,36,10);
\coordinate (m44) at (36,36,10);

\draw (n44) -- (m11);
\draw (n45) -- (m12);
\draw (n46) -- (m13);
\draw (n47) -- (m14);

\draw (n54) -- (m21);
\draw (n55) -- (m22);
\draw (n56) -- (m23);
\draw (n57) -- (m24);

\draw (n64) -- (m31);
\draw (n65) -- (m32);
\draw (n66) -- (m33);
\draw (n67) -- (m34);

\draw (n74) -- (m41);
\draw (n75) -- (m42);
\draw (n76) -- (m43);
\draw (n77) -- (m44);

\draw [fill=black!15] (m11)--(m41)--(m44)--(m14)--(m11);

\draw (m12)--(m42);
\draw (m13)--(m43);

\draw (m21)--(m24);
\draw (m31)--(m34);

\draw [->] (0,-7)-- (42,-7);
\node at (44,-8) {$\zeta_1$};
\draw [dotted] (0,-7)--(n11);
\node at (-2,-13) {$\ka_1$};
\draw [dotted] (18,-7)--(n41);
\node at (16,-13) {$\ka_2$};
\draw [dotted] (36,-7)--(n71);
\node at (34,-13) {$|X|^+$};

\draw [->] (-7,0)-- (-7,42);
\node at (-8,44) {$\zeta_2$};
\draw [dotted] (-7,0)--(n11);
\node at (-12,-6) {$\ka_1$};
\draw [dotted] (-7,18)--(n14);
\node at (-12,12) {$\ka_2$};
\draw [dotted] (-7,36)--(n17);
\node at (-13,29) {$|X|^+$};

\node at (8,44) {$L_1$};

\node at (46,70) {$L_2$};

\end{tikzpicture}

\caption{Comparisons between two layers $L_1\leq L_2$, with indexing sets having the smallest elements~$\ka_1\leq\ka_2$.}

\label{fig:2layers}

\end{center}

\end{figure}
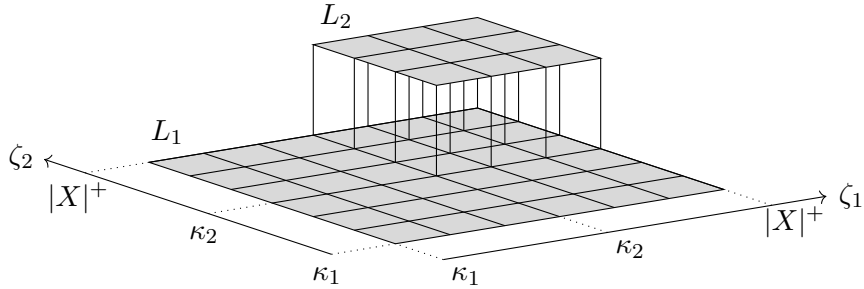

Putting the above information together yields a visual representation of the sublattice of type~\ref{CT1} congruences, as shown in Figure \ref{fig:types12}.
From this diagram, one can see that $\Cong(\M)$ has precisely three atoms; these are $\lam_1^1\cap\rho_{\aleph_0}^1$, $\lam_{\aleph_0}^1\cap\rho_1^1$ and $\lam_1^2\cap\rho_1^2$.

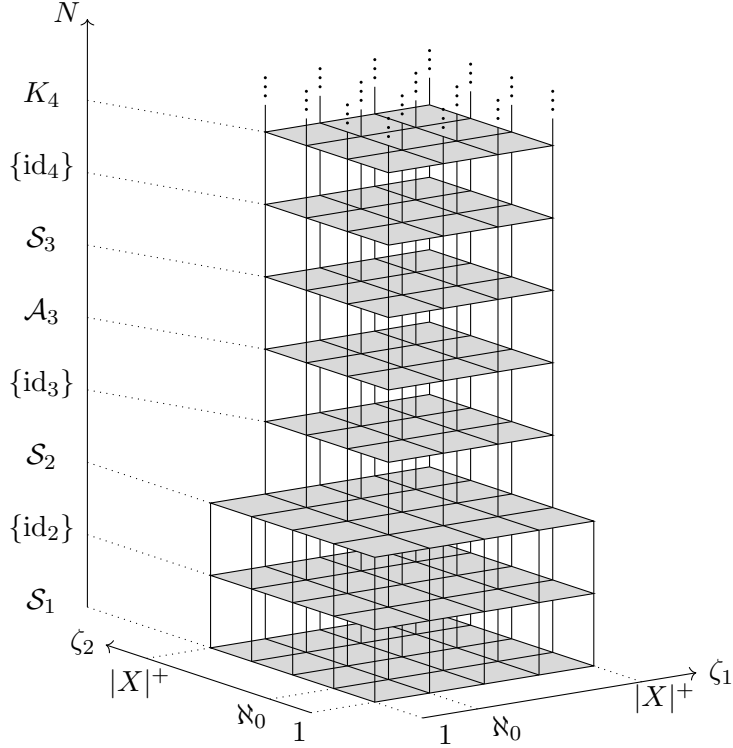
\begin{figure}[ht]

\begin{center}

\begin{tikzpicture}[x={(1.2mm,0.2mm)},y={(-0.9mm,0.3mm)},z={(0mm,1.2mm)},inner sep=0.7mm]

\foreach \z in {0,1}
  {\fill [black!15,draw=black,fill opacity=1] (0,0,8*\z)--(24,0,8*\z)--(24,24,8*\z)--(0,24,8*\z)--(0,0,8*\z);
  \foreach \x in {1,...,3} \draw [black] (6*\x, 0,8*\z)--(6*\x,24,8*\z);
  \foreach \y in {1,...,3} \draw [black] (0,6*\y,8*\z)--(24,6*\y,8*\z);
  \foreach \x in {0,...,4}
  \foreach \y in {0,...,4}
     \draw (6*\x,6*\y, 8*\z)--(6*\x,6*\y, 8*\z +8); }

\fill [black!15,draw=black,fill opacity=1] (0,0,16)--(24,0,16)--(24,24,16)--(0,24,16)--(0,0,16);
  \foreach \x in {1,...,3} \draw [black] (6*\x, 0,16)--(6*\x,24,16);
  \foreach \y in {1,...,3} \draw [black] (0,6*\y,16)--(24,6*\y,16);

\foreach \z in {3,...,7} {
  \foreach \x in {1,...,4} 
  \foreach \y in {1,...,4}
     \draw (6*\x,6*\y, 8*\z)--(6*\x,6*\y, 8*\z -8); 
  \fill [black!15,draw=black,fill opacity=1] (6,6,8*\z)--(24,6,8*\z)--(24,24,8*\z)--(6,24,8*\z)--(6,6,8*\z);
  \foreach \x in {2,3} \draw [black] (6*\x, 6,8*\z)--(6*\x,24,8*\z);
  \foreach \y in {2,3} \draw [black] (6,6*\y,8*\z)--(24,6*\y,8*\z);
}

\foreach \x in {1,...,4} 
\foreach \y in {1,...,4} {
 \draw (6*\x,6*\y, 56)--(6*\x,6*\y, 59);
\fill (6*\x,6*\y,60) circle [radius=0.2mm];
\fill (6*\x,6*\y,61) circle [radius=0.2mm];
\fill (6*\x,6*\y,62) circle [radius=0.2mm];
}

\draw [->] (0,-7,0)-- (30,-7,0);
\node at (32,-8,0) {$\zeta_1$};
\draw [dotted] (0,-7,0)--(0,0,0);
\node at (-2,-13,0) {$1$};
\draw [dotted] (6,-7,0)--(6,0,0);
\node at (4,-14,0) {$\aleph_0$};
\draw [dotted] (24,-7,0)--(24,0,0);
\node at (22,-13,0) {$|X|^+$};

\draw [->] (-7,0,0)-- (-7,30,0);
\node at (-8,32,0) {$\zeta_2$};
\draw [dotted] (-7,0,0)--(0,0,0);
\node at (-12,-5,0) {$1$};
\draw [dotted] (-7,6,0)--(0,6,0);
\node at (-12,2,0) {$\aleph_0$};
\draw [dotted] (-7,24,0)--(0,24,0);
\node at (-13,17,0) {$|X|^+$};

\draw [->] (0,42,0)--(0,42,65);
\node at (0,45,65) {$N$};
\foreach \z in {0,...,2} 
  \draw [dotted] (0,42,8*\z)--(0,24,8*\z);
\foreach \z in {3,...,7} 
  \draw [dotted] (0,42,8*\z)--(6,24,8*\z);

\node at (-2,46,0) {$\S_1$};
\node at (-2,46,8) {$\{\id_2\}$};
\node at (-2,46,16) {$\S_2$};
\node at (-2,46,24) {$\{\id_3\}$};
\node at (-2,46,32) {$\mathcal{A}_3$};
\node at (-2,46,40) {$\S_3$};
\node at (-2,46,48) {$\{\id_4\}$};
\node at (-2,46,56) {$K_4$};

\end{tikzpicture}

\caption{Hasse diagram of congruences of type \ref{CT1}.  Here $\mathcal{A}_n$ denotes the alternating group, and $K_4=\{\id_4,(1,2)(3,4),(1,3)(2,4),(1,4)(2,3)\}$ the Klein $4$-group.}

\label{fig:types12}

\end{center}

\end{figure}

Hasse diagrams for type \ref{CT2} congruences are more complicated, primarily because the~$\pre$ order on reversals given in \eqref{eq:prePsi} (cf.~Theorem \ref{thm:comparisons} \ref{it:comp3} and Corollary \ref{cor:CT2comp}) is not a total order in general.  However, one may readily visualise the $\pre$ order for $|X|=\aleph_n$ for small natural numbers~$n$.  To do so, note that when $|X|=\aleph_n$, we have $|X|^+=\aleph_{n+1}$ and $\eta\in[\aleph_0,\aleph_{n+1}]$, so any reversal from $\cR_{\aleph_n}$ is of the form
\begin{equation}\label{eq:Psi_aleph}
\Psi\colon \{\aleph_k,\ldots,\aleph_n\}\to\{1,\aleph_0,\aleph_1,\ldots,\aleph_k\} \qquad\text{for some $k\in\{0,1,\ldots,n+1\}$.}
\end{equation}
Hasse diagrams of the lattices $(\cR_{\aleph_n},\pre)$ are given in Figure~\ref{fig:Psi} for $n=0,1,2$; for convenience, in the figure a reversal $\Psi$ as in \eqref{eq:Psi_aleph} is depicted as a tuple $(\Psi(\aleph_k),\ldots,\Psi(\aleph_n))$.
From such a diagram, we may deduce the Hasse diagram of \ref{CT2} congruences by inserting appropriate copies of Figure~\ref{fig:2layers}; this is done in Figure \ref{fig:CT2} for $|X|=\aleph_0$ and~$\aleph_1$.  Figure \ref{fig:CongPX} gives the Hasse diagram of the entire lattice $\Cong(\M)$ for $|X|=\aleph_2$.
As indicated by these diagrams, $\Cong(\M)$ has precisely one co-atom, namely the congruence $R_{|X|} \cup \mu_{|X|}^{|X|^+} = \left(\lam_{|X|^+}^{|X|}\cap\rho_{|X|^+}^{|X|}\right)\cup\mu_{|X|}^{|X|^+}$, corresponding to
the greatest non-empty reversal, which is $(|X|)$ in the above notation.

\begin{figure}[ht]
\begin{center}
\begin{tikzpicture}[scale=0.9]
\begin{scope}
\node (11) at (0,0) {$(1)$};
\node (12) at (0,1) {$(\aleph_0)$};
\node (13) at (0,2) {$\emptyset$};
\draw(11)--(12)--(13);
\end{scope}
\begin{scope}[shift={(3.5,0)}]
\node (11) at (0,0) {$(1,1)$};
\node (12) at (0,1) {$(\aleph_0,1)$};
\node (13) at (0,2) {$(\aleph_0,\aleph_0)$};
\node (21) at (2,2) {$(1)$};
\node (22) at (2,3) {$(\aleph_0)$};
\node (23) at (2,4) {$(\aleph_1)$};
\node (24) at (2,5) {$\emptyset$};
\draw(11)--(12)--(13) (12)--(21) (13)--(22) (21)--(22)--(23)--(24);
\end{scope}
\begin{scope}[shift={(9,0)}]
\node (11) at (0,0) {$(1,1,1)$};
\node (12) at (0,1) {$(\aleph_0,1,1)$};
\node (13) at (0,2) {$(\aleph_0,\aleph_0,1)$};
\node (14) at (0,3) {$(\aleph_0,\aleph_0,\aleph_0)$};
\node (21) at (2,2) {$(1,1)$};
\node (22) at (2,3) {$(\aleph_0,1)$};
\node (23) at (2,4) {$(\aleph_0,\aleph_0)$};
\node (31) at (4,4) {$(\aleph_1,1)$};
\node (32) at (4,5) {$(\aleph_1,\aleph_0)$};
\node (33) at (4,6) {$(\aleph_1,\aleph_1)$};
\node (41) at (6,5) {$(1)$};
\node (42) at (6,6) {$(\aleph_0)$};
\node (43) at (6,7) {$(\aleph_1)$};
\node (44) at (6,8) {$(\aleph_2)$};
\node (45) at (6,9) {$\emptyset$};
\draw(11)--(12)--(13)--(14) (12)--(21) (13)--(22) (14)--(23) (21)--(22)--(23) (22)--(31) (23)--(32) (31)--(32)--(33) (31)--(41) (32)--(42) (33)--(43) (41)--(42)--(43)--(44)--(45);
\end{scope}
\end{tikzpicture}
\caption{Hasse diagrams of the reversal posets $(\cR_{\aleph_n},\pre)$ for $n=0$ (left), $n=1$ (middle) and $n=2$ (right).}
\label{fig:Psi}
\end{center}
\end{figure}
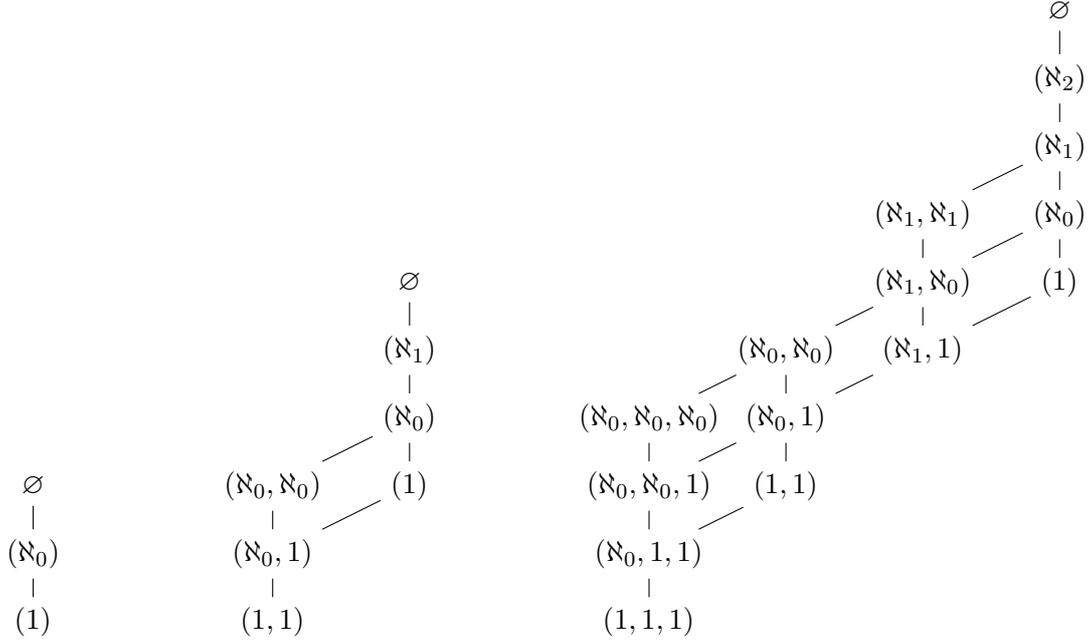

\begin{figure}[ht]
\begin{center}
\begin{tikzpicture}[scale=0.9,x={(1.2mm,0.2mm)},y={(-0.9mm,0.3mm)},z={(0mm,1.2mm)},inner sep=0.7mm]
\begin{scope}
\congsquare0001
\congsquareconnections0000011
\congsquare0011
\congsquareconnections1111120
\fill (6,6,16) circle [radius=.5mm];
\end{scope}
\begin{scope}[shift={(30,-24)}]]
\congsquare0002
\congsquareconnections0000012
\congsquare0012
\congsquareconnections0010022
\congsquareconnections1114121
\congsquare0022
\congsquare4121
\congsquareconnections4124131
\congsquareconnections1124131
\congsquare4131
\congsquareconnections4134141
\congsquare4141
\congsquareconnections5245250
\fill (6*5,6*2,8*5) circle [radius=.5mm];
\end{scope}
\end{tikzpicture}
\caption{Hasse diagram of congruences of type \ref{CT2} for $|X|$ equal to $\aleph_0$ (left) and $\aleph_1$ (right).}
\label{fig:CT2}
\end{center}
\end{figure}
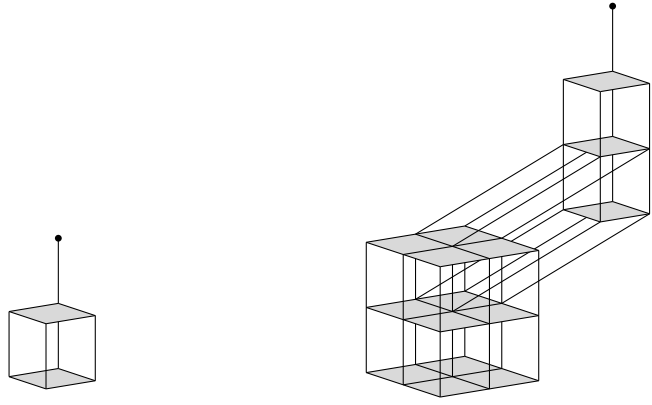

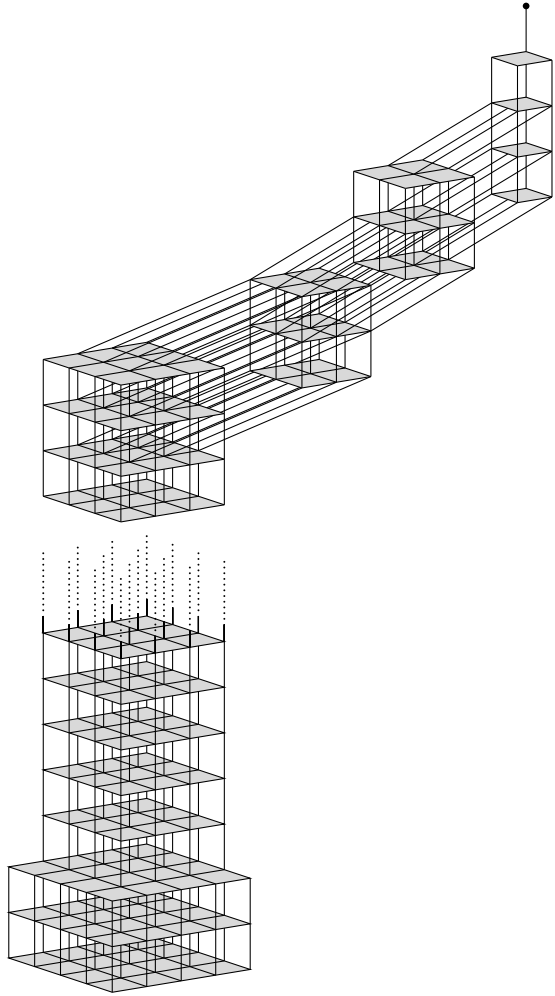
\begin{figure}[ht]
\begin{center}
\begin{tikzpicture}[scale=0.9,x={(1.2mm,0.2mm)},y={(-0.9mm,0.3mm)},z={(0mm,1.2mm)},inner sep=0.7mm,scale=.7]

\foreach \z in {0,1}
  {\fill [black!15,draw=black,fill opacity=1] (0,0,8*\z)--(24,0,8*\z)--(24,24,8*\z)--(0,24,8*\z)--(0,0,8*\z);
  \foreach \x in {1,...,3} {\draw [black] (6*\x, 0,8*\z)--(6*\x,24,8*\z);}
  \foreach \y in {1,...,3} {\draw [black] (0,6*\y,8*\z)--(24,6*\y,8*\z);}
  \foreach \x in {0,...,4}
  \foreach \y in {0,...,4}
     \draw (6*\x,6*\y, 8*\z)--(6*\x,6*\y, 8*\z +8); }

\fill [black!15,draw=black,fill opacity=1] (0,0,16)--(24,0,16)--(24,24,16)--(0,24,16)--(0,0,16);
  \foreach \x in {1,...,3} \draw [black] (6*\x, 0,16)--(6*\x,24,16);
  \foreach \y in {1,...,3} \draw [black] (0,6*\y,16)--(24,6*\y,16);

\foreach \z in {3,...,7} {
  \foreach \x in {1,...,4} 
  \foreach \y in {1,...,4}
     \draw (6*\x,6*\y, 8*\z)--(6*\x,6*\y, 8*\z -8); 
  \fill [black!15,draw=black,fill opacity=1] (6,6,8*\z)--(24,6,8*\z)--(24,24,8*\z)--(6,24,8*\z)--(6,6,8*\z);
  \foreach \x in {2,3} \draw [black] (6*\x, 6,8*\z)--(6*\x,24,8*\z);
  \foreach \y in {2,3} \draw [black] (6,6*\y,8*\z)--(24,6*\y,8*\z);
}

\foreach \x in {1,...,4} 
\foreach \y in {1,...,4}
\foreach \z in {0,...,10} {
 \draw (6*\x,6*\y, 56)--(6*\x,6*\y, 59);
\fill (6*\x,6*\y,60+\z) circle [radius=0.2mm];
}

\begin{scope}[shift={(6,6,80)}]
\congsquare0003
\congsquareconnections0000013
\congsquare0013
\congsquareconnections0010023
\congsquareconnections1116122
\congsquare0023
\congsquareconnections0020033
\congsquareconnections1126132
\congsquare0033
\congsquareconnections1136142

\congsquare6122
\congsquareconnections6126132
\congsquare6132
\congsquareconnections6136142
\congsquareconnections613{9}142
\congsquare6142
\congsquareconnections614{9}152
\congsquare{9}142
\congsquareconnections{9}14{9}152
\congsquareconnections{10}24{13}251
\congsquare{9}152
\congsquareconnections{9}15{9}162
\congsquareconnections{10}25{13}261
\congsquare{9}162
\congsquareconnections{10}26{13}271

\congsquare{13}251
\congsquareconnections{13}25{13}261
\congsquare{13}261
\congsquareconnections{13}26{13}271
\congsquare{13}271
\congsquareconnections{13}27{13}281
\congsquare{13}281

\congsquareconnections{14}38{14}390

\fill (14*6,3*6,9*8) circle [radius=.7mm];
\end{scope}

\end{tikzpicture}
\caption{Hasse diagram of $\Cong(\P_X)\cong\Cong(\PB_X)$ where $|X|=\aleph_2$.}
\label{fig:CongPX}
\end{center}
\end{figure}

\section{Properties of the lattice}\label{sect:global}

In Section \ref{sect:order} we described the inclusion order and the meet and join operations on the congruence lattices $\Cong(\P_X)$ and $\Cong(\PB_X)$.  In this section, we use these descriptions to deduce some ``global'' properties of the lattices.  Specifically, we show that they are distributive in Subsection~\ref{subsect:dist}, and well quasi-ordered in Subsection~\ref{subsect:wqo}; we also describe in Subsection~\ref{subsect:*} the sublattice of $*$-congruences: i.e., the congruences that also preserve the involution.  As usual, throughout this section, $X$ is an arbitrary infinite set and $\M$ stands for either the partition monoid $\P_X$ or the partial Brauer monoid~$\PB_X$.

\subsection{Distributivity}\label{subsect:dist}

We now use Theorem \ref{thm:mj}, which describes meets and joins of congruences on $\M$, to show that the lattice $\Cong(\M)$ is distributive.  
In the following proof, we use the fact that totally ordered sets are distributive lattices with meet and join given by minimum and maximum, respectively; in particular,
\[
\min(a,\max(b,c)) = \max(\min(a,b),\min(a,c)) \AND \max(\min(a,b),a)=a ,
\]
for arbitrary $a,b,c$ belonging to a totally ordered set.

\begin{thm}
\label{thm:distr}
The congruence lattices of $\P_X$ and $\PB_X$, with $X$ infinite, are distributive.
\end{thm}

\begin{proof}
As is well known, it suffices to prove that meet distributes over join.  Thus, let $\si_1,\si_2,\si_3$ be any three congruences on $\M$, and write $\tau_1 = \si_1\wedge(\si_2\vee\si_3)$ and $\tau_2 = (\si_1\wedge\si_2)\vee(\si_1\wedge\si_3)$.  We prove that $\tau_1=\tau_2$ by showing that they have the same type, and then that the values of all the relevant parameters are equal for the two congruences, which is accomplished by repeated application of Theorem \ref{thm:mj}.  
We begin by observing that
\[
\eta(\tau_1)=\min(\eta(\si_1),\max(\eta(\si_2),\eta(\si_3)))
=\max(\min(\eta(\si_1),\eta(\si_2)),\min(\eta(\si_1),\eta(\si_3)))=
\eta(\tau_2).
\]

In particular, $\tau_1$ and $\tau_2$ are of the same type.
Also, when they are of type \ref{CT1}, $n(\tau_1)=n(\tau_2)$.
The proof that $\zeta_1(\tau_1)=\zeta_1(\tau_2)$ and
$\zeta_2(\tau_1)=\zeta_2(\tau_2)$ is identical to the above proof for $\eta$.
Equality of the remaining parameters depends on the type of $\tau_1$ and $\tau_2$,
which in turn depends on the types of $\si_1,\si_2,\si_3$.
Thus there are eight cases, and in each case it is just a matter of following through the formulas for meets and joins given in Theorem \ref{thm:mj}.  In fact, since $\vee$ is commutative, we may assume that $\eta(\si_2)\leq\eta(\si_3)$, which reduces the number of cases to six.  
As an illustrative sample, we treat a couple of cases here.  The other four are dealt with in a similar fashion.

If $\si_1$ and $\si_2$ are of type \ref{CT1}, and $\si_3$ of type \ref{CT2}, then $\tau_1$ and $\tau_2$ are of type \ref{CT1}; moreover, $\si_2\vee\si_3$ is of type \ref{CT2}, and $\si_1\wedge\si_2$ and $\si_1\wedge\si_3$  are of type \ref{CT1}, and we have ${N(\tau_1) = N(\si_1) = \max( \min(N(\si_1),N(\si_2)) , N(\si_1) ) = N(\tau_2)}$.

If $\si_1$, $\si_2$ and $\si_3$ are all of type \ref{CT2}, then so too are $\tau_1$ and $\tau_2$, and Proposition~\ref{prop:Psi_lattice} gives $\Psi(\tau_1) = \Psi(\si_1) \wedge (\Psi(\si_2)\vee\Psi(\si_3)) = (\Psi(\si_1)\wedge\Psi(\si_2)) \vee (\Psi(\si_1)\wedge\Psi(\si_3)) = \Psi(\tau_2)$.
\end{proof}

\subsection{Well quasi-orderedness}\label{subsect:wqo}

After Corollary \ref{cor:ideals_M} we made the observation that the ideals of $\M$ are (totally) well-ordered by inclusion. It immediately follows that the set of Rees congruences on $\M$ has the same property.  The lattice $\Cong(\M)$ of all congruences is certainly not totally ordered by inclusion.  Nonetheless, it satisfies the following:

\begin{thm}
\label{thm:wqo}
The lattice of congruences of infinite $\P_X$ or $\PB_X$ is well quasi-ordered under inclusion, meaning that it has no infinite strictly descending chains and no infinite antichains. 
\end{thm} 

\begin{proof}
Let $\Cong_1(\M)$ and $\Cong_2(\M)$ be the sublattices of $\Cong(\M)$ consisting of all congruences of type \ref{CT1} or \ref{CT2}, respectively.  Clearly it suffices to prove that these are both wqo.

Recall that the set $\NN=\set{ N }{ N\unlhd \S_n \text{ for some } n\in [1,\aleph_0)}$ is totally ordered under $\pre$, as defined in \eqref{eq:N_pre}.  For $N\in\NN$, we write $n(N)$ for the degree of the permutations from $N$ (i.e.,~$N\leq\S_{n(N)}$).  If $\Psi\in\cR$ is a non-empty reversal, we write $\eta(\Psi)=\min(\dom\Psi)$; we also define~$\eta(\emptyset)=|X|^+$.

We begin with $\Cong_1(\M)$. By Theorem \ref{thm:comparisons} \ref{it:comp1} (cf.~\eqref{eq:CT1_cont}) it is isomorphic to
\begin{align*}
C_1 ={} & \bigset{ (N,\zeta_1,\zeta_2) }{ N\in\NN,\ n(N)\leq2,\ \zeta_1,\zeta_2\in \{1\}\cup [\aleph_0,|X|^+]} \\[2truemm]
{}\cup{} &\bigset{ (N,\zeta_1,\zeta_2) }{N\in\NN,\  n(N)\geq3,\ \zeta_1,\zeta_2\in  [\aleph_0,|X|^+]},
\end{align*}
under the ordering
\[
(N,\zeta_1,\zeta_2) \leq (N',\zeta_1',\zeta_2') \quad\iff\quad
\text{$N\pre N'$, \ $\zeta_1\leq\zeta_1'$ \ and \  $\zeta_2\leq\zeta_2'$.}
\]
The poset $C_1$ is in turn a subposet of $D_1=\NN\times [1,|X|^+]\times [1,|X|^+]$ under the component-wise ordering.  Since $\NN$ and $[1,|X|^+]$ are wqo, so too is $D_1$ by Lemma \ref{lem:dixon}. It now follows that $C_1$ is wqo, and hence  $\Cong_1(\M)$ is wqo as well.

Let us now turn to $\Cong_2(\M)$.  By Corollary \ref{cor:CT2comp}, it is isomorphic to
\[
C_2 = \bigset{(\Psi,\ze_1,\ze_2)}{\Psi\in\cR,\ \ze_1,\ze_2\in[\eta(\Psi),|X|^+]},
\]
under the ordering 
\[
(\Psi,\ze_1,\ze_2) \leq (\Psi',\ze_1',\ze_2') \quad\iff\quad \text{$\Psi\pre \Psi'$, \ $\zeta_1\leq\zeta_1'$ \ and \  $\zeta_2\leq\zeta_2'$.}
\]
The poset $C_2$ is in turn a subposet of $D_2=\cR\times [1,|X|^+]\times [1,|X|^+]$ under the component-wise ordering.  Since $\cR$ and $[1,|X|^+]$ are wqo (the former by Corollary~\ref{cor:Rwqo}), so too is $D_2$ by Lemma~\ref{lem:dixon}. It now follows that $C_2$ is wqo, and hence  $\Cong_2(\M)$ is wqo as well.
\end{proof}

\subsection[The $*$-congruence lattice]{\boldmath The $*$-congruence lattice}\label{subsect:*}

Recall that a congruence $\si$ on a regular $*$-semigroup $S$ is a \emph{$*$-congruence} if it is also compatible with the involution of $S$: i.e., if $(x,y)\in\si\implies(x^*,y^*)\in\si$ for all $x,y\in S$.  In light of the identity~$x=x^{**}$, this definition is equivalent to $\si$ being equal to $\si^*=\set{(x^*,y^*)}{(x,y)\in\si}$.  The meet and join in $\Cong(S)$ of two $*$-congruences is easily checked to be a $*$-congruence, and it follows that the set $\Cong^*(S)$ of all $*$-congruences is a sublattice of $\Cong(S)$; 
this also follows from the general result that the congruence lattice of any 
(universal) algebra is a sublattice of the lattice of equivalence relations on the 
carrier set \cite[Theorem II.5.3]{BS1981}.

For any cardinal $1\leq\ze\leq|X|^+$, and for any normal subgroup $N$ of some finite $\S_n$, we have
\[
\lam_\ze^*=\rho_\ze \COMMA
\rho_\ze^*=\lam_\ze \COMMA
\mu_\ze^*=\mu_\ze \COMMA
R_\ze^*=R_\ze \COMMA
\nu_N^*=\nu_N.
\]
Together with Theorems \ref{thm-main}, \ref{thm:distr} and \ref{thm:wqo}, these observations quickly lead to the following:

\begin{thm}\label{thm:*}
Let $\M$ be either $\P_X$ or $\PB_X$, where $X$ is an infinite set, and let $\si$ be a congruence on $\M$.  Then $\si$ is a $*$-congruence if and only if $\ze_1(\si)=\ze_2(\si)$.  The $*$-congruence lattice $\Cong^*(\M)$ is distributive and well quasi-ordered.  \epfres
\end{thm}

The inclusion order on $*$-congruences, and formulae for meets and joins, are of course all still given by Theorems \ref{thm:comparisons} and \ref{thm:mj}.  Hasse diagrams for $\Cong^*(\M)$ are obtained from those of $\Cong(\M)$ by replacing each layer with its vertical diagonal, as indicated in Figure \ref{fig:1layer}.

\section{Minimal generation of congruences}\label{sect:gen}

If $S$ is a semigroup and $\Om$ a subset of $S\times S$, we denote by $\Om^\sharp$ the congruence on $S$ generated by $\Om$: i.e., the least congruence containing $\Om$.  If $\Om=\big\{(x,y)\big\}$ consists of a single pair, we write $\cg xy=\Om^\sharp$, and refer to this as a \emph{principal congruence}.  The \emph{(congruence) rank} of a congruence~$\si$, denoted $\crank(\si)$, is the least cardinality of a subset $\Om\sub S\times S$ such that $\si=\Om^\sharp$.

In this section we explore these ideas for congruences on $\M$, which as usual denotes $\P_X$ or~$\PB_X$ for some infinite set $X$.  Specifically, we classify the principal congruences in Subsection~\ref{subsect:pc} (see Theorem \ref{thm:pc}), and we calculate the ranks of all congruences in Subsection \ref{subsect:crank} (see Theorems \ref{thm:ranksCT1} and~\ref{thm:ranksCT2}).

\subsection{Principal congruences}\label{subsect:pc}

The next result classifies all principal congruences $\cg\al\be$ on $\M$; since $\cg\al\be=\cg\be\al$, it suffices to consider pairs $(\al,\be)$ with $\rank(\al)\geq\rank(\be)$.  

\newpage

\begin{thm}\label{thm:pc}
Let $\M$ be either $\P_X$ or $\PB_X$, where $X$ is an infinite set, and let $\al,\be\in\M$ with $\rank(\al)\geq\rank(\be)$.
\ben

\item 
\label{it:pc1}
If $\al=\be$, then $\cg\al\be=\De_{\M}$.

\item 
\label{it:pc2}
If $n=\rank(\al)<\aleph_0$, $(\al,\be)\in{\H}$ and $\al\not=\be$, then $\cg\al\be=\lam_\ze^N\cap\rho_\ze^N$, where $N\normal\S_n$ is normally generated by $\phi(\al,\be)$, and
\[
\ze=\begin{cases}
1&\text{if $n=2$}\\
\aleph_0 &\text{if $n\geq3$.}
\end{cases}
\]

\item 
\label{it:pc3}
If $n=\rank(\al)\leq1$ and $(\al,\be)\not\in{\H}$, then $\cg\al\be=\lam_{\ze_1}^{n+1}\cap\rho_{\ze_2}^{n+1}$, where
\[ 
\ze_1 = \begin{cases}
1 &\text{if }\ \overline{\alpha}=\overline{\beta}\\
\max(\aleph_0,|\ol\al\sd\ol\be|^+) &\text{otherwise}
\end{cases}
\AND
\ze_2 = \begin{cases}
1 &\text{if }\ \underline{\alpha}=\underline{\beta}\\
\max(\aleph_0,|\ul\al\sd\ul\be|^+) &\text{otherwise.}
\end{cases}
\]

\item 
\label{it:pc4}
If $2\leq n=\rank(\al)<\aleph_0$ and $(\al,\be)\not\in{\H}$, then $\cg\al\be=\lam_{\ze_1}^{n+1}\cap\rho_{\ze_2}^{n+1}$, where
\[ 
\ze_1 = \max(\aleph_0,
|\ol\al\sd\ol\be|^+) 
\AND
\ze_2 = 
\max(\aleph_0,|\ul\al\sd\ul\be|^+).
\]

\item 
\label{it:pc5}
If $\ka=\rank(\al)\geq\aleph_0$ and $|\al\sd\be|\geq\ka$, then $\cg\al\be = \lam_{\ze_1}^{\ka^+}\cap\rho_{\ze_2}^{\ka^+} = (\lam_{\ze_1}^{\ka^+}\cap\rho_{\ze_2}^{\ka^+})\cup\mu_1^{|X|^+}$, where
$\ze_1=\max(\ka^+,|\ol\al\sd\ol\be|^+)$
and
$\ze_2=\max(\ka^+,|\ul\al\sd\ul\be|^+)$.

\item 
\label{it:pc6}
If $\ka=\rank(\al)\geq\aleph_0$ and $0<|\al\sd\be|<\ka$, then
${\cg\al\be = \mu_\eta^{\ka^+} = (\lam_\eta^\eta\cap\rho_\eta^\eta)\cup\mu_\eta^{\ka^+}\cup\mu_1^{|X|^+}}$,
where $\eta=\max(\aleph_0,|\al\sd\be|^+)$.

\een
\end{thm}

\pf
First, one may check that the stated cases exhaust all possible pairs $(\al,\be)$; in part \ref{it:pc2}, note that $(\al,\be)\in{\H}$ and $\al\not=\be$ together imply $n\geq2$.  The proof of each part then follows the same pattern.  We respectively write $\si$ and $\tau$ for $\cg\al\be$ and the congruence it is claimed to equal (each $\tau$ is a congruence by Theorem \ref{thm-main}).  Since $(\al,\be)\in\tau$, we have $\si\sub\tau$ in all cases.  To obtain the reverse containment, we estimate the relevant parameters of $\si$ using the information provided by the generating pair $(\al,\be)$, and then apply Theorem \ref{thm:comparisons}.

\pfitem{\ref{it:pc1}}  This is clear.

\pfitem{\ref{it:pc2}}  From $\si\sub\tau$, we immediately obtain $n(\si)\leq n(\tau)=n$.  From Lemma \ref{la330} (with $\ka=n$ and $q=n-1$) and then Lemma \ref{la329} (with $\ka=q=n-1$), we have $n(\si)\geq n$, and so $n(\si)=n$.  Since $N(\si)$ is a normal subgroup of $\S_n$ containing $\phi(\al,\be)$, we clearly have $N(\si)\supseteq N=N(\tau)$.  Finally, let $i\in\{1,2\}$.  If $n\leq2$, then clearly $\ze_i(\si)\geq1=\ze_i(\tau)$; if $n\geq3$, then Lemmas \ref{la328} and~\ref{la328b} give $\ze_i(\si)\geq\aleph_0=\ze_i(\tau)$.

\pfitem{\ref{it:pc3}}  This time Lemma \ref{la329} (and $\si\sub\tau$) gives $n(\si)=n(\tau)=n+1$, and we clearly have ${N(\si)\supseteq\{\id_{n+1}\}=N(\tau)}$.  
By Lemma \ref{la325} and transitivity, we have $(\wh\al,\wh\be)\in\si\restr_{D_0}$, and we note that ${|\ol{\;\!\wh\al\;\!}\sd\ol{\;\!\wh\be\;\!}|=|\ol\al\sd\ol\be|}$.
If $0<|\ol\al\sd\ol\be|<\aleph_0$, then $\ze_1(\si)\geq2$, and so Lemma \ref{la328} gives ${\ze_1(\si)\geq\aleph_0=\ze_1(\tau)}$.  Otherwise, clearly $\ze_1(\si)\geq|\ol\al\sd\ol\be|^+=\ze_1(\tau)$.  The inequality $\ze_2(\si)\geq\ze_2(\tau)$ is dual.

\pfitem{\ref{it:pc4}}  The proof is essentially identical to the previous part, but noting that also $\ze_1(\si)\geq\aleph_0$ for $|\ol\al\sd\ol\be|=0$, by Lemmas \ref{la328} and \ref{la328b}.

\pfitem{\ref{it:pc5}}  From $\si\sub\tau$, we have $\eta(\si)\leq\eta(\tau)=\ka^+$.  If $\rank(\be)<\ka=\rank(\al)$, then $\eta(\si)\geq\ka^+$ by definition.  If $\rank(\be)=\ka$, then $(\al,\be)\in\si\restr_{D_\ka}$; if we had $\eta(\si)<\ka^+$, then $\ka\geq\eta(\si)$, so Lemma~\ref{lem:tech1} gives $|\al\sd\be|<\ka$, a contradiction; so $\eta(\si)\geq\ka^+$ in this case also.  Thus, regardless of the value of $\rank(\be)$, we have $\eta(\si)=\ka^+=\eta(\tau)$.  

Next note that since $\Psi(\si)$ and $\Psi(\tau)$ are both maps $[\ka^+,|X|]\to\{1\}\cup[\aleph_0,\ka^+]$, and since $\Psi(\tau)$ maps every element of $[\ka^+,|X|]$ to $1$, we clearly have $\Psi(\si)\succeq\Psi(\tau)$.

It remains to show that $\ze_1(\si)\geq\ze_1(\tau)$, the case of~$\ze_2$ being dual.  By Lemma \ref{la328b} we have $\zeta_1(\sigma)\geq\eta(\sigma)=\kappa^+$.
The proof that $\ze_1(\si)\geq|\ol\al\sd\ol\be|^+$ is analogous to the corresponding step in part \ref{it:pc3} above.
It follows that ${\ze_1(\si)\geq\max(\ka^+,|\ol\al\sd\ol\be|^+)=\ze_1=\ze_1(\tau)}$.

\pfitem{\ref{it:pc6}}  First note that $\mu_\eta^{\ka^+} = (\lam_\eta^\eta\cap\rho_\eta^\eta)\cup\mu_\eta^{\ka^+}\cup\mu_1^{|X|^+}$ by Lemma \ref{lem:mu}.  Again, $\si\sub\tau$ gives ${\eta(\si)\leq\eta(\tau)=\eta=\max(\aleph_0,|\al\sd\be|^+)}$.  For the converse, we clearly have $\eta(\si)\geq\aleph_0$, since $\si$ is of type \ref{CT2} as $\rank(\al)\geq\aleph_0$.  If $|\al\sd\be|<\aleph_0$, then it also follows that $\eta(\si)\geq|\al\sd\be|^+$ in this case.  So now we assume that $|\al\sd\be|\geq\aleph_0$.  Since $\al\in D_\ka$ and $|\al\sd\be|<\ka$, we must have $\be\in D_\ka$ as well (or else $\ka$ transversals of $\al$ would belong to $\al\sm\be$).  Thus, $(\al,\be)\in\si\restr_{D_\ka}$.  Taking any disjoint subsets $Y,Z\sub X$ with $|Y|=\ka$ and $|Z|=|\al\sd\be|$, Lemma \ref{lem:tech2} gives $(\ep_{Y\cup Z},\ep_Y)\in\si$, from which it follows that $(\ep_Z,\emptypart)=(\ep_{Y\cup Z}\ep_Z,\ep_Y\ep_Z)\in\si$.  But $\rank(\ep_Z)>\rank(\emptypart)$, and so $\eta(\si)>\rank(\ep_Z)=|\al\sd\be|$, giving $\eta(\si)\geq|\al\sd\be|^+$ in this case also.  Thus, regardless of the value of $|\al\sd\be|$, $\eta(\si)\geq\max(\aleph_0,|\al\sd\be|^+)=\eta$, and so $\eta(\si)=\eta$.  From this, and using Lemma~\ref{la328b}, we also obtain $\ze_i(\si)\geq\eta(\si)=\eta=\ze_i(\tau)$ for $i=1,2$.

Note that $\Psi(\si)(\ka)\geq|\al\sd\be|^+$ by definition, since $(\al,\be)\in\si\restr_{D_\ka}$.  Since $|\al\sd\be|>0$, we also have $\Psi(\si)(\ka)>1$, so Lemma \ref{la339} gives $\Psi(\si)(\ka)\geq\aleph_0$.  Thus, ${\Psi(\si)(\ka)\geq\max(\aleph_0,|\al\sd\be|^+)=\eta}$.  Since $\Psi(\si)$ is order-reversing, for any $\vk\in[\eta,\ka]$, we have $\Psi(\si)(\vk)\geq\Psi(\si)(\ka)\geq\eta=\Psi(\tau)(\vk)$; since also $\Psi(\tau)(\vk)=1$ for all $\vk\in[\ka^+,|X|]$, it follows that $\Psi(\si)\succeq\Psi(\tau)$.
\epf

Note that the only infinite limit cardinal that can appear as a parameter in a principal congruence is $\aleph_0$. 
Specifically, we can have $\zeta_i(\sigma)=\aleph_0$ in cases
\ref{it:pc2}--\ref{it:pc4}, \ref{it:pc6}, and also~${\eta(\sigma)=\aleph_0}$
in case \ref{it:pc6}.

\subsection{Congruence ranks}\label{subsect:crank}

In the next two theorems we calculate the rank of each congruence on $\M$.  It turns out that congruences can have infinite (even uncountable) ranks; to describe these, we require the concept of cofinality.

Recall that a subset $Q$ of a poset $P$ is \emph{cofinal} if for every $p\in P$, there exists $q\in Q$ such that $q\geq p$.  
The \emph{cofinality} of $P$, denoted $\cof(P)$, is defined to be the least cardinality of a cofinal subset of $P$.  
Note that if $P$ does not have any maximal elements, then any cofinal subset $Q$ of $P$ satisfies the (ostensibly stronger) condition: for every $p\in P$, there exists $q\in Q$ such that $q> p$.

\begin{lemma}\label{lem:cof}
Suppose  $\si$ is a congruence on a semigroup $S$, and that $\si=\bigcup_{p\in P}\si^{(p)}$ where~$P$ is a well-ordered chain and $\set{\si^{(p)}}{p\in P}$ is a non-decreasing chain of proper subcongruences of~$\si$.
Then $\crank(\si)\geq\cof(P)$.
\end{lemma}

\pf
Suppose $\si=\Om^\sharp$ where $\Om\sub S\times S$ and $|\Om|=\crank(\si)$.  For each $(x,y)\in\Om$ write $q(x,y)=\min\set{p\in P}{(x,y)\in\si^{(p)}}$, and let $Q=\set{q(x,y)}{(x,y)\in\Om}$.  We claim that~$Q$ is cofinal in $P$.  Indeed, suppose to the contrary that there exists $p\in P$ such that $q<p$ for all~$q\in Q$: i.e., $q(x,y)<p$ for all $(x,y)\in\Om$.  Then by the chain assumption on the subcongruences it follows that $\Om\sub\si^{(p)}$, and hence $\si=\Om^\sharp\sub\si^{(p)}$, contradicting the fact that $\si^{(p)}$ is a proper subcongruence.  With the claim established, we have $\crank(\si)=|\Om|\geq|Q|\geq\cof(P)$.
\epf

If $\ze$ is a cardinal, we write $\cof(\ze)=\cof[0,\ze)$ for the cofinality of the set $[0,\ze)$ of all cardinals (strictly) less than $\ze$.
Clearly $\cof(\ze)=1$ if $\ze$ is a successor cardinal.  If $\ze$ is a limit cardinal, then $\aleph_0\leq\cof(\ze)\leq\ze$: for example, $\cof(\aleph_0)=\aleph_0=\cof(\aleph_\omega)$.  
The existence of uncountable cardinals~$\ze$ with $\cof(\ze)=\ze$ is unprovable in ZFC \cite[Theorem 12.12]{Jech2003}.

Note that the definition of $\cof(\ze)$ in the previous paragraph is not standard.  Indeed, $\cof(\ze)$ is usually defined to be the cofinality of the set of all \emph{ordinals} less than $\ze$ (which is in fact the usual definition of $\ze$ itself); see for example \cite[p31]{Jech2003}.  If $\ze$ is a limit cardinal, then the two notions coincide.  We have used the current definition so that successor cardinals will satisfy $\cof(\ze)=1$, which will simplify the statements of the following theorems.

On several occasions we will use without explicit reference the following two facts:
\bit
\item If $\xi,\ze$ are cardinals with $\xi<\ze$, then any cofinal subset of $[\xi,\ze)$ is also cofinal in $[0,\ze)$, from which it quickly follows that $\cof(\ze)=\cof[\xi,\ze)$.
\item If $\ze$ is an uncountable limit cardinal, and if $\Xi$ is cofinal in $[0,\ze)$, then $\set{\xi^+}{\xi\in\Xi,\ \xi\geq\aleph_0}$ is a cofinal subset of $[\aleph_0,\ze)$ consisting entirely of successor cardinals, and of the same size as $\Xi$.
\eit

We begin with type \ref{CT1} congruences:

\begin{thm}
\label{thm:ranksCT1}
Let $\si=\lambda_{\zeta_1}^N\cap\rho_{\zeta_2}^N$ be a congruence on~$\M$ of type \ref{CT1}, where $\M$ denotes either~$\P_X$ or $\PB_X$ for some infinite set $X$.
\begin{enumerate}[label=\textup{(\Roman*)},leftmargin=7mm]
\item
\label{it:ranks2}
If at least one of $\zeta_1,\zeta_2$ is an uncountable limit cardinal, then $\si$ is not finitely generated and 
\[
\crank(\sigma)={\max}\big({\cof(\zeta_1)},\cof(\zeta_2)\big).
\]

\item
\label{it:ranks1}
If neither $\zeta_1$ nor $\zeta_2$ is an uncountable limit cardinal, then $\sigma$ is finitely generated and its rank is as follows:
\begin{enumerate}[label=\textup{(\alph*)}, leftmargin=9mm]
\item
\label{it:ranks1a}
$0$ when $N=\S_1$ and $\zeta_1=\zeta_2=1$ \emph{(}i.e., $\sigma=\Delta_{\M}$\emph{)},
\item
\label{it:ranks1b}
$1$ when one of the following is satisfied:
\begin{enumerate}[label=\textup{(\roman*)}, leftmargin=9mm]
\item
\label{it:ranks1bi}
$N=\S_1$ 
and not both $\zeta_1,\zeta_2$ equal $1$,
\item
\label{it:ranks1bii}
$N=\{\id_n\}$ with $n\geq 2$,
\item
\label{it:ranks1biii}
$N=\S_2$ and $\zeta_1=\zeta_2=1$,
\item
\label{it:ranks1biv}
$n\geq 3$, $N\neq\{\id_n\}$ and $\zeta_1=\zeta_2=\aleph_0$,
\end{enumerate}
\item
\label{it:ranks1c}
$2$ when one of the following is satisfied:
\begin{enumerate}[label=\textup{(\roman*)}, leftmargin=9mm]
\item
\label{it:ranks1ci}
$N=\S_2$
and not both $\zeta_1,\zeta_2$ equal $1$,
\item
\label{it:ranks1cii}
$n\geq 3$, $N\neq\{\id_n\}$ and not both $\zeta_1,\zeta_2$ equal $\aleph_0$.
\end{enumerate}
\end{enumerate}
\end{enumerate}
\end{thm}

\begin{proof}
\pfitem{\ref{it:ranks1}}
We first assume that neither $\ze_1$ nor $\ze_2$ is an uncountable limit cardinal.  
One may confirm by direct inspection that under this assumption, the parameters associated to ${\si=\lam_{\ze_1}^N\cap\rho_{\ze_2}^N}$ satisfy precisely one of the listed groups of constraints in \ref{it:ranks1a}--\ref{it:ranks1c}.
It is therefore sufficient to verify that the congruences in each group have the rank as stated.
In fact, part \ref{it:ranks1a} is clear, as $\si=\Delta_{\M}$ is the least congruence, and hence is generated by $\emptyset$, while part \ref{it:ranks1b} follows from Theorem \ref{thm:pc}, so we just consider the remaining case.

\pfitem{\ref{it:ranks1c} \ref{it:ranks1ci}}  
Theorem \ref{thm:pc} says that $\si=\lam_{\ze_1}^{\S_2}\cap\rho_{\ze_2}^{\S_2}$ is not principal.  However, we have $\sigma=\sigma_1\cup\sigma_2$, where $\sigma_1=\lambda_{\zeta_1}^2\cap \rho_{\zeta_2}^2$ and $\sigma_2=\lambda_1^{\S_2}\cap \rho_1^{\S_2}$ are both principal.  It follows that $\crank(\sigma)=2$.

\pfitem{\ref{it:ranks1c} \ref{it:ranks1cii}}  
The proof is the same as the previous case, but with $\sigma_1=\lambda_{\zeta_1}^n\cap \rho_{\zeta_2}^n$ and~${\sigma_2=\lambda_{\aleph_0}^N\cap \rho_{\aleph_0}^N}$.

\pfitem{\ref{it:ranks2}}
We now consider the case in which at least one of $\ze_1,\ze_2$ is an uncountable limit cardinal; by symmetry we may assume that $\ze_1$ is.  The proof splits into two parts; first showing that the stated value of $\crank(\si)$ is a lower bound, and then an upper bound.

\pfitem{($\geq$)}  
First note that
$\si = \bigcup \bigset{ \lam_\xi^N\cap\rho_{\ze_2}^N}{ \xi\in[\aleph_0,\ze_1)}$.
Since the $\lam_\xi^N\cap\rho_{\ze_2}^N$ form a non-decreasing chain of proper subcongruences of $\si$ (by Theorems \ref{thm-main} and \ref{thm:comparisons}), Lemma \ref{lem:cof} says that $\crank(\si)\geq\cof[\aleph_0,\ze_1)=\cof(\ze_1)$.

If $\ze_2$ is also an uncountable limit cardinal, then the dual of the previous argument gives $\crank(\si)\geq\cof(\ze_2)$; otherwise ${\crank(\si)\geq\cof(\ze_1)\geq\aleph_0\geq\cof(\ze_2)}$.  Thus, in either case, we have ${\crank(\si)\geq{\max}\big({\cof}(\ze_1),\cof(\ze_2)\big)}$.

\pfitem{($\leq$)}  
Let $\Xi_1$ be a cofinal subset of $[\aleph_0,\ze_1)$ consisting entirely of successor cardinals and having size~$\cof(\ze_1)$.  
Define $\Xi_2\sub[\aleph_0,\ze_2)$ analogously if $\ze_2$ is also an uncountable limit cardinal;
otherwise let $\Xi_2=\{\ze_2\}$.  Then $\si = \bigcup\bigset{\lam_{\ka_1}^N\cap\rho_{\ka_2}^N}{\ka_i\in\Xi_i}$, with each $\lam_{\ka_1}^N\cap\rho_{\ka_2}^N$ of rank at most $2$ by part~\ref{it:ranks1}.  It follows that
$\crank(\si) \leq 2\cdot|\Xi_1|\cdot|\Xi_2| \leq 2\cdot\cof(\ze_1)\cdot\cof(\ze_2) = {\max}\big({\cof}(\ze_1),\cof(\ze_2)\big)$.
\end{proof}

We now work towards the corresponding result for type \ref{CT2} congruences.  As the statement is even more involved than for the \ref{CT1} congruences, it will be convenient to first identify the cases that need to be considered, and we do this in the next lemma.  For the proof, and for later use, note that if $\si$ is a congruence of type \ref{CT2} with $\eta(\si)=\aleph_0$, then from $\xi_k<\dots<\xi_1$ and $\xi_1,\dots,\xi_k\in\{1\}\cup [\aleph_0,\eta]$, it follows that $k\leq 2$, and that $\si$ has one of the forms
\begin{equation}\label{eq:aleph_0}
\si = (\lam_{\ze_1}^{\aleph_0}\cap\rho_{\ze_2}^{\aleph_0})\cup\mu_{\aleph_0}^{\eta_1}\cup\mu_1^{|X|^+} 
\quad \text{or}\quad 
\si = (\lam_{\ze_1}^{\aleph_0}\cap\rho_{\ze_2}^{\aleph_0})\cup\mu_{\aleph_0}^{|X|^+}
\quad \text{or}\quad 
\si = (\lam_{\ze_1}^{\aleph_0}\cap\rho_{\ze_2}^{\aleph_0})\cup\mu_1^{|X|^+}.
\end{equation}

\begin{lemma}\label{lem:CT2ranks}
The parameters associated to a congruence $\si=(\lambda_{\zeta_1}^\eta\cap\rho_{\zeta_2}^\eta)\cup\mu_{\xi_1}^{\eta_1}\cup\dots\cup\mu_{\xi_k}^{\eta_k}$ of type~\ref{CT2} satisfy precisely one of the following three conditions:
\begin{enumerate}[label=\textup{(\Roman*)},leftmargin=7mm]

\item
\label{it:I}
at least one of the following two conditions holds:
\begin{enumerate}[label=\textup{(I.\arabic*)},leftmargin=11mm]
\item
\label{it:Ii}
at least one of $\eta,\zeta_1,\zeta_2,\xi_1,\dots,\xi_k,\eta_1,\dots,\eta_k$
is an uncountable limit cardinal, or
\item
\label{it:Iii}
$k=1$, $\eta=\aleph_0$ and $\xi_1=1$ \emph{(}i.e., $\sigma=\lambda_{\zeta_1}^{\aleph_0}\cap \rho_{\zeta_2}^{\aleph_0}$\emph{)}, 
\een

\item
\label{it:II}
$\eta=\xi_1=\aleph_0$ and none of $\ze_1,\ze_2,\eta_1$ is an uncountable limit cardinal, or

\item
\label{it:III}
none of $\zeta_1,\zeta_2,\xi_1,\dots,\xi_k,\eta_1,\dots,\eta_k$ is an uncountable limit cardinal, and $\eta$ is not a limit cardinal.  
\end{enumerate}
\end{lemma}

\pf
First we show that at least one of \ref{it:I}--\ref{it:III} holds.  To do so, suppose \ref{it:Ii}, \ref{it:II} and~\ref{it:III} do not hold.  Combining the negations of \ref{it:Ii} and~\ref{it:III}, we have $\eta=\aleph_0$ and so $\si$ has one of the forms in~\eqref{eq:aleph_0}.  Combining this with the negations of \ref{it:Ii} and \ref{it:II}, it follows that $\xi_1=1$ and so~$\si$ is of the third form listed in \eqref{eq:aleph_0}, meaning that \ref{it:Iii} holds.

The following pairs are clearly mutually exclusive: \ref{it:Ii} and~\ref{it:III}; \ref{it:Iii} and \ref{it:II}; \ref{it:Iii} and~\ref{it:III}; \ref{it:II} and \ref{it:III}.  That \ref{it:Ii} and \ref{it:II} are also mutually exclusive follows from \eqref{eq:aleph_0}.
\epf

Note in passing that while conditions \ref{it:I}, \ref{it:II} and \ref{it:III} above are mutually exclusive, the two sub-condi\-tions~\ref{it:Ii} and \ref{it:Iii} are not.
Here now is the result giving ranks of type \ref{CT2} congruences, with the subdivisions taken from the conditions in Lemma \ref{lem:CT2ranks}.

\begin{thm}
\label{thm:ranksCT2}
Let ${\sigma=(\lambda_{\zeta_1}^\eta\cap\rho_{\zeta_2}^\eta)\cup\mu_{\xi_1}^{\eta_1}\cup\dots\cup\mu_{\xi_k}^{\eta_k}}$ be a congruence on~$\M$ of type \ref{CT2}, where $\M$ denotes either $\P_X$ or $\PB_X$ for some infinite set $X$.
\begin{enumerate}[label=\textup{(\Roman*)},leftmargin=7mm]

\item
\label{it:CT2ranksI}
If at least one of the following two conditions holds:
\begin{enumerate}[label=\textup{(I.\arabic*)},leftmargin=11mm]
\item
\label{it:CT2ranksIi}
at least one of $\eta,\zeta_1,\zeta_2,\xi_1,\dots,\xi_k,\eta_1,\dots,\eta_k$
is an uncountable limit cardinal, or
\item
\label{it:CT2ranksIii}
$k=1$, $\eta=\aleph_0$ and $\xi_1=1$ \emph{(}i.e., $\sigma=\lambda_{\zeta_1}^{\aleph_0}\cap \rho_{\zeta_2}^{\aleph_0}$\emph{)},
\een
then $\si$ is not finitely generated and
\[
\crank(\sigma)={\max}\big({\cof(\eta)},\cof(\zeta_1),\cof(\zeta_2),\cof(\xi_1),\dots,\cof(\xi_k),\cof(\eta_1),\dots,\cof(\eta_k)\big).
\]

\item
\label{it:CT2ranksII}
If $\eta=\xi_1=\aleph_0$, and if none of $\ze_1,\ze_2,\eta_1$ is an uncountable limit cardinal, then $\si$ is finitely generated and
\[
\crank(\si) = \begin{cases}
1 &\text{if $\ze_1=\ze_2=\aleph_0$}\\
2 &\text{otherwise.}
\end{cases}
\]

\item
\label{it:CT2ranksIII}
If none of $\zeta_1,\zeta_2,\xi_1,\dots,\xi_k,\eta_1,\dots,\eta_k$ is an uncountable limit cardinal, and if $\eta$ is not a limit cardinal, then $\sigma$ is finitely generated and its rank is as follows:
\begin{enumerate}[label=\textup{(\alph*)}, leftmargin=9mm]
\item
\label{it:CT2ranksIIIa}
$k-1$ when $k\geq 2$, $\xi_k=1$ and $\xi_1=\zeta_1=\zeta_2=\eta$,
\item
\label{it:CT2ranksIIIb}
$k$ when one of the following holds:
\begin{enumerate}[label=\textup{(\roman*)}, leftmargin=9mm]
\item
\label{it:CT2ranksIIIbi}
$k=1$ and $\xi_1=1$ \emph{(}including the case in which $\si=\nabla_{\M}$\emph{)},
\item
\label{it:CT2ranksIIIbii}
$\xi_k\neq 1$ and $\xi_1=\zeta_1=\zeta_2=\eta$,
\item
\label{it:CT2ranksIIIbiii}
$k\geq2$, $\xi_k=1$ and at least one of $\xi_1,\zeta_1,\zeta_2$ does not equal $\eta$,
\end{enumerate}
\item
\label{it:CT2ranksIIIc}
$k+1$ when $\xi_k\neq 1$ and at least one of $\xi_1,\zeta_1,\zeta_2$ does not equal $\eta$.
\end{enumerate}
\end{enumerate}
\end{thm}

\pf
Throughout the proof, we write
\[
\si_0 = \lambda_{\zeta_1}^\eta\cap\rho_{\zeta_2}^\eta \AND \si_i = \mu_{\xi_i}^{\eta_i} \quad\text{for each $1\leq i\leq k$.}
\]
Note that $\si_0,\si_1,\ldots,\si_k$ are all congruences, but they need not be relatively incomparable in general; for example, if $\xi_k=1$ then $\si_k=\De_{\M}$ is contained in each of the other $\si_i$.  We also write~$\ch\si_i=\bigcup_{j\not=i}\si_j$ for each $0\leq i\leq k$.  These are also congruences; specifically, we have
\begin{equation}\label{eq:ch_si}
\ch\si_i = \begin{cases}
\De_{\M}  &\quad\text{if $i=0$ and $\xi_1=1$} \\[2truemm]
(\lam_{\xi_1}^{\xi_1}\cap\rho_{\xi_1}^{\xi_1}) \cup\mu_{\xi_1}^{\eta_1}\cup\cdots\cup\mu_{\xi_k}^{\eta_k}  &\quad\text{if $i=0$ and $\xi_1\not=1$} \\[2truemm]
(\lam_{\ze_1}^\eta\cap\rho_{\ze_2}^\eta)  \cup\mu_{\xi_1}^{\eta_1}\cup\cdots\cup \mu_{\xi_{i-1}}^{\eta_{i-1}}\cup\mu_{\xi_{i+1}}^{\eta_{i+1}}\cup\cdots\cup\mu_{\xi_k}^{\eta_k}  &\quad\text{if $1\leq i\leq k-1$,} \\[2truemm]
(\lam_{\ze_1}^\eta\cap\rho_{\ze_2}^\eta)  \cup\mu_{\xi_1}^{\eta_1}\cup\cdots\cup \mu_{\xi_{k-1}}^{\eta_{k-1}}\cup\mu_1^{\eta_k} &\quad\text{if $i=k$.}
\end{cases}
\end{equation}
Note that we used Lemma \ref{lem:mu} in the case of $i=0$ and $\xi_1\not=1$.
We begin with the second part.

\pfitem{\ref{it:CT2ranksII}}
Suppose $\eta=\xi_1=\aleph_0$ and none of $\ze_1,\ze_2,\eta_1$ is an uncountable limit cardinal.
Here $\si$ has one of the first two forms in \eqref{eq:aleph_0}.  In fact, since $\mu_1=\De_{\M}$, these may both be simplified to $\si = (\lam_{\ze_1}^{\aleph_0}\cap\rho_{\ze_2}^{\aleph_0})\cup\mu_{\aleph_0}^{\eta_1}$.
If $\ze_1=\ze_2=\aleph_0$, then $\si$ is principal by Theorem \ref{thm:pc} \ref{it:pc6}.  
Suppose now that at least one of $\ze_1,\ze_2$ is uncountable.  By again consulting Theorem \ref{thm:pc} we see that $\si$ is not principal.  However, by Theorem \ref{thm:mj} \ref{it:mj2} we have $\si=\tau_1\vee\tau_2$ where $\tau_1 = (\lam_{\aleph_0}^{\aleph_0}\cap\rho_{\aleph_0}^{\aleph_0})\cup\mu_{\aleph_0}^{\eta_1}$ and $\tau_2=\lam_{\ze_1}^1\cap\rho_{\ze_2}^1$.  Since $\tau_1$ and $\tau_2$ are both principal by Theorem \ref{thm:pc}, it follows that $\crank(\si)=2$.

\pfitem{\ref{it:CT2ranksIII}}
Suppose next that none of $\zeta_1,\zeta_2,\xi_1,\dots,\xi_k,\eta_1,\dots,\eta_k$ is an uncountable limit cardinal, and that $\eta$ is not a limit cardinal.
First, one may check that the parameters associated to $\si$ satisfy exactly one of the stated sets of constraints.
The proof splits into two parts; first showing that the stated value of $\crank(\si)$ is a lower bound, and then an upper bound.

\pfitem{($\geq$)}  For any $0\leq i\leq k$, we have $\si=\si_i\cup\ch\si_i$, with $\si_i$ and $\ch\si_i$ both congruences.  In particular, if~$\ch\si_i$ is properly contained in $\si$ for some $i$, then any generating set for $\si$ must contain at least one element of $\si_i\sm\ch\si_i$. 
Noting that  $\si_i\sm\ch\si_i$ and $\si_j\sm\ch\si_j$
are disjoint when $i\neq j$, 
 to show that the stated value of $\crank(\si)$ is a lower bound, it suffices to show that the set of all such $i$ has size at least this stated value.  From \eqref{eq:ch_si}, one may easily check that:
\bit
\item in case \ref{it:CT2ranksIIIa}, $\ch\si_1,\ldots,\ch\si_{k-1}$ are all properly contained in $\si$,
\item in case \ref{it:CT2ranksIIIb} \ref{it:CT2ranksIIIbi}, $\ch\si_0$ is properly contained in $\si$,
\item in case \ref{it:CT2ranksIIIb} \ref{it:CT2ranksIIIbii}, $\ch\si_1,\ldots,\ch\si_k$ are all properly contained in $\si$,
\item in case \ref{it:CT2ranksIIIb} \ref{it:CT2ranksIIIbiii}, $\ch\si_0,\ldots,\ch\si_{k-1}$ are all properly contained in $\si$,
\item in case \ref{it:CT2ranksIIIc}, $\ch\si_0,\ldots,\ch\si_k$ are all properly contained in $\si$.
\eit
In each case, this leads to the desired lower bound.

\pfitem{$(\leq$)}  Since $\si=\si_0\cup\si_1\cup\cdots\cup\si_k$, and since each $\si_i$ is principal by Theorem \ref{thm:pc}, it follows that ${\crank(\si)\leq k+1}$.  This deals with case \ref{it:CT2ranksIIIc}.  If $\xi_k=1$, then $\si_k=\De_{\M}$, and so ${\si=\si_0\cup\si_1\cup\cdots\cup\si_{k-1}}$, giving $\crank(\si)\leq k$ in this case; this deals with \ref{it:CT2ranksIIIb} \ref{it:CT2ranksIIIbi} and \ref{it:CT2ranksIIIb} \ref{it:CT2ranksIIIbiii}.  If $\xi_1=\ze_1=\ze_2=\eta$, then $\si_0=\lam_{\xi_1}^{\xi_1}\cap\rho_{\xi_1}^{\xi_1}\sub\mu_{\xi_1}^{\xi_1}\sub\mu_{\xi_1}^{\eta_1}=\si_1$, using Lemma \ref{la32} \ref{it:32ix} for the first inclusion, so that $\si=\si_1\cup\cdots\cup\si_k$, giving $\crank(\si)\leq k$ in this case; this deals with~\ref{it:CT2ranksIIIb}~\ref{it:CT2ranksIIIbii}.  Combining the previous two sentences, we have $\si=\si_1\cup\cdots\cup\si_{k-1}$ in case~\ref{it:CT2ranksIIIa}, giving $\crank(\si)\leq k-1$ in this case.

\pfitem{\ref{it:CT2ranksI}}
Suppose first that \ref{it:CT2ranksIii} holds but not \ref{it:CT2ranksIi}, so that $\si=\lam_{\ze_1}^{\aleph_0}\cap\rho_{\ze_2}^{\aleph_0}$ with neither $\ze_1$ nor $\ze_2$ an uncountable limit cardinal.  Here we must show that $\crank(\si)=\aleph_0$.  Now,
\[
\si = \bigcup_{n\in[3,\aleph_0)}(\lambda_{\zeta_1}^n\cap \rho_{\zeta_2}^n),
\]
so Lemma \ref{lem:cof} gives $\crank(\si)\geq\cof[3,\aleph_0)=\aleph_0$.  On the other hand, Theorem~\ref{thm:pc}~\ref{it:pc4} says that each $\lambda_{\zeta_1}^n\cap \rho_{\zeta_2}^n$ is principal, and it follows that $\crank(\si)\leq\aleph_0$.

For the rest of the proof we assume that condition \ref{it:CT2ranksIi} holds.  As in Part \ref{it:CT2ranksIII}, we separately establish that the stated value is both a lower and an upper bound for $\crank(\sigma)$.

\pfitem{($\geq$)}
We must show that $\crank(\si)\geq\cof(\ka)$ for each $\kappa\in\{\eta,\zeta_1,\zeta_2,\xi_1,\dots,\xi_k,\eta_1,\dots,\eta_k\}$.  In fact it suffices to do so for every such $\ka$ that happens to be an uncountable limit cardinal; 
indeed, at least one such $\ka$ exists by assumption, and for any $\ka'\in\{\eta,\zeta_1,\zeta_2,\xi_1,\dots,\xi_k,\eta_1,\dots,\eta_k\}$
that is not an uncountable limit cardinal  we then have $\crank(\sigma)\geq\cof(\kappa)\geq\aleph_0\geq\cof(\kappa')$.

\pfcase{1}  First suppose some $\xi_i$ is an uncountable limit cardinal.  For the possibility that $i=k$, it will be convenient (only here) to define $\xi_{k+1}=\aleph_0$.  Since $\xi_{i+1}<\xi_i$, and since $\xi_i$ is a limit cardinal, the interval $[\xi_{i+1}^+,\xi_i)$ is non-empty.  For each $\ka\in[\xi_{i+1}^+,\xi_i)$, let 
\[
\si^{(\ka)} = \ch\si_i \cup \mu_{\ka}^{\eta_i} 
= (\lam_{\ze_1}^\eta\cap\rho_{\ze_2}^\eta)\cup\mu_{\xi_1}^{\eta_1}\cup\cdots\cup\mu_{\xi_{i-1}}^{\eta_{i-1}} \cup \mu_\ka^{\eta_i} \cup \mu_{\xi_{i+1}}^{\eta_{i+1}}\cup\cdots \cup\mu_{\xi_k}^{\eta_k}.
\]
All $\sigma^{(\ka)}$ are congruences by Theorem \ref{thm-main}, and they form a non-descending chain by Theorem~\ref{thm:comparisons}.  Since clearly $\si=\bigcup_{\ka\in[\xi_{i+1}^+,\xi_i)}\si^{(\ka)}$, Lemma \ref{lem:cof} then gives $\crank(\si)\geq\cof[\xi_{i+1}^+,\xi_i)=\cof(\xi_i)$.

\pfcase{2}  If some $\eta_i$ is an uncountable limit cardinal, then we use the same argument as the previous case, but with $\si^{(\ka)} = \ch\si_i \cup \mu_{\xi_i}^\ka$ for each $\ka\in[\eta_{i-1}^+,\eta_i)$, keeping in mind $\eta_0=\eta$.

\pfcase{3}  Next suppose $\eta$ is an uncountable limit cardinal.  If $\eta=\xi_1$, then $\crank(\si)\geq\cof(\eta)$ by Case 1, so suppose $\eta>\xi_1$.  Then we take $\si^{(\ka)} = \ch\si_0\cup(\lam_{\ze_1}^\ka\cap\rho_{\ze_2}^\ka)$ for each $\ka\in[\xi_1',\eta)$, where $\xi_1'=\max(\xi_1,\aleph_0)$.

\pfcase{4}  Finally suppose $\ze_1$ is an uncountable limit cardinal (the case of $\ze_2$ is dual).  If $\ze_1=\eta$, then $\crank(\si)\geq\cof(\ze_1)$ by Case 3, so suppose $\ze_1>\eta$.  Then we take $\si^{(\ka)} = \ch\si_0\cup(\lam_\ka^\eta\cap\rho_{\ze_2}^\eta)$ for each $\ka\in[\eta,\ze_1)$.

\newpage

\pfitem{($\leq$)}
Since at least one of $\eta,\zeta_1,\zeta_2,\xi_1,\dots,\xi_k,\eta_1,\dots,\eta_k$ is a limit cardinal, it follows that the maximum of their cofinalities is infinite and is equal to the sum of all their cofinalities.  Hence, since $\crank(\si)\leq\sum_{i=0}^k\crank(\si_i)$, it is sufficient to show that
\begin{align}
\label{eq:cof1}
\crank(\si_0) &\leq
\cof(\eta)+\cof(\zeta_1)+\cof(\zeta_2),
\\ 
\label{eq:cof2}
\crank(\si_i) &\leq
\cof(\xi_i)+\cof(\eta_i) \qquad\qquad\qquad \text{for } i=1,\dots,k.
\end{align}

We begin with \eqref{eq:cof1}.  If each $\ka\in\{\eta,\ze_1,\ze_2\}$ is a successor cardinal, then $\crank(\si_0)=1$ by Theorem \ref{thm:pc} \ref{it:pc5}, and \eqref{eq:cof1} follows trivially; so we assume that at least one such $\ka$ is a limit cardinal, noting that $\cof(\ka)\geq\aleph_0$ for this $\ka$.  We now define three sets $\Xi_\eta$, $\Xi_{\zeta_1}$ and $\Xi_{\zeta_2}$ as follows.  If $\eta=\aleph_0$ we set $\Xi_\eta=[0,\aleph_0)$, if~$\eta$ is a successor cardinal we set $\Xi_\eta=\{\eta\}$, and otherwise we let~$\Xi_\eta$ be a cofinal subset of $[\aleph_0,\eta)$ of cardinality $\cof(\eta)$ consisting of successor cardinals.  Similarly, for~$i=1,2$, we let $\Xi_{\zeta_i}$ be a cofinal subset of $[\aleph_0,\zeta_i)$ of cardinality $\cof(\zeta_i)$ consisting of successor cardinals if $\zeta_i$ is an uncountable limit cardinal, and $\Xi_{\zeta_i}=\{\zeta_i\}$ otherwise (including the case of~$\ze_i=\aleph_0$).  Note that $|\Xi_\kappa|\leq\cof(\kappa)$ for each $\kappa\in\{\eta,\zeta_1,\zeta_2\}$.  Now,
\[
\si_0 = \lambda_{\zeta_1}^\eta\cap\rho_{\zeta_2}^\eta =\bigcup \bigset{\lambda_{\zeta_1'}^{\eta'}\cap\rho_{\zeta_2'}^{\eta'}}{\eta'\in\Xi_\eta,\ \zeta_1'\in\Xi_{\zeta_1},\ \zeta_2'\in\Xi_{\zeta_2},\ \zeta_1',\zeta_2'\geq\eta'},
\]
with each congruence $\lambda_{\zeta_1'}^{\eta'}\cap\rho_{\zeta_2'}^{\eta'}$ of rank at most~$2$ by Theorem \ref{thm:pc} \ref{it:pc5} or Theorem \ref{thm:ranksCT1} \ref{it:ranks1}.  Thus,
\begin{align*}
\crank(\si_0) &\leq \sum \bigset{{\crank}(\lambda_{\zeta_1'}^{\eta'}\cap\rho_{\zeta_2'}^{\eta'})}{\eta'\in\Xi_\eta,\ \zeta_1'\in\Xi_{\zeta_1},\ \zeta_2'\in\Xi_{\zeta_2},\ \zeta_1',\zeta_2'\geq\eta'} \\
&\leq 2\cdot|\Xi_\eta|\cdot |\Xi_{\zeta_1}|\cdot |\Xi_{\zeta_2}| \leq 2\cdot\cof(\eta)\cdot\cof(\zeta_1)\cdot\cof(\zeta_2) =  \cof(\eta)+\cof(\zeta_1)+\cof(\zeta_2),
\end{align*}
as required.

The proof of \eqref{eq:cof2} is similar.  If $\xi_i=1$, then $\si_i=\De_{\M}$ and \eqref{eq:cof2} holds trivially, so we assume that $\xi_i$ is infinite, noting that $\eta_i>\eta\geq\xi_i\geq\aleph_0$.  Let $\Xi_{\xi_i}=\{\xi_i\}$ if $\xi_i$ is not an uncountable limit cardinal; otherwise, let $\Xi_{\xi_i}$ be a cofinal subset of $[\aleph_0,\xi_i)$ consisting of successor cardinals and having size $\cof(\xi_i)$.  Let $\Xi_{\eta_i}=\{\eta_i\}$ if $\eta_i$ is a successor cardinal; otherwise $\eta_i$ is an uncountable limit cardinal, and we let $\Xi_{\eta_i}$ be a cofinal subset of $[\xi_i^+,\eta_i)$ consisting of successor cardinals and having size $\cof(\eta_i)$.  Note that every element of $\Xi_{\xi_i}$ is less than every element of $\Xi_{\eta_i}$.   Now,
\[
\si_i = \mu_{\xi_i}^{\eta_i} = \bigcup\bigset{\mu_{\xi'}^{\eta'}}{\xi'\in\Xi_{\xi_i},\ \eta'\in\Xi_{\eta_i}} ,
\]
with each congruence $\mu_{\xi'}^{\eta'}$ principal by Theorem \ref{thm:pc} \ref{it:pc6}.  Thus,
\[
\crank(\si_i) \leq |\Xi_{\xi_i}|\cdot|\Xi_{\eta_i}| \leq \cof(\xi_i)\cdot\cof(\eta_i) \leq \cof(\xi_i)+\cof(\eta_i).
\]
This completes the proof of \eqref{eq:cof2}, and indeed of the entire theorem.
\epf

\section{Other diagram monoids and transformation monoids}
\label{sect:OM}

In this section we present the historically earlier results classifying the congruences on finite diagram monoids (Subsection \ref{subsect:FDM}) and on finite and infinite transformation monoids (Subsection~\ref{subsect:FTM}), within the conceptual and notational framework developed in this paper, and we compare the respective congruence lattices. We conclude by discussing some possible directions for further research (Subsections \ref{subsect:OTM} and \ref{subsect:ideals}).

\subsection{Finite diagram monoids}\label{subsect:FDM}

Congruences on finite partition and partial Brauer monoids were classified in \cite{EMRT2018}, alongside several other finite diagram monoids including Brauer and Temperley-Lieb monoids.
In the following theorem we provide a translation of \cite[Theorems~5.4 and~6.1]{EMRT2018} using the terminology of this article.
The salient points are that only congruences of type \ref{CT1} are present 
(plus the universal congruence of course), and that the dimensions of layers are cut down to just $2\times 2$ and singletons.
When~${X=\{1,\ldots,n\}}$ we write $\P_n$ for $\P_X$, and so on, and we note that $|\ol\al\sd\ol\be|,|\ul\al\sd\ul\be|<2n$ for all $\al,\be\in\P_n$.

\begin{thm}
\label{thm:PnBn}
Let $n\in\N$, $n\geq 2$, and let $\Mn$ stand for $\P_n$ or $\PB_n$.  The distinct congruences on $\Mn$ are precisely the universal congruence $\nabla_{\Mn}$, and the congruences $\lambda_{\zeta_1}^N\cap\rho_{\zeta_2}^N$ where
\bit
\item
$N\unlhd\S_q$ for some $1\leq q\leq n$,
\item
$\zeta_1,\zeta_2\in \{1,2n\}$ if $q\leq 2$,
\item
$\zeta_1=\zeta_2=2n$ if $q\geq 3$.
\eit
The lattice $\Cong(\Mn)$ is shown in Figure \ref{fig:CongMn}.  The $*$-congruences are those with $\ze_1=\ze_2$, and these are represented by white vertices in the figure.
\end{thm}

\begin{figure}[ht]
\begin{center}

\begin{tikzpicture}[scale=0.9,x={(1.2mm,0.2mm)},y={(-0.9mm,0.3mm)},z={(0mm,1.2mm)},inner sep=0.7mm]

\foreach \z in {0,1}
  {\fill [black!15,draw=black,fill opacity=1] (0,0,8*\z)--(16,0,8*\z)--(16,16,8*\z)--(0,16,8*\z)--(0,0,8*\z);
  \fill (0,0,8*\z) circle [radius=0.8mm];
  \fill (16,0,8*\z) circle [radius=0.8mm];
  \fill (16,16,8*\z) circle [radius=0.8mm];
  \fill (0,16,8*\z) circle [radius=0.8mm];
  \foreach \x in {0,1}
  \foreach \y in {0,1}
     \draw (16*\x,16*\y, 8*\z)--(16*\x,16*\y, 8*\z +8); }

\foreach \z in {2}
  {\fill [black!15,draw=black,fill opacity=1] (0,0,8*\z)--(16,0,8*\z)--(16,16,8*\z)--(0,16,8*\z)--(0,0,8*\z);
  \fill (0,0,8*\z) circle [radius=0.8mm];
  \fill (16,0,8*\z) circle [radius=0.8mm];
  \fill (16,16,8*\z) circle [radius=0.8mm];
  \fill (0,16,8*\z) circle [radius=0.8mm];
}

\draw (16,16,16) -- (16,16,43);
\draw (16,16,50+5) -- (16,16,61+5);

\foreach \x in {0,...,10} {\fill (16,16,44+\x) circle [radius=0.2mm];}

\foreach \z in {0,1,2}
{\fill (16,16,24+8*\z) circle [radius=0.8mm];}

\fill (16,16,53+5) circle [radius=0.8mm];
\fill (16,16,61+5) circle [radius=0.8mm];

\draw [->] (0,-7,0)-- (22,-7,0);
\node at (24,-8,0) {$\zeta_1$};
\draw [dotted] (0,-7,0)--(0,0,0);
\node at (-2,-13,0) {$1$};
\draw [dotted] (16,-7,0)--(16,0,0);
\node at (14,-13,0) {$2n$};

\draw [->] (-7,0,0)-- (-7,24,0);
\node at (-8,25,0) {$\zeta_2$};
\draw [dotted] (-7,0,0)--(0,0,0);
\node at (-12,-5,0) {$1$};
\draw [dotted] (-7,16,0)--(0,16,0);
\node at (-13,9,0) {$2n$};

\draw [->] (0,34,0)--(0,34,63+5);
\node at (0,37,61+8) {$N$};
\foreach \z in {0,...,2} 
  \draw [dotted] (0,34,8*\z)--(0,16,8*\z);
\foreach \z in {3,4,5} 
  \draw [dotted] (0,34,8*\z)--(16,16,8*\z);

\draw [dotted] (0,34,53+5)--(16,16,53+5);

\node at (0,41,0) {$\S_1$};
\node at (0,41,8) {$\{\id_2\}$};
\node at (0,41,16) {$\S_2$};
\node at (0,41,24) {$\{\id_3\}$};
\node at (0,41,32) {$\mathcal{A}_3$};
\node at (0,41,40) {$\S_3$};
\node at (0,41,53+5) {$\S_n$};
\node[right] at (16,16,61+5) {$\nabla_{\Mn}$};

\foreach \x in {0,8,16} {
\fill (0,0,\x) circle [radius=1.1mm];
\fill[white] (0,0,\x) circle [radius=0.8mm];
}

\foreach \x in {0,8,16,24,32,40,58,66} {
\fill (16,16,\x) circle [radius=1.1mm];
\fill[white] (16,16,\x) circle [radius=0.8mm];
}

\end{tikzpicture}

\end{center}
\caption{Hasse diagram of $\Cong(\P_n)$ and $\Cong(\PB_n)$ for $n\in\N$, $n\geq 2$.  See Theorem \ref{thm:PnBn} for more details; cf.~Figure~\ref{fig:types12} and~\cite[Figure 5]{EMRT2018}.}
\label{fig:CongMn}
\end{figure}
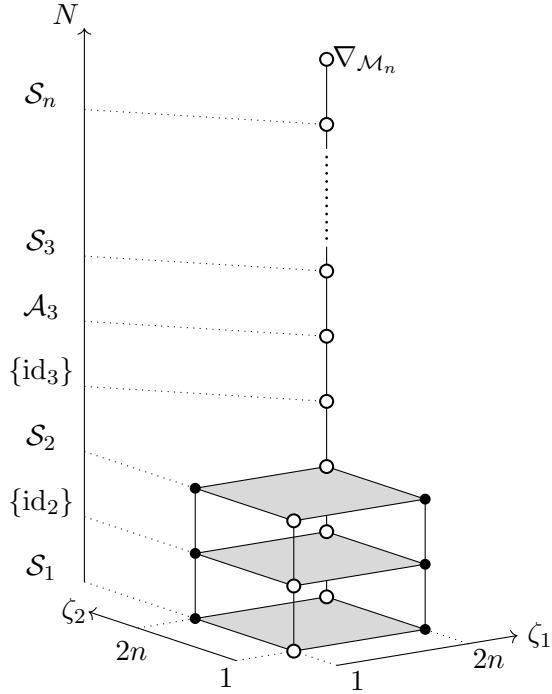

As a further aid, we provide a translation between Theorem \ref{thm:PnBn} and \cite[Theorems~5.4 and~6.1]{EMRT2018} in Table \ref{tab:dict}.

\begin{table}

\begin{center}
\begin{tabular}{|c|c|}
\hline
\textbf{Theorem \ref{thm:PnBn}} & \textbf{\cite[Theorems 5.4 and 6.1]{EMRT2018}}
\\ \hline
$\lambda_{1}^{\S_1}\cap\rho_1^{\S_1}=\Delta_{\Mn}$ &
$\mu_0=\Delta_{\Mn}$
\\
$\lambda_{2n}^{\S_1}\cap\rho_1^{\S_1}$ &
$\lambda_0$
\\
$\lambda_{1}^{\S_1}\cap\rho_{2n}^{\S_1}$ &
$\rho_0$
\\
$\lambda_{1}^{\{\id_2\}}\cap\rho_1^{\{\id_2\}}$ &
$\mu_1$
\\
$\lambda_{2n}^{\{\id_2\}}\cap\rho_1^{\{\id_2\}}$ &
$\lambda_1$
\\
$\lambda_{1}^{\{\id_2\}}\cap\rho_{2n}^{\{\id_2\}}$ &
$\rho_1$
\\
$\lambda_{1}^{\S_2}\cap\rho_1^{\S_2}$ &
$\mu_{\S_2}$
\\
$\lambda_{2n}^{\S_2}\cap\rho_1^{\S_2}$ &
$\lambda_{\S_2}$
\\
$\lambda_{1}^{\S_2}\cap\rho_{2n}^{\S_2}$ &
$\rho_{\S_2}$
\\
$\lambda_{2n}^N\cap\rho_{2n}^N$ &
$R_N$, $\{\id_q\}\not=N\normal\S_q$, $2\leq q\leq n$
\\
$\lambda_{2n}^{\{\id_{q}\}}\cap\rho_{2n}^{\{\id_{q}\}}$ &
$R_{q-1}$, $1\leq q\leq n$
\\
\hline
\end{tabular}

\caption{The correspondence between (non-universal) congruences on $\P_n$ and $\PB_n$ listed in Theorem \ref{thm:PnBn}
and those from \cite[Theorems 5.4 and 6.1]{EMRT2018}.}
\label{tab:dict}
\end{center}
\end{table}

\subsection{Full transformation monoids}\label{subsect:FTM}

The \emph{full transformation monoid} $\T_X$ is the monoid of all transformations of the set $X$ (i.e., all functions $X\to X$) under composition.
Congruences on full transformation monoids were classified by Mal'cev \cite{Malcev1952}.
Clifford and Preston present a very nice account of Mal'cev's results in \cite[Section 10.8]{CPbook2}, and it
has in fact to a great extent served as a motivation and a guide for our work.
An even more modern account, but restricted to the finite case, can be found in
\cite[Section~6.3]{GMbook}.
In what follows we explain how Clifford and Preston's rendering of Mal'cev's results for both the finite and infinite cases can be couched in our terminology.

As in \cite[Section 2]{EF2012},
for us a \emph{transformation} $\alpha$ on a set $X$ will be a special kind of partition, namely one in which every block has the form $A\cup\{b'\}$,
where $A\subseteq X$ and $b\in X$, including the possibility that $A=\emptyset$.
Equivalently, $\al\in\P_X$ is a transformation if and only if $\dom(\al)=X$ and $\coker(\al)=\De_X$.
The set of all such transformations is a submonoid of $\P_X$ isomorphic to the full transformation monoid $\T_X$; thus, from now on, we will identify $\T_X$ with this submonoid.
When~$X=\{1,\dots,n\}$ we write $\T_n$ for $\T_X$.

In what follows we will use all the notation developed for $\P_X$ as restricted to $\T_X$, in particular the ${\J}={\D}$-classes $D_\xi$ for $1\leq\xi\leq|X|$, ideals $I_\xi$ for $1\leq\xi\leq|X|^+$ (noting that $I_1=\emptyset$ since the minimal rank of a transformation is $1$), and relations such as
\bit
\item $R_\xi=\De_{\T_X}\cup(I_\xi\times I_\xi)$ for $1\leq\xi\leq|X|^+$, including $R_1=\De_{\T_X}$ and $R_{|X|^+}=\nabla_{\T_X}$,
\item $R_N=R_q\cup\nu_N$ for finite $1\leq q\leq|X|$ and $N\normal\S_q$, including $R_{\S_1}=\De_{\T_X}$. 
\eit
Note that $\T_X$ is not closed under the involution of $\P_X$; in fact, $\T_X$ has no involution at all, as evidenced by the fact that the bottom $\D$-class $D_1$ has more $\L$-classes than $\R$-classes, so there are no $*$-congruences to speak of.
The congruences on $\T_n$ for $n\in\N$ are listed in
\cite[Theorem~10.68]{CPbook2}, which can readily be translated into our notation as follows:

\begin{thm}\label{thm:CongTn}
The distinct congruences of the full transformation monoid $\T_n$ for $n\in \N$, $n\geq 2$, are precisely
the universal congruence $\nabla_{\T_n}$, and the congruences
$R_N$ for $N\unlhd \S_q$ $(1\leq q\leq n)$. The lattice $\Cong(\T_n)$ is a chain.
\end{thm}

It is worthy of note that the parameters $\zeta_1$ and $\zeta_2$ play no role in the description of congruences on finite $\T_X$, and this will also be the case for infinite $X$.  This is in fact not surprising at all for $\zeta_2$, because $|\underline{\alpha}\sd\underline{\beta}|=0$ for all $\alpha,\beta\in\T_X$.  
The irrelevance of $\zeta_1$ is only a little less obvious: indeed, since transformations have no upper non-transversals, $|\overline{\alpha}\sd\overline{\beta}|<2\eta$ for all $\alpha,\beta\in I_\eta$, and hence $\lambda_\zeta^\eta=R_\eta$ for infinite $\zeta\geq\eta$.
Together, these two observations imply that $\lambda_{\ze_1}^\eta\cap\rho_{\ze_2}^\eta=R_\eta$ for infinite $\ze_1,\ze_2\geq\eta$.  (We also have $\ol\al=\{X\}$ for all $\al\in D_1=I_2$, so even the fact that $\ze_1,\ze_2$ are allowed to be $1$ when $\eta\leq2$ plays no role in the lower part of $\Cong(\T_X)$.)

When $X$ is infinite, $\T_X$ still has congruences of the form $R_N$ (for $N\normal\S_n$ with $1\leq n<\aleph_0$), but there are further congruences involving a parameter Clifford and Preston call the \emph{difference rank}.  
For $\al\in\T_X$ and $x\in X$ we write $x\alpha=y$, where $y'$ is the unique element  of $X'$ that belongs to the block of $\al$ containing $x$; for a subset $Z\subseteq X$ we write $Z\alpha=\set{ z\alpha}{ z\in Z}$.
For $\alpha,\beta\in\T_X$, let
\[
X_0=X_0(\alpha,\beta)=\set{ x\in X}{ x\alpha\neq x\beta},
\]
and then define the \emph{difference rank}
\[
\drank(\alpha,\beta)={\max}\big(|X_0\alpha|,|X_0\beta|\big).
\]
Let us immediately record the following relationship between $\drank(\alpha,\beta)$ and $|\alpha\sd\beta|$:

\begin{lemma}
\label{lem:drank}
For arbitrary $\alpha,\beta\in\T_X$ we have
\[
\drank(\alpha,\beta)\leq |\alpha\sd\beta|\leq 4\drank(\alpha,\beta).
\]
In particular, if either $\drank(\al,\be)$ or $|\al\sd\be|$ is zero or infinite, then ${\drank(\al,\be)=|\al\sd\be|}$.
\end{lemma}

\pf
It suffices to prove the claimed inequalities.
First note that for any $b\in X_0\al$, the block of $\al$ containing $b'$ does not belong to $\be$; it quickly follows that $|X_0\al|\leq|\al\sm\be|$.  Combining this with its dual yields
\[
\drank(\al,\be) = \max\big( |X_0\al|,|X_0\be|\big) \leq |X_0\al|+|X_0\be| \leq |\al\sm\be|+|\be\sm\al| = |\al\sd\be|,
\]
establishing the first inequality.  To prove the second, let
\[
Y = \set{b\in X}{\text{the block of $\al$ containing $b'$ does not belong to $\be$}},
\]
noting that $|Y|=|\al\sm\be|$.  Let $b\in Y$, and let the blocks of $\al$ and $\be$ containing $b'$ be $A\cup\{b'\}$ and $C\cup\{b'\}$, respectively.  Since $A\not=C$, we have either $A\not\sub C$ or $C\not\sub A$.  In the former case, $\emptyset\not=A\sm C\sub X_0$ and so $b\in X_0\al$; similarly, in the latter case we have $b\in X_0\be$.  This shows that $Y\sub X_0\al\cup X_0\be$, and so $|\al\sm\be|=|Y|\leq|X_0\al|+|X_0\be|\leq2\drank(\al,\be)$.  Adding this to its dual completes the proof.
\epf

For $\xi\in\{1\}\cup [\aleph_0,|X|^+]$ Clifford and Preston define a relation
\[
\Delta_\xi=\bigset{ (\alpha,\beta)\in\T_X\times\T_X}{ \drank(\alpha,\beta)<\xi},
\]
in terms of which \cite[Theorem 10.72]{CPbook2} asserts that, in addition to the congruences of the form~$R_N$, the remaining non-universal congruences on infinite $\T_X$ all have the form
\begin{equation}\label{eq:TX_CT2}
R_{\eta_1}\cup (\Delta_{\xi_1}\cap R_{\eta_2})\cup\dots\cup
(\Delta_{\xi_{k-1}}\cap R_{\eta_k})\cup \Delta_{\xi_k},
\end{equation}
where $k\geq1$, $\xi_k<\xi_{k-1}<\dots<\xi_1\leq\eta_1<\eta_2<\dots<\eta_k\leq |X|$,
and all $\xi_i,\eta_i$ are infinite with the possible exception of $\xi_k=1$.
By Lemma \ref{lem:drank}, it immediately follows that $\De_\xi=\mu_\xi$ for any $\xi\in\{1\}\cup [\aleph_0,|X|^+]$.
In particular, $\De_\xi$ is a congruence on $\T_X$ for any such $\xi$ (cf.~Lemma \ref{la315}), a fact that is not proved explicitly by Mal'cev, and which Semla and Sullivan note is ``not entirely obvious''; see the first footnote on p240 of their translation of \cite{Malcev1952}.  In order to avoid confusion with diagonal relations, we will continue to denote this relation by $\mu_\xi$.

In light of the above discussion (and renaming $\eta_1,\eta_2,\ldots,\eta_k$ as $\eta,\eta_1,\ldots,\eta_{k-1}$ in \eqref{eq:TX_CT2}, and defining $\eta_k=|X|^+$), it follows that \cite[Theorems 10.68, 10.72]{CPbook2} can be re-stated as follows:

\begin{thm}
\label{thm:CongTX}
The distinct congruences on the full transformation monoid $\T_X$, for $X$ infinite,
are precisely the universal congruence $\nabla_{\T_X}$, and the following:
\ben
\item
\label{it:TX1}
$R_N$, where $N\unlhd\S_n$, $n\in[1,\aleph_0)$,
\item
\label{it:TX2}
$R_\eta\cup\mu_{\xi_1}^{\eta_1}\cup\dots\cup\mu_{\xi_k}^{\eta_k}$, where
\bit
\item 
$k\geq1$, $\eta\in [\aleph_0,|X|]$, $\eta_1,\ldots,\eta_k\in [\eta,|X|^+]$, $\xi_1,\ldots,\xi_k\in \{1\}\cup[\aleph_0,\eta]$, and 
\item 
$\xi_k<\dots<\xi_1\leq \eta<\eta_1<\dots<\eta_{k-1}<\eta_k=|X|^+$.
\eit
\een
\end{thm}

Theorem \ref{thm:comparisons} can easily be adapted to characterise the inclusion order on congruences of $\T_X$; we omit the details.
As another comparison between Theorems \ref{thm-main} and \ref{thm:CongTX},
one can observe the following relationships between the lattices $\Cong(\P_X)$ and $\Cong(\T_X)$.
\ben
\item\label{it:TX_PX1}
$\Cong(\T_X)$ embeds as a sublattice into $\Cong(\P_X)$, where the embedding maps any congruence of $\T_X$ listed in Theorem \ref{thm:CongTX} to the congruence of $\P_X$ with the same description in Theorem~\ref{thm-main} (noting that $R_N=\lam_{|X|^+}^N\cap\rho_{|X|^+}^N$ and $R_\eta=\lam_{|X|^+}^\eta\cap\rho_{|X|^+}^\eta$ in $\P_X$ for suitable $N$ and $\eta$).
\item\label{it:TX_PX2}
$\Cong(\T_X)$ is a quotient of $\Cong(\P_X)$.
An epimorphism $\Cong(\P_X)\rightarrow\Cong(\T_X)$
is given by mapping any congruence of $\P_X$ listed in Theorem \ref{thm-main} to the congruence of $\T_X$ with the same parameters (noting that $\lam_{\ze_1}^N\cap\rho_{\ze_2}^N=R_N$ in $\T_X$, etc., as observed above).
The kernel classes of this epimorphism are precisely the layers of $\Cong(\P_X)$ as defined in Section~\ref{sect:Hasse}.
\item\label{it:TX_PX3}
It follows from \ref{it:TX_PX1} (or \ref{it:TX_PX2}) and Theorem \ref{thm:wqo} that $\Cong(\T_X)$ is well quasi-ordered; we are not aware of any previous proof of this fact.
\item\label{it:TX_PX4}
It follows from \ref{it:TX_PX1} (or \ref{it:TX_PX2}) and Theorem \ref{thm:distr} that $\Cong(\T_X)$ is distributive.  This was already observed by Clifford and Preston in \cite[Theorem 10.77]{CPbook2}, as a consequence of the meet and join operations on $\Cong(\T_X)$ being precisely intersection and union.  The latter is not the case in $\Cong(\P_X)$; for example, the union of the congruences $\lam_1^1$ and $\rho_1^1$ is not a congruence.
\item \label{it:TX_PX5}
The congruences on infinite $\T_X$ listed in Theorem \ref{thm:CongTX} \ref{it:TX1} form a chain isomorphic to ${(\NN,{\pre})}$, as defined in Subsection \ref{subsect:mj}; thus, the structure of this part of the lattice $\Cong(\T_X)$ is independent of $|X|$, in contrast to the situation for $\P_X$ and $\PB_X$.  The remaining congruences (including $\nabla_{\T_X}$) form a lattice isomorphic to $(\cR,{\pre})$, as defined in Section~\ref{sect:reversals}.  Figure~\ref{fig:CongTX} shows the lattice $\Cong(\T_X)$ in the case that $|X|=\aleph_2$; cf.~Figures \ref{fig:Psi} and \ref{fig:CongPX}.
\een

\begin{figure}[ht]
\begin{center}
\begin{tikzpicture}[scale=0.7,rotate=270]
\foreach \x in {0,...,7} {\fill (0,\x)circle(.12);}
\draw(0,2)--(0,7.5);
\draw[dotted] (0,7.5)--(0,9);
\foreach \x/\y in {0/10,0/11} {\fill (\x,\y)circle(.12);}
\foreach \x/\y in {0/0,-1/1,1/1,-2/2,0/2,-1/3,1/3,0/4,2/4,-1/5,1/5,0/6} {\fill (0+\x*0.707,11+\y*0.707)circle(.12);}
\foreach \x/\y in {0/16.243,0/17.243} {\fill (\x,\y)circle(.12);}
\foreach \x/\y in {0/0,0/1,0/2,0/3,0/4,0/5,0/6,0/10,0/15.243,0/16.243} {\draw[->-=0.56] (\x,\y)--(\x,\y+1);}
\draw[->] (0,7)--(0,7.56);
\foreach \x/\y in {0/0,-1/1,1/1,0/2,1/3,0/4,2/4,1/5} {\draw[->-=0.56] (0+\x*0.707,11+\y*0.707)--(0+\x*0.707-0.707,11+\y*0.707+0.707);}
\foreach \x/\y in {0/0,-1/1,-2/2,0/2,-1/3,1/3,0/4,-1/5} {\draw[->-=0.56] (0+\x*0.707,11+\y*0.707)--(0+\x*0.707+0.707,11+\y*0.707+0.707);}
\end{tikzpicture}
\caption{Hasse diagram of $\Cong(\T_X)\cong\Cong(\I_X)$ when $|X|=\aleph_2$.  The inclusion relation is directed left-to-right.}
\label{fig:CongTX}
\end{center}
\end{figure}
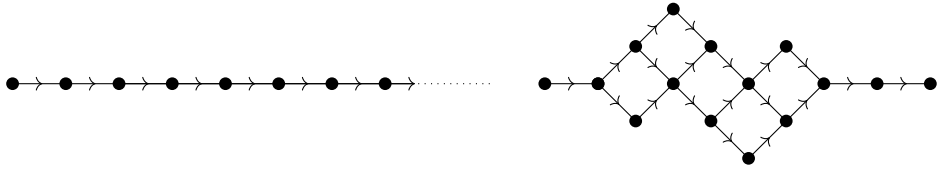

\subsection{Other monoids}\label{subsect:OTM}

Another significant monoid of transformations is the \emph{symmetric inverse monoid} $\I_X$. It consists of all \emph{partial bijections} on $X$, which for us will be partitions from $\P_X$ (indeed, from $\PB_X$) with two-element transversals and singleton non-transversals.
Equivalently, $\I_X$ consists of all partitions $\al$ satisfying $\ker(\alpha)=\coker(\alpha)=\Delta_X$; cf.~\cite[Section~2]{EF2012}.
When $X=\{1,\dots,n\}$ we write $\I_n$ for $\I_X$. 

Congruences on $\I_X$ were characterised by Liber \cite{Liber1953} using a similar approach to Mal'cev \cite{Malcev1952}; see also \cite{Scheiblich1973} for a proof using specialised techniques for inverse semigroups, and \cite[Section~6.3]{GMbook} for a recent treatment in the finite case.  Note that the involution in $\P_X$ restricts to the ordinary inversion operation in $\I_X$; since any semigroup congruence on an inverse semigroup is automatically compatible with inversion, any congruence on $\I_X$ is a $*$-congruence.  In our terminology, the results are as follows:

\begin{thm}
\label{thm:CongIn}
The distinct congruences of the symmetric inverse monoid $\I_n$ for $n\in \N$, $n\geq 2$, are precisely
the universal congruence $\nabla_{\I_n}$, and the congruences
$R_N$ for $N\unlhd \S_q$ $(1\leq q\leq n)$. The lattice $\Cong(\I_n)$ is a chain isomorphic to $\Cong(\T_n)$.
\end{thm}

\begin{thm}
\label{thm:CongIX}
The distinct congruences on the symmetric inverse monoid $\I_X$, for $X$ infinite, are precisely the universal congruence $\nabla_{\I_X}$, and the following:
\ben
\item
$R_N$, where $N\unlhd\S_n$, $n\in[1,\aleph_0)$,
\item
$R_\eta \cup \mu_{\xi_1}^{\eta_1}\cup \cdots \cup \mu_{\xi_k}^{\eta_k}$,
 where 
\bit
\item 
$k\geq1$, $\eta\in [\aleph_0,|X|]$, $\eta_1,\ldots,\eta_k\in [\eta,|X|^+]$, $\xi_1,\ldots,\xi_k\in \{1\}\cup[\aleph_0,\eta]$, and 
\item 
$\xi_k<\dots<\xi_1\leq \eta<\eta_1<\dots<\eta_{k-1}<\eta_k=|X|^+$.
\eit
\een
The lattice $\Cong(\I_X)$ is isomorphic to $\Cong(\T_X)$.
\end{thm}

Although at least two proofs of the latter result already exist \cite{Liber1953,Scheiblich1973}, we note that the method in the current paper yields yet another:

\begin{proof}[{\bf Sketch of proof}]
That the listed relations are congruences follows from the fact that they are restrictions of their counterparts from $\P_X$.
To prove that every congruence of $\I_X$ is among those listed, one needs to go through the argument presented in Sections \ref{sect:stage2} and \ref{sect:tech} and check that it holds for $\I_X$. This is accomplished by checking that all the partitions constructed in the course of the proof are in fact partial bijections, provided that the given partitions are partial bijections to begin with.
In fact, many of the arguments simplify radically during this process,
for instance due to the fact that $\wh{\alpha}=\epsilon_{\emptyset}$ for all
$\alpha\in \I_X$.
\end{proof}

Using duality in category theory, FitzGerald and Leech introduced the \emph{dual symmetric inverse monoid} $\JJ_X$ in \cite{FL1998}; this monoid consists of all \emph{block bijections} on $X$: i.e., all bijections between quotient sets of $X$.  As in \cite{Maltcev2007} and \cite[Section 2]{EF2012}, $\JJ_X$ may be identified with the submonoid of~$\P_X$ consisting of all partitions $\al$ with ${\dom(\al)=\codom(\al)=X}$: i.e., all partitions with no non-transversals.  Congruences on finite $\JJ_X$ were classified in \cite{KM2011}, and the statement is analogous to Theorem \ref{thm:CongIn}.  It would be interesting to apply the methods in the current paper to the infinite case (which, as far as the authors are aware, has not previously been considered); note that the~$\lam/\rho$ relations would play no role here, for the same reason as in $\T_X$.  

The monoid $\F_X$ of all \emph{uniform} block bijections \cite{FL1998,Fit2003} also seems very worthy of attention; a block bijection $\binom{A_i}{B_i}$ from $\JJ_X$ is uniform if $|A_i|=|B_i|$ for all $i$.  The monoid $\F_X$ may also be characterised as the submonoid of~$\JJ_X$ generated by all idempotents and units \cite[Proposition~3.1]{FL1998}.  
While the monoid $\F_X$ has many similarities with $\JJ_X$ (and $\I_X$, $\P_X$, etc.), it has a far more complicated ideal structure; indeed, while the ideals of $\JJ_X$ form a chain, this not true in $\F_X$.  Even in the finite case, the poset of \emph{principal} ideals of $\F_n$ is isomorphic to the poset of all integer partitions of~$n$ under the reverse refinement order; cf.~\cite[Section 3]{FL1998}.
This poset is shown in Figure~\ref{fig:F5} for $n=4,5$.  Furthermore, maximal subgroups of~$\F_X$ are direct products of symmetric groups of various degrees, rather than simply being individual symmetric groups.  All of this leads to a very non-linear lattice structure.  However, computational evidence suggests the situation might be amenable to the kind of analysis carried out in this paper and in \cite{EMRT2018}.  Figure \ref{fig:F5} gives Hasse diagrams of $\Cong(\F_4)$ and $\Cong(\F_5)$, calculated using GAP \cite{GAP}. 

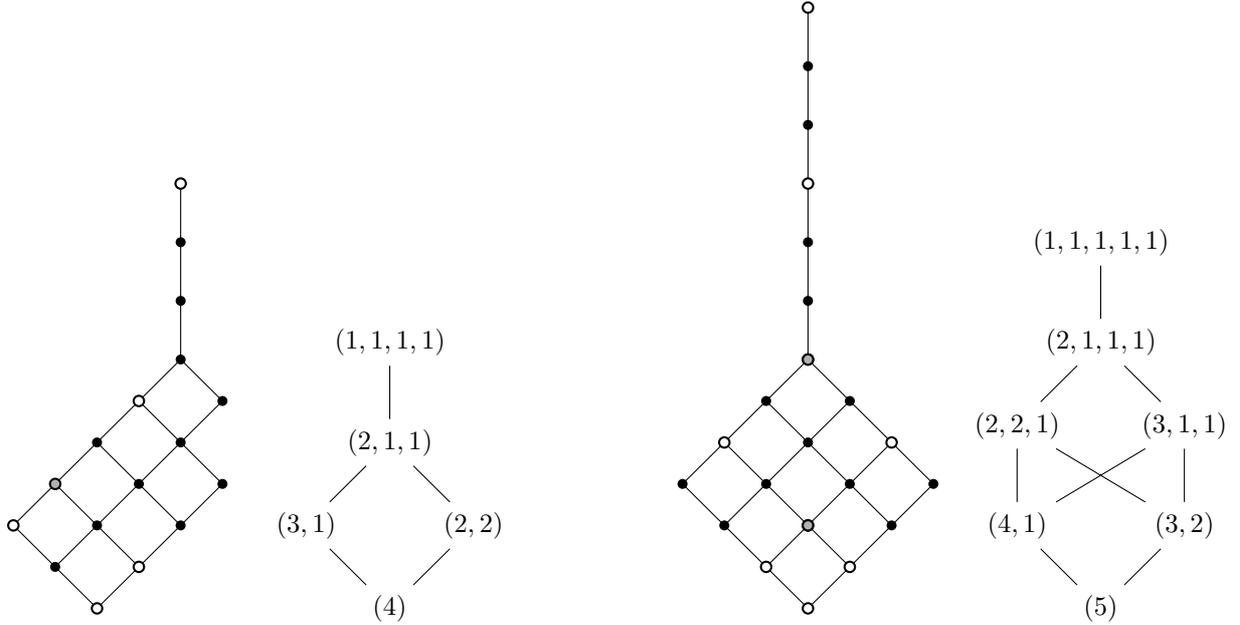
\begin{figure}
\begin{center}
\begin{tikzpicture}[scale=0.55]
\begin{scope}
\draw (2,10.242)--(2,6)--(-2,2)--(0,0)--(3,3)--(1,5) (2,6)--(3,5)--(-1,1) (2,2)--(0,4) (1,1)--(-1,3);
\foreach \x/\y in {-1/1,0/2,2/2,1/3,3/3,0/4,2/4,3/5,2/6,2/7.414,2/8.828} {\fill (\x,\y)circle(.12);}
\foreach \x/\y in {0/0,1/1,-2/2,1/5,2/10.242} {\fill (\x,\y)circle(.15); \fill[white] (\x,\y)circle(.1);}
\foreach \x/\y in {-1/3} {\fill (\x,\y)circle(.15); \fill[black!30] (\x,\y)circle(.1);}
\end{scope}
\begin{scope}[shift={(7,0)}]
\node (A) at (0,6.414) {\small $(1,1,1,1)$};
\node (B) at (0,4) {\small $(2,1,1)$};
\node (C) at (-2,2) {\small $(3,1)$};
\node (D) at (2,2) {\small $(2,2)$};
\node (E) at (0,0) {\small $(4)$};
\draw (A)--(B)--(C)--(E)--(D)--(B) ;
\end{scope}
\begin{scope}[shift={(17,0)}]
\draw (0,14.485)--(0,6)--(3,3)--(0,0)--(-3,3)--(0,6) (-2,4)--(1,1) (-1,5)--(2,2) (-2,2)--(1,5) (-1,1)--(2,4);
\foreach \x/\y in {-2/2,2/2,-3/3,-1/3,1/3,3/3,0/4,-1/5,1/5,0/7.414,0/8.828,0/11.657,0/13.071} {\fill (\x,\y)circle(.12);}
\foreach \x/\y in {0/14.485,0/10.242,0/6,-2/4,2/4,0/2,-1/1,1/1,0/0} {\fill (\x,\y)circle(.15); \fill[white] (\x,\y)circle(.1);}
\foreach \x/\y in {0/6,0/2} {\fill (\x,\y)circle(.15); \fill[black!30] (\x,\y)circle(.1);}
\end{scope}
\begin{scope}[shift={(24,0)}]
\node (A) at (0,8+2*.414) {\small $(1,1,1,1,1)$};
\node (B) at (0,6+.414) {\small $(2,1,1,1)$};
\node (C) at (2,4+.414) {\small $(3,1,1)$};
\node (D) at (-2,4+.414) {\small $(2,2,1)$};
\node (E) at (-2,2) {\small $(4,1)$};
\node (F) at (2,2) {\small $(3,2)$};
\node (G) at (0,0) {\small $(5)$};
\draw (A)--(B)--(C)--(E)--(G)--(F)--(D)--(B) (C)--(F) (D)--(E);
\end{scope}
\end{tikzpicture}
\end{center}
\vspace{-5mm}
\caption{Hasse diagrams of the lattice $\Cong(\F_n)$ and the poset of integer partitions of $n$, for $n=4$ (left) and $n=5$ (right).  Rees congruences are indicated by white or gray vertices for principal and non-principal ideals, respectively. }
\label{fig:F5}
\end{figure}

As noted in \cite[p6]{JEgrpm} and \cite[p277]{FL2011}, the \emph{partial transformation monoid} $\P\T_X$ (which consists of all partial transformations of $X$) does not canonically embed in $\P_X$ in the way that $\T_X$ and~$\I_X$ do.  Nevertheless, the methods of the current paper could certainly be adapted to recover the known description of $\Cong(\P\T_X)$ given by \v{S}utov \cite{Sutov1961}.  

Congruences on several other families of monoids could potentially be explored using the methods developed here and in \cite{EMRT2018}: examples include (finite and infinite) twisted diagram monoids \cite{DE4,BDP2002,ERtwisted1,ERtwisted2}, rook partition monoids \cite{Grood2006}, monoids of partitioned binary relations \cite{MM2013} and the submonoids of $\P_X$ and $\PB_X$ generated by all idempotents and units \cite{JEipms,EF2012,JE_IBM}.  The latter submonoids of $\P_X$ and $\PB_X$ are analogous to the submonoid $\F_X$ of $\JJ_X$ discussed above, and the elements of these submonoids may be characterised in terms of a property similar to uniformity of block bijections; see \cite[Theorem 6.1]{JE_IBM} and \cite[Theorem 33]{EF2012}.  Congruences on the corresponding submonoid of $\T_X$ were classified in \cite{MSS2000}.

There are also infinite analogues of the planar partition monoid $\PP_n$ and the Motzkin monoid~$\mathscr M_n$ considered in \cite[Section 7]{EMRT2018}.  Suppose $(X,\leq)$ is a totally ordered set.  We first extend~$\leq$ to a total order on $X\cup X'$ by further defining $x\leq y'$ for all $x,y\in X$, and ${x'\leq y'\iff y\leq x}$.  (For example, if $X=\{1,\ldots,n\}$ with the usual order, then we have $1<\cdots<n<n'<\cdots<1'$.)  We say a partition $\al\in\P_X$ is \emph{planar} if we never have $a<x<b<y$ where $a,b\in A$ and $x,y\in B$ for distinct blocks $A$ and $B$ of $\al$.  It follows from \cite[Lemma 7.1]{EMRT2018} that when $X=\{1,\ldots,n\}$ this is equivalent to there being a graphical representation of $\al$ where the edges are drawn within the rectangle spanned by the vertices and do not intersect; thus, in Figure \ref{fig:P6} for example, $\be$ is planar but~$\al$ is not.  One may show that the set~$\PP_{(X,\leq)}$ of all planar partitions is a submonoid of $\P_X$.  Note that the structure of $\PP_{(X,\leq)}$ depends crucially on the ordering on $X$, and not just its size; for example, taking $X$ to be $\N=\{0,1,2,\ldots\}$ or $\ZZ=\{\ldots,-2,-1,0,1,2,\ldots\}$ under the usual orderings, all subgroups of $\PP_{(\N,\leq)}$ are trivial, while $\PP_{(\ZZ,\leq)}$ contains infinite cyclic groups.  One may also define an infinite Motzkin monoid~$\mathscr M_{(X,\leq)}=\PP_{(X,\leq)}\cap\PB_X$; all of the partitions in Figure \ref{fig:PBX} belong to $\mathscr M_{(\N,\leq)}$.  It would be interesting to study these monoids, even in fairly ``controlled'' cases, for example when $(X,\leq)$ is well-ordered, or when $X$ is some subset of the reals or rationals under the usual order.

\subsection{Ideals}\label{subsect:ideals}

It would also be interesting to study congruences on the ideals of $\P_X$ or $\PB_X$ (or any of the other monoids discussed above), considered as semigroups in their own right.  Indeed, the current paper and \cite{EMRT2018} can be viewed as treating the ideal $I_{|X|^+}$, while the congruences on the ideal $I_1$ are easily described since $I_1=D_0$ is a rectangular band.  It is not clear whether congruences on~$I_\xi$ for $2\leq\xi\leq|X|$ will all be restrictions of congruences on $\P_X$ or $\PB_X$, or whether extra congruences can arise.

The corresponding question for congruences on ideals of full transformation semigroups was answered in 1977 by Klimov \cite{Klimov1977}.  The minimal ideal of a full transformation semigroup $\T_X$ is a right-zero semigroup of size $|X|$, so every equivalence on this ideal is a congruence.
As an application of the theory developed in \cite{Klimov1977}, it was shown that every congruence on a non-minimal ideal~$I$ of $\T_X$ is the restriction of a congruence on $\T_X$; in particular, the congruence lattice of such an ideal is isomorphic to the interval $[\De_{\T_X},R_I]$ in $\Cong(\T_X)$.  The key ingredients in the proof of this result are:
\bit
\item[(1)] Mal'cev's description of the congruences of~$\T_X$ (stated in Theorems~\ref{thm:CongTn} and~\ref{thm:CongTX} above); 
\item[(2)] the fact that every non-minimal ideal $I$ of $\T_X$ is \emph{fully reductive}, meaning that for every congruence $\si$ on $I$, and for every $\al,\be\in I$, the following implication holds:
\[
\bigset{(\ga\al\de,\ga\be\de)}{\ga,\de\in I} \sub \si \implies (\al,\be)\in\si;
\]
\item[(3)] the fact that any congruence on a fully reductive semigroup $S$ is liftable to any ideal extension of $S$.
\eit
%
%
%
%
Our main result (Theorem \ref{thm-main}) describes the congruences on infinite~$\M$, which as usual denotes either the partition monoid $\P_X$ or the partial Brauer monoid $\PB_X$.  One might then hope to deduce a description of the congruences on an arbitrary non-minimal ideal~$I_\xi$ of $\M$ by following Klimov's approach: i.e., by showing that such an ideal is fully reductive.  Intriguingly, however, it turns out that no proper ideal of infinite $\M$ is fully reductive:

\begin{prop}
If $X$ is infinite, then the only fully reductive ideal of $\M$ is $\M$ itself.
\end{prop}

\pf
Since $\M$ is a monoid, it is fully reductive.
Conversely, consider some proper ideal $I_\xi$ of $\M$, where $1\leq\xi\leq|X|$.  Let $\ze=\max(\aleph_0,\xi)$, noting that ${\aleph_0\leq\ze\leq|X|}$ and $\xi\leq\ze$.  The relation $\lam_\ze=\lam_\ze^{|X|^+}$ is a congruence on $\M$ (cf.~Lemma~\ref{la33}), so the restriction $\si=\lam_\ze\restr_{I_\xi}$ is a congruence on $I_\xi$.  
We prove the proposition by showing that there exist $\al,\be\in I_\xi$ such that
\[
\bigset{(\ga\al\de,\ga\be\de)}{\ga,\de\in I_\xi} \sub \si \qquad\text{but}\qquad (\al,\be)\not\in\si.
\]
To do so, consider any $\al,\be\in D_0$ with $|\ol\al\sd\ol\be|\geq\ze$.  Then $(\al,\be)\not\in\lam_\ze$, and so $(\al,\be)\not\in\si$.  
Now let $\ga,\de\in I_\xi$ be arbitrary.  We must show that $(\ga\al\de,\ga\be\de)\in\si$: i.e., that $|\ol{\ga\al\de}\sd\ol{\ga\be\de}|<\ze$.  Write $\ga=\partABCD$, noting that $|I|=\rank(\ga)<\xi$.  Each $C_j$ ($j\in J$) is a non-transversal of both $\ga\al\de$ and $\ga\be\de$, and so belongs to both $\ol{\ga\al\de}$ and $\ol{\ga\be\de}$.  Every other block of $\ol{\ga\al\de}$ and $\ol{\ga\be\de}$ is a union of the $A_i$.  It follows that $|\ol{\ga\al\de}\sd\ol{\ga\be\de}|\leq2|I|<2\xi\leq\ze$, completing the proof that $(\ga\al\de,\ga\be\de)\in\si$.
\epf

Thus, to describe the congruences of the ideals of diagram monoids, new techniques are required, and this is the subject of a recent work by the authors \cite{ER2020}.

\subsection*{Acknowledgements}

The first author is supported by ARC Future Fellowship FT190100632.
The second author is supported by EPSRC grant EP/S020616/1.
We thank Mikhail Volkov and Mark Sapir for useful discussions, and for drawing our attention to Klimov's paper \cite{Klimov1977}.
We also thank the referee for their careful reading of the paper, and for their valuable suggestions, especially for pointing out the second clause in Corollary \ref{cor:iso}.

\footnotesize
\def\bibspacing{-1.1pt}
\bibliography{biblio}
\bibliographystyle{abbrv}

\end{document}